\documentclass{article}
\usepackage{preamble_modified}
\usepackage{graphicx}
\usepackage{authblk}

\usepackage{amsmath,amssymb,amsfonts}

\usepackage[numbers]{natbib}
 \bibpunct[, ]{(}{)}{,}{a}{}{,}%

\usepackage{rotating}
\usepackage{fancyvrb}

\providecommand{\keywords}[1]
{
  \small	
  \textbf{Keywords.} #1
}

\begin{document}

\title{Prediction-Correction Algorithm for \\ Time-Varying Smooth Non-Convex Optimization}

\author[1]{Hidenori Iwakiri \footnote{rtuocck2.020.hi@gmail.com}}
\author[1]{Tomoya Kamijima \footnote{kamijima-tomoya101@g.ecc.u-tokyo.ac.jp}}
\author[2,3]{Shinji Ito   \footnote{i-shinji@nec.com}}
\author[1,3]{Akiko Takeda \footnote{takeda@mist.i.u-tokyo.ac.jp}}

\affil[1]{Department of Mathematical Informatics, The University of Tokyo, Tokyo, Japan}
\affil[2]{NEC Corporation, Kanagawa, Japan}
\affil[3]{Center for Advanced Intelligence Project, RIKEN, Tokyo, Japan}

\maketitle

\begin{abstract}
Time-varying optimization problems are prevalent in various engineering fields, and the ability to solve them accurately in real-time is becoming increasingly important. 
The prediction-correction algorithms used in smooth time-varying optimization can achieve better accuracy than that of the time-varying gradient descent (TVGD) algorithm. 
However, none of the existing prediction-correction algorithms can be applied to general non-strongly-convex functions, and most of them are not computationally efficient enough to solve large-scale problems. 
Here, we propose a new prediction-correction algorithm that is applicable to large-scale and general non-convex problems and that is more accurate than TVGD. 
Furthermore, we present convergence analyses of the TVGD and proposed prediction-correction algorithms for non-strongly-convex functions for the first time. 
In numerical experiments using synthetic and real datasets, the proposed algorithm is shown to be able to reduce the convergence error as the theoretical analyses suggest and outperform the existing algorithms.
\end{abstract}


\section{Introduction}
Time-varying optimization problem appears in various engineering fields, such as robotics~\citep{ardeshiri2011convex}, control~\citep{hours2014parametric}, signal processing~\citep{jakubiec2012d}, electronics~\citep{dall2016optimal}, and machine learning~\citep{simonetto2020survey}. For example, in robot control systems, we need to control the movements of agents to achieve the desired results in a time-varying environment.
In recommendation systems, the best suggestion would be non-stationary when the user's preference or item's value is time-varying, or the revealed rating information is updated with time. Furthermore, recent developments in computers have heightened the need for real-time optimization methods. In real-time optimization, finding an exact optimal solution every time is generally impossible due to the short computational time; thus, we cannot avoid solving the optimization problem approximately at every iteration. 

In this paper, we consider the following time-varying optimization problem:
\begin{align} \label{eq:time_varying_problem}
  \underset{x \in \mathbb{R}^d}{\mathrm{minimize}}\ f(x;t_k).
\end{align}
where $t_k := kh$ denotes the time step, and $h > 0$ is the sampling period.
The value of $h$ should be predetermined before solving the problem. It should be as small as possible, considering factors such as the physical constraints (e.g., the robot's response speed) and the algorithm's computational time.
Although some studies have developed algorithms in a continuous-time setting (i.e., in the limit of $h\to +0$)
~\citep{fazlyab2018interior, sun2022continuous},
  the practical implementation requires a non-zero $h$, which may lead to instability~\citep{xie2022powerflow}.
In contrast, algorithms developed in the discrete-time setting 
are stable even when the sampling period is not short. 
Previous studies in this setting have focused on the dependence of the algorithm's performance on $h$, as will be explained later.

The two most important issues in this problem are to find the optimal trajectory $\{ x_k^\ast := {\mathrm{argmin}}_{x \in \mathbb{R}^d}\ f(x;t_k) \}$ as fast as possible and to track it as accurately as possible, within a limited amount of computational time per iteration. In this paper, we focus on the latter point:
we aim at improving the tracking accuracy after a sufficiently long time has elapsed.
We assume that the objective function $f(x;t)$ is smooth in terms of $x$ and $t$. In particular, we suppose that the shape of the objective function does not change abruptly with time.

A prediction-correction algorithm has been proposed for such an optimization problem. 
It consists of two steps (see Algorithm~\ref{algo:pred_corr_template}): 
in the prediction step, the algorithm approximates the function value at the next time step, which is unknown, and predicts a point with a good property based on the approximation;
in the correction step, it corrects the prediction based on the function value that is revealed after the prediction.
We can use the gradient descent (GD)-based method for time-invariant problems in the correction step. We refer to an algorithm that uses the GD algorithm in the correction step and does not involve the prediction step as a time-varying GD (TVGD) algorithm.
The prediction step requires a design specific to time-varying optimization, and several prediction algorithms have been proposed to reduce the asymptotical tracking error from the $O(h)$ level attained by
TVGD to $O(h^2)$ or better~\citep{simonetto2016aclassof, simonetto2017constrained,lin2019simplified, bastianello2020primal}.

\begin{figure}[h]
\begin{algorithm}[H]
\caption{Prediction-Correction Algorithm Template} 
\begin{algorithmic}[1]\label{algo:pred_corr_template}
  \REQUIRE Initial solution $x_{0\mid -1}$
\FOR{$k=0,1,2,\ldots$}
    \STATE // time $t_k$ ($:= kh$)
    \STATE Incur the loss $f(x_{k \mid k-1};t_{k})$ 
    \STATE Acquire the objective function $f(\cdot;t_{k})$
    \STATE Correct the prediction $x_{k \mid k-1}$ to $x_k$ based on $f(\cdot;t_{k})$
    \STATE Predict $x_{k+1 \mid k}$ based on approximation of $f(\cdot;t_{k+1})$
\ENDFOR
\end{algorithmic}
\end{algorithm}
\end{figure}


However, all of the existing prediction-correction algorithms are only applicable to strongly-convex (SC) functions or the sum of SC functions and convex functions with nice properties. Moreover, most of them require the Hessian or its inverse to be calculated in the prediction step; thus, they take a long computational time per iteration when the feasible region has high dimensions. In real-time optimization including time-varying optimization, decreasing the computational cost is especially crucial: the long computational time for each iteration leads to a long sampling period, which in turn leads to low tracking accuracy.

\subsection{Contributions}
This paper proposes new prediction-correction algorithms: the First-Order Approximation Minimization (FOA-Min) and the Cauchy Point algorithm (CP).
As far as we know, FOR-Min is the first proposal of a prediction-correction algorithm that is applicable to general non-convex functions and is guaranteed to achieve higher accuracy than that of the TVGD algorithm.
Table~\ref{table:algo_comparison} summarizes the accuracies and required oracles of the existing and proposed algorithms.
Although the previous studies used the tracking error $\| x_k - x_k^\ast \|$ to evaluate the algorithm's performance, here we evaluate it by using the function value or gradient norm to cover non-SC functions.
FOA-Min
is found to be more accurate
than TVGD and the existing prediction-correction algorithms for Polyak-Łojasiewicz (PL) functions in terms of the function value and for non-convex functions in terms of the gradient norm. Moreover,
it uses only the gradient and its norm in the prediction step; thus, the computational time per iteration is comparable to that of TVGD and is much shorter than that of U-FOPC~\citep{simonetto2017constrained} and AGT~\citep{simonetto2016aclassof}, which need to calculate the Hessian and its inverse, respectively.
The proposed Cauchy Point (CP) algorithm optimizes the approximation of the objective function at the next time step more precisely than does FOA-Min using the Hessian.

\begin{table*}[h] 
\centering
\caption{Comparison of accuracies and required oracles of existing and proposed algorithms. U-FOPC has no theoretical convergence guarantees for non-SC functions, and AGT cannot be applied to them.}

\begin{tabular}{c|ccc|cc}
    \toprule
    \multicolumn{1}{c}{} & \multicolumn{3}{|c|}{Existing} &  \multicolumn{2}{c}{Proposed}
    \\ \hline
     Algorithm & TVGD  & U-FOPC & AGT  & FOA-Min & CP
     \\ \cline{2-6} 
     \begin{tabular}{c} Required oracles \end{tabular} & \begin{tabular}{c} Grad \end{tabular} & \begin{tabular}{c} Grad, \\ Hessian \end{tabular} & \begin{tabular}{c} Grad, \\ Hessian inv. \end{tabular} & \begin{tabular}{c} Grad, \\ Grad norm \end{tabular} & \begin{tabular}{c} Grad, \\ Hessian \end{tabular}
     \\ \hline \hline
    \begin{tabular}{c} PL \ (-optimal) \end{tabular} & $O(h)$ & - & N/A & $O(h^2)$ & $O(h^2)$
    \\ 
    \begin{tabular}{c} Non-convex \ (-stationary) \end{tabular} & $O(\sqrt{h})$ & - & N/A & $O(h)$ & $O(h)$
    \\ \bottomrule
    \end{tabular}%

\label{table:algo_comparison}
\end{table*}

We also present convergence analyses of TVGD and the prediction-correction algorithms for non-SC optimization for the first time. 
We find that they converge linearly to the error bounds shown in Table~\ref{table:algo_comparison} for PL functions.
We also prove that these algorithms can find approximate stationary solutions of non-convex functions at some time step, and once the algorithms converge, all the subsequent iterates satisfy a desirable property in the sense of the function value and gradient norm. This implies that we can track a stationary point or can find better points after convergence.

  We need to impose additional assumptions to theoretically guarantee the improvement in accuracy obtained by the proposed methods over TVGD. We place a new assumption on the relationship between derivatives in terms of $x$ and $t$, which holds for several important problems, such as optimization of functions with a parallel shift, linear regression, and SC optimization.
We again exploit the equivalent reformulation of the optimization problem to prove that the SC optimization satisfies the assumption.

The proposed algorithms are compared with the existing ones in three numerical experiments: a non-convex toy example, linear regression using synthetic datasets, and matrix factorization using real datasets. The results demonstrate that the proposed algorithms reduce the convergence error as the theoretical analyses suggest and that FOA-Min has better tracking accuracy compared with TVGD on a practical large-scale non-convex problem.

\subsection{Recent Work on the Prediction-Correction Algorithm}
Various prediction-correction algorithms have been proposed and analyzed for solving time-varying smooth optimization problems including an SC function. \citet{simonetto2016aclassof} proposed a Taylor expansion-based prediction for unconstrained optimization, and \citet{simonetto2017constrained} extended it to constrained optimization without using the inverse of the Hessian. \citet{bastianello2020primal} developed a unified framework for time-varying optimization based on the prediction-correction paradigm and provided an extrapolation-based prediction, which is useful
when the Hessian of the objective function is time-invariant. \citet{lin2019simplified} proposed a simple and efficient prediction method without using function information. It was proved that the method achieves the same asymptotical tracking error $O(h^2)$ as those of other prediction methods.

Prediction-correction algorithms for more specific problems have also been studied recently. \citet{bastianello2020distributed, simonetto2017networked, wang2020allocation} developed algorithms to solve distributed optimization problems where multiple agents communicate and cooperate with each other to minimize an overall cost function that varies over time. \citet{xie2022powerflow} applied multiple online algorithms including a discrete-time prediction-correction algorithm to optimal power flow (OPF), a time-varying optimization problem on power systems. Four different algorithms were compared in a numerical experiment using real data, and it was concluded that a discrete-time prediction-correction algorithm is the best choice in most cases.
Prediction-correction algorithms have also been leveraged for estimating the dynamical state of a power system~\citep{song2020state} and for model predictive control (MPC)~\citep{paternain2018mpc} with theoretical guarantees.

\subsection{Notation}
Let $x^\ast(t)$ denote an optimal solution of the time-varying optimization problem $\mathrm{minimize}_{x\in\mathbb{R}^d} f(x;t)$.
For an integer $k \in \mathbb{Z}_{\geq 0}$, let $x_k^\ast$ and $f_k^\ast$ be an optimal solution and optimal value of $\mathrm{minimize}_{x\in\mathbb{R}^d} f(x;t_k)$, respectively. We will abuse the notation and denote the derivative of $f(x;t)$ in terms of $t \in \mathbb{R}$ by $\nabla_t f(x;t)$. We denote second-order derivatives by $\nabla_{ba} f(x;t):=\nabla_{b}(\nabla_{a} f(x;t))$
for $a,b \in \{x,t\}$.
For a sequence $\{ x_k \}$, let $f_k := f(x_k;t_k)$, $\nabla_{a} f_k := \nabla_{a} f(x_k;t_k)$ and $\nabla_{ba} f_k := \nabla_{ba} f(x_k;t_k)$ for $a, b \in \{ x, t \}$.

\section{Preliminary}
In this section, we introduce the TVGD and existing prediction-correction algorithms for time-varying smooth optimization and outline the theoretical analyses provided by the previous work.
\subsection{TVGD Algorithm}

The TVGD algorithm
is the most naive
method of 
time-varying smooth optimization. At each time step, TVGD updates the solution $x_k$ by using the current gradient, as described in Algorithm~\ref{algo:GD}.

\begin{figure}[h]
    \begin{algorithm}[H]
    \caption{Time-Varying Gradient Descent (TVGD)~\citep{popkov2005gradient}} 
    \begin{algorithmic}[1]\label{algo:GD}
      \REQUIRE Initial solution $x_0$,
      step size $\beta$\\
    
    \FOR{$k=0,1,2,\ldots$}
    \STATE // time $t_k$ ($:= kh$)
    \STATE Incur the loss $f_{k}$
    \STATE Acquire the objective function $f(\cdot;t_{k})$
    \STATE Initialize the sequence of corrected variables as $\hat{x}_k^0 = x_{k}$
    \FOR{$c=0,1,\ldots,C-1$}
        \STATE 
        $$
        \hat{x}_{k}^{c+1} = \hat{x}_k^{c} - \beta \nabla_x f(\hat{x}_k^c;t_{k})
        $$
    \ENDFOR
    \STATE Set the solution as $x_{k+1} = \hat{x}_k^C$
    \ENDFOR
    \end{algorithmic}
    \end{algorithm}
    \vspace{-3mm}
\end{figure}

We can prove that the tracking error $\{\| x_k - x_k^\ast \|\}$ converges linearly to $O(h)$ if the objective function is SC and the following assumptions hold~\citep{popkov2005gradient}:

\begin{assu}\label{assu:prev_sc}
$\ $
    \begin{itemize}
        \item[$\mathrm{(i)}$]
        $\forall x\in\mathbb{R}^d, \forall t\geq 0,\ \| \nabla_{xx} f(x;t) \| \leq L_1$.
        \item[$\mathrm{(ii)}$] $\forall x\in\mathbb{R}^d, \forall t\geq 0,\ \| \nabla_{tx} f(x;t) \| \leq L_2$.
        \item[$\mathrm{(iii)}$] $\forall x\in\mathbb{R}^d, \forall t\geq 0,\ \nabla_{xx} f(x;t) \succeq m\mathrm{I}$.
    \end{itemize}
\end{assu}

Assumption~\ref{assu:prev_sc}(i) and (iii) require that the objective function is twice differentiable, $L_1$-smooth, and $m$-strongly-convex in terms of $x$. Assumption~\ref{assu:prev_sc}(ii) implies that the time variation of the gradient of the objective function can be bounded.

%
%
\subsection{Prediction-Correction Algorithm Based on Taylor Expansion}
Next, let us examine the Taylor expansion-based prediction-correction algorithm proposed in \citep{simonetto2017constrained}.
In the prediction step,
we seek a solution to the following quadratic optimization problem:
\begin{align*} 
    &\min_{x \in \mathbb{R}^d}\ \hat{f}^{2, \gamma}(x;t_{k+1}) := f_k + \gamma \agbra*{\nabla_x f_k, x-x_k} + h\nabla_t f_k\\
    &+ \frac{1}{2}(x-x_k)^\top \nabla_{xx} f_k(x - x_k) + h\agbra*{\nabla_{tx}f_k, x - x_k} + \frac{h^2}{2}\nabla_{tt}f_k,
\end{align*}
where $\gamma$ is a parameter satisfying $0\leq \gamma \leq 1$.
This function becomes a second-order Taylor series approximation of the objective function at the next time step when $\gamma$ is set to $1$.
We update the prediction sequence  
$\{ \hat{x}_k^{p} \}_{p=0, \ldots, P-1}$ by applying the (time-invariant) GD algorithm to the above problem:
\begin{align} \label{eq:existing_inexact_prediction}
   &\hat{x}_k^{p+1} = \hat{x}_k^{p} -\alpha \paren{\nabla_{xx} f_k(\hat{x}_k^{p}-x_k) + h\nabla_{tx} f_k + \gamma\nabla_x f_k},\nonumber \\
   &\mathrm{where}\ \ \hat{x}_k^0 = x_k,\ \hat{x}_k^{P} = x_{k+1\mid k}.
\end{align}

See Algorithm 2 in \citep{simonetto2017constrained} for the complete procedure.

Under the same assumption as the TVGD algorithm, we can prove that the tracking error $\{\| x_k - x_k^\ast \|\}$ generated by this algorithm converges globally and the asymptotical tracking error $\lim_{k \to \infty} \| x_k - x_k^\ast \|$ can be bounded by $O(h)$ (Theorem 3 in \citep{simonetto2017constrained}).
Moreover, we can obtain a better local convergence error $O(h^2)$ when the third derivatives of the objective function are bounded, the number of prediction steps is sufficiently large, and several other conditions are satisfied (Theorem 4 in \citep{simonetto2017constrained}).

\section{Analysis of TVGD Algorithm for Non-Strongly-Convex Functions} \label{chap:analysis_of_gd}

%

%

In this section, we analyze the performance of the TVGD algorithm (Algorithm~\ref{algo:GD}) for time-varying non-SC optimization.
Though the TVGD algorithm is an existing method, it is one of our contributions to apply this algorithm to non-SC settings and to analyze its performance theoretically.
The previous studies used the tracking error $\| x_k - x_k^\ast \|$ to evaluate the algorithm's performance. This evaluation metric is valid since a point whose function value is close to the optimal value is always close to the optimal point for SC functions. However, this relationship does not hold for non-SC functions. Therefore, we decided to evaluate the algorithm's performance on the basis of the function value or gradient norm.

\subsection{Assumption} \label{sec:GD_assumptions}
We place an assumption on the time variation of the function value.
In particular, we want to assume the Lipschitzness of $f(x;t)$ in terms of $t$, while Assumption~\ref{assu:prev_sc}(ii) requires the Lipschitzness of $\nabla_x f(x;t)$ in terms of $t$. However, we do not have to assume the Lipschitzness of $f(x;\cdot)$ itself. In fact, Problem \eqref{eq:time_varying_problem} is equivalent to
\begin{align}\label{eq:time_varying_problem_f-G}
    \underset{x \in \mathbb{R}^d}{\mathrm{minimize}}\ \bar{f}(x;t_k) := [f(x;t_k) - \mathscr{G}(t_k)],\ k\in\mathbb{Z}_{\geq 0},
\end{align}
where $\mathscr{G}:\mathbb{R}\to\mathbb{R}$ is a univariate function. Therefore, we can see that it is sufficient to assume the Lipschitzness of $f(x;\cdot) - \mathscr{G}(\cdot)$:
\begin{assu}\label{assu:lip_t}
    There exists a function $\mathscr{G}:\mathbb{R}\to\mathbb{R}$ satisfying
    $
        \forall x\in\mathbb{R}^d, \forall s, t\geq 0,\ 
        | [f(x;s) - \mathscr{G}(s)] - [f(x;t) - \mathscr{G}(t)] | \leq G_2 |s - t|.
    $
\end{assu}

Under this assumption, we can prove that the optimal value of the redefined objective function $f(x;t) - \mathscr{G}(t)$ is also $G_2$-Lipschitz continuous in terms of $t$:

\begin{lemm}\label{lemm:lip_t_optimal_value}
    For any $x\in\mathbb{R}^d$ and $s, t\geq 0$, we have
\begin{align*}
    | [f(x^\ast(s);s) - \mathscr{G}(s)] - [f(x^\ast(t);t) - \mathscr{G}(t)] | \leq G_2 |s - t|.
\end{align*}
\end{lemm}

\begin{proof}
For any $s, t\geq 0$, we have
\begin{align*}
[f(x^\ast(s);s) - \mathscr{G}(s)] \geq [f(x^\ast(s);t) - \mathscr{G}(t)] - G_2|s-t| \geq  [f(x^\ast(t);t) - \mathscr{G}(t)] - G_2|s - t|.
\end{align*}

Since this inequality also holds when $s$ and $t$ are exchanged, $f(x^\ast(\cdot);\cdot) - \mathscr{G}(\cdot)$ is $G_2$-Lipschitz.
\end{proof}


\subsection{Analysis for Non-convex functions}
We start with an analysis of the average convergence error of the TVGD algorithm for smooth and possibly non-convex functions. We call $\hat{x}_k$ an $\epsilon$-stationary point of $f(x;t)$
if $\| \nabla_x f(\hat{x}_k;t_k) \| \leq \epsilon$ holds. 

Since Algoirthm~\ref{algo:GD} does not exploit the time variation of the function value, it is invariant to adding terms that depend only on $t$ to the objective function. Therefore, without loss of generality, we assume $\forall t\geq 0,\ \mathscr{G}(t) = 0$ in the remainder of this subsection. We also set the number of correction steps $C$ to 1 for simplicity.
   
\begin{theo}\label{theo:smooth_without_prediction_ANE}
Consider the sequence $\brc{x_k}$ generated by Algorithm~\ref{algo:GD}.
Suppose that Assumptions~\ref{assu:prev_sc}(i) and \ref{assu:lip_t} hold, and set the stepsize as $\beta = 1 / L_1$. Then, for all $k_0\in\mathbb{Z}_{\geq 0}$, 
the average of the gradient norm for $T_{k_0}:= \frac{f_{k_0} - f_{k_0}^\ast}{2h}$ iterations satisfies
$
\frac{1}{T_{k_0}}\sum_{k=k_0}^{k_0 + T_{k_0}-1} \norm*{ \nabla_x f_k}
\leq
2\sqrt{L_1(1+G_2)h}
$. 
\end{theo}

\begin{proof}
For simplicity, we prove only the statement when $k_0 = 0$. 
Since $f$ is $L_1$-smooth in terms of $x$, we have
\begin{align*}
    f(x_{k+1};t_k) - f(x_k;t_k)
    \leq
    \agbra*{\nabla_x f(x_k;t_k), x_{k+1}- x_k} + \frac{L_1}{2} \| x_{k+1} - x_k \|^2.
\end{align*}
From the update rule of $\{x_k\}$ and the Lipschitzness of $f$ in terms of $t$, we have
\begin{align}
    \paren*{\beta - \frac{L_1\beta^2}{2}}\norm*{ \nabla_x f(x_k;t_k)}^2  
    &\leq 
    f(x_{k};t_k) - f(x_{k+1};t_k) \nonumber
    \\
    &\leq 
    (f(x_{k};t_k) - f(x_{k+1};t_{k+1})) + (f(x_{k+1};t_{k+1}) - f(x_{k+1};t_{k})) \nonumber
    \\
    &\leq 
    (f(x_{k};t_k) - f(x_{k+1};t_{k+1})) + G_2h. \label{eq:GD_smooth_inequality}
\end{align}
By summing up the above inequality for all iterations $0\leq k \leq T-1$ and
setting $\beta = 1/L_1$, we obtain
\begin{align*}
    \frac{1}{2L_1} \sum_{k=0}^{T-1} \norm*{ \nabla_x f(x_k;t_k)}^2
    &\leq
    \paren*{f(x_0;t_0) - f(x_T;t_T)} + TG_2h \\
    &\leq
    \paren*{f(x_0;t_0) - f^\ast_T} + TG_2h \\
    &=
    \paren*{f(x_0;t_0) - f^\ast_0} + 2TG_2h,
\end{align*}
where the last equality holds due to Lemma~\ref{lemm:lip_t_optimal_value}.

From
the Cauchy-Schwarz inequality, when the number of iterations is $T = \frac{f(x_0;t_0) - f^\ast_0}{2h}$, we have
\begin{align*}
\frac{1}{T}\sum_{k=0}^{T-1} \norm*{ \nabla_x f(x_k;t_k)}
&\leq \sqrt{ \frac{1}{T} \sum_{k=0}^{T-1} \norm*{ \nabla_x f(x_k;t_k)}^2 }\\
&\leq \sqrt{\frac{2L_1(f(x_0;t_0) - f_0^\ast)}{T} + 4L_1G_2h} \\
&= 2\sqrt{L_1(1+G_2)h}.
\end{align*}
\end{proof}

The above results show that the TVGD algorithm can output $O(\sqrt{h})$-stationary points on average in terms of time, but they do not guarantee that every iterate becomes a good solution after finding an $O(\sqrt{h})$-stationary point. Therefore, we will present another type of theoretical guarantee: once the algorithm converges, all the subsequent iterates satisfy a desirable property in the sense of the function value and gradient norm.

\begin{theo}\label{theo:smooth_without_prediction_last}
    Suppose the same settings as in Theorem~\ref{theo:smooth_without_prediction_ANE}.
    Then, once a $2\sqrt{L_1(1+G_2)h}$-stationary point is reached at iteration $\bar{T}$,
     every subsequent iterate $x_k\ ( k\geq \bar{T})$
     satisfies at least one of the following two conditions:
     \begin{itemize}
         \item[(a)] The iterate $x_k$ is a $2\sqrt{L_1(1+G_2)h}$-stationary point of $f(x;t_k)$. 
         \item[(b)] There exists an integer $l<k$ such that $x_l$ is a $2\sqrt{L_1(1+G_2)h}$-stationary point of $f(x;t_l)$, and 
         $
             (f_k - f_k^\ast) < (f_l - f_l^\ast) + 2G_2h - \frac{1}{2L_1} \| \nabla_x f_l \|^2
        $.
     \end{itemize}
 \end{theo}
 
This theorem implies that once TVGD finds a $2\sqrt{L_1(1+G_2)h}$-stationary point, every subsequent iterate also becomes a $2\sqrt{L_1(1+G_2)h}$-stationary point, or its optimality gap is larger than that at a $2\sqrt{L_1(1+G_2)h}$-stationary point $x_{l}$
by at most
$2G_2h - \frac{1}{2L_1} \| \nabla_x f_l \|^2$. 

 \begin{proof}
    For some $k > \bar{T}$, suppose that (a) does not hold, that is, $\| \nabla_x f(x_k;t_k) \| > 2\sqrt{L_1(1+G_2)h}$ holds. Let $l<k$ be an integer satisfying
     \begin{align*}
         \| \nabla_x f(x_l;t_l) \| &\leq 2\sqrt{L_1(1+G_2)h}, \\ 
         l < \forall j < k,\ \| \nabla_x f(x_j;t_j) \| &> 2\sqrt{L_1(1+G_2)h}.
     \end{align*}
Inequality \eqref{eq:GD_smooth_inequality} with $\beta = 1/L_1$ yields
\begin{align*}
    \forall k \geq 0,\ \frac{1}{2L_1}\norm*{ \nabla_x f(x_k;t_k)}^2 
    &\leq 
    (f(x_{k};t_k) - f(x_{k+1};t_{k+1})) + G_2h,
\end{align*}
     The above inequalities and Lemma~\ref{lemm:lip_t_optimal_value} imply 
     \begin{align*}
         (f(x_{l+1};t_{l+1}) - f(x_{l};t_{l})) - (f_{l+1}^\ast - f_l^\ast) &\leq
         f(x_{l+1};t_{l+1}) - f(x_{l};t_{l}) + G_2h \\
         &\leq \frac{1}{2L_1} \paren*{4L_1G_2h - \norm*{ \nabla_x f(x_l;t_{l})}^2},\\
         l < \forall j < k,\ (f(x_{j+1};t_{j+1}) - f(x_{j};t_{j})) - (f_{j+1}^\ast - f_j^\ast) &\leq
         f(x_{j+1};t_{j+1}) - f(x_{j};t_{j}) + G_2h \\
         &\leq \frac{1}{2L_1} \paren*{4L_1G_2h - \norm*{ \nabla_x f(x_j;t_{j})}^2} \\
         &< \frac{1}{2L_1} \paren*{4L_1G_2h - 4L_1(1+G_2)h} = -2h < 0.
     \end{align*}
 By summing up the above inequalities, we can obtain
 \begin{align*}
     (f(x_{k};t_{k}) - f_{k}^\ast) - (f(x_l;t_l) - f_{l}^\ast) < \frac{1}{2L_1} \paren*{4L_1G_2h - \norm*{ \nabla_x f(x_l;t_{l})}^2},
 \end{align*}
 which implies that (b) holds.
 \end{proof}

\subsection{Analysis for PL functions}
Next, we analyze the performance of the TVGD algorithm for PL functions. Compared with the SC function, the PL function is a more general function class.
\begin{defi}\label{defi:PL}
    The time-varying function $f: \mathbb{R}^d\times\mathbb{R} \to \mathbb{R}$ is a $\mu$-PL function in terms of  $x$ when
$
    \forall x\in\mathbb{R}^d,\ \forall t\geq 0,\ 
    \frac{1}{2} \| \nabla_x f(x;t) \|^2  \geq \mu(f(x;t) - f(x^\ast(t);t)).
$
\end{defi}

It is well known that when the GD algorithm is applied to a time-invariant PL function $f(x)$, the optimality gap $\{ f(x_k)- f^\ast_k \}$ converges linearly to zero~\citep{polyak1963gradient}. By using this property and the Lipschitzness of $f(x;t)$ in terms of $t$, we can easily prove that the optimality gap converges linearly to $O(h)$.

\begin{theo} \label{theo:GD_PL_linear_convergence}
Consider the sequence $\left\{x_k\right\}$ generated by Algorithm~\ref{algo:GD}. Suppose that the objective function $f$ is a $\mu$-PL function in terms of $x$ and that Assumptions~\ref{assu:prev_sc}(i) and \ref{assu:lip_t} hold. Set the stepsize as $\beta = 1/L_1$, and let $\rho:=1-\frac{\mu}{L_1}\in[0, 1)$. Then, the optimality gap $\{ f_k- f_k^\ast \}$ is bounded as follows:
\begin{align*}
    & \forall k \in \mathbb{N},\ f_{k} - f^\ast_{k} \leq \rho^{k} (f_0 - f^\ast_0) + \frac{2(1-\rho^{k})}{1-\rho} G_2h.
\end{align*}
Therefore, the optimality gap converges linearly to an asymptotical error bound:
\begin{align*}
& \lim_{k\to\infty} (f_{k} - f^\ast_{k}) = \frac{2G_2}{1-\rho}h.
\end{align*}
\end{theo}

\begin{proof}
When the (time-invariant) GD algorithm is applied to a PL function $f(\cdot;t_k)$, we have
\begin{align} \label{eq:PL_property}
f(x_{k+1};t_{k}) - f^\ast_{k} \leq \rho (f(x_k;t_{k}) - f^\ast_{k}).
\end{align}
Therefore, we can obtain
\begin{align*}
f(x_{k+1};t_{k+1}) - f^\ast_{k+1}
&\leq 
f(x_{k+1};t_{k}) - f_k^\ast + 2G_2h\\
&\leq 
\rho (f(x_{k};t_{k}) - f_k^\ast) + 2G_2h \\
&\leq 
\rho^{k+1} (f(x_0;t_{0}) - f_0^\ast) + 2\paren*{\sum_{i=0}^k \rho^i} G_2h \\
&=
\rho^{k+1} (f(x_0;t_{0}) - f_0^\ast) + \frac{2(1-\rho^{k+1})}{1-\rho} G_2h,
\end{align*}
where the first inequality follows from Assumption~\ref{assu:lip_t} and Lemma~\ref{lemm:lip_t_optimal_value}. This inequality yields
$$
\lim_{k\to\infty} (f_{k} - f^\ast_{k}) = \frac{2G_2}{1-\rho}h  = O(h).
$$
\end{proof}

We can also prove that when we assume the Lipschitzness of the gradient in terms of $t$ instead of that of the function value, the TVGD algorithm can find an $O(h)$-stationary point, and all the subsequent iterates satisfy a desired property (see Appendix~\ref{chap:PL_lip_grad_t} for more details).
%
%

\section{Proposed Methods} \label{chap:proposed_methods}
\subsection{Prediction-Correction Algorithms for Non-Strongly-Convex Functions}
If we implement no prediction, the gap of the function values at two subsequent time steps becomes $f(x_k;t_{k+1}) - f_k = O(h)$. 
Here, we denote a predicted solution at time $t_{k+1}$ based on the information aqcuired at time $t_k$ as $x_{k+1|k}$, and
we predict $x_{k+1 \mid k}$ in a way that $f(x_{k+1 \mid k};t_{k+1}) - f_k$ becomes smaller than $O(h)$ (Lemma~\ref{lemm:func_value_error_foa_min_cp}).
The function value at the next time step can be approximated by a first-order Taylor expansion,
\begin{align*}
    \hat{f}^1(x;t_{k+1})
    &:= f_k + \agbra*{\nabla_x f_k, x-x_k} + h\nabla_t f_k.
\end{align*} However, the approximation accuracy becomes worse when $\| x - x_k \|$ is large. To avoid it, we consider the following constrained optimization problem:
\begin{align*}
    \min_{x \in \mathbb{R}^d} \hat{f}^1(x;t_{k+1})\ \ \mathrm{s.t.}\  \| x - x_k \| \leq \zeta h,
\end{align*}
where $\zeta>0$ is some constant.
When the gradient norm $\| \nabla_x f_k \|$ is not zero, the solution is
$$x_{k+1 \mid k} = x_k - \zeta h \frac{\nabla_x f_k}{\| \nabla_x f_k \|}.$$
If we further assume that $|\nabla_t f_k| - \zeta \| \nabla_x f_k \| \leq 0$ holds (we will assume so later in Assumption~\ref{assu:prediction-correction}), we can obtain $\hat{f}^1(x_{k+1 \mid k};t_{k+1}) \leq f_k$ and $f(x_{k+1 \mid k};t_{k+1}) - f_k = O(h^2)$. 

\begin{figure}[h]
    \begin{algorithm}[H]
    \caption{First-Order Approximation Minimization (FOA-Min) /  Cauchy Point (CP) } 
    \begin{algorithmic}[1]\label{algo:foa_min_cp}
      \REQUIRE Initial solution $x_{0 \mid -1}$, number of correction steps $C$, step size for correction step $\beta$, radius $\zeta$, sufficiently small positive value $\delta$ \\
    \FOR{$k=0,1,2,\ldots$}
        \STATE // time $t_k$ ($:= kh$)
        \STATE Incur the loss $f(x_{k \mid k-1};t_{k})$
        \STATE Acquire the objective function $f(\cdot;t_{k})$
        \STATE Initialize the sequence of corrected variables as $\hat{x}_k^0 = x_{k\mid k-1}$
        \FOR{$c=0,1,\ldots,C-1$}
            \STATE 
            $
            \hat{x}_{k}^{c+1} = \hat{x}_k^{c} - \beta \nabla_x f(\hat{x}_k^c;t_{k})
            $
        \ENDFOR
        \STATE Set the correction as $x_{k} = \hat{x}_k^C$
        \STATE Set \label{line:prediction_start}
        $
        g_k = \nabla_x f_k\ \  \mathrm{or}\ \  g_k = 2\nabla_x f_k - \nabla_x f(x_k;t_{k-1}) 
        $
        \IF{ $\| g_k \| \leq \delta$} \label{line:if_gk_is_small}
        \STATE $ x_{k+1\mid k} = x_k $
        \ELSE
        \STATE \ \\ \vspace{-12mm}
        \begin{align*}
            &x_{k+1\mid k} = x_k - 
            \left\{
            \begin{array}{cc}
            \zeta h \frac{g_k}{\| g_k \|} & (\textrm{FOA-Min}) \\
            \zeta h \frac{g_k}{\| g_k \|} & (\textrm{CP}, g_k^\top \nabla_{xx} f_k g_k \leq 0) \\
            \min\left\{ \frac{\|g_k \|^3}{g_k^\top \nabla_{xx} f_k g_k}, \zeta h \right\} \frac{g_k}{\| g_k \|} & (\textrm{CP}, g_k^\top \nabla_{xx} f_k g_k > 0) \\
            \end{array}
            \right.
        \end{align*}
        \ENDIF \label{line:prediction_end}
    \ENDFOR
    \end{algorithmic}
    \end{algorithm}
\end{figure}

We can also consider optimizing the second-order Taylor expansion,
\begin{align*}
    &\hat{f}^2(x;t_{k+1})
    := f_k + \agbra*{\nabla_x f_k, x-x_k} + h\nabla_t f_k\\
    &+ \frac{1}{2}(x-x_k)^\top \nabla_{xx} f_k(x - x_k) + h\agbra*{\nabla_{tx}f_k, x - x_k} + \frac{h^2}{2}\nabla_{tt}f_k
\end{align*}
with $\| x - x_k \| \leq \zeta h$. This type of problem appears in the trust-region subproblem, and its approximate solution called the Cauchy point~\citep{nocedal2006numerical} can be computed in closed form:

\begin{align*}
    x_{k+1 \mid k} =
    \left\{
        \begin{array}{cc}
            \zeta h \frac{\check{g}_k}{\| \check{g}_k \|} & (\check{g}_k^\top \nabla_{xx} f_k \check{g}_k \leq 0) \\
            \min\left\{ \frac{\|\check{g}_k \|^3}{\check{g}_k^\top \nabla_{xx} f_k \check{g}_k}, \zeta h \right\} \frac{\check{g}_k}{\| \check{g}_k \|} & (\check{g}_k^\top \nabla_{xx} f_k \check{g}_k > 0)
        \end{array}
    \right.,
\end{align*}
where $\check{g}_k$ is defined as $\check{g}_k:=\nabla_x f_k + h\nabla_{tx} f_k$ and we assume that its norm is non-zero.
Since $\nabla_{tx} f_k$ is not generally available, we need to approximate it by a backward difference $\tilde{\nabla}_{tx} f_k$. Then, an approximation of $\check{g}_k$ can be computed as
\[\nabla_x f_k + h \tilde{\nabla}_{tx} f_k = \nabla_x f_k + h \frac{ \nabla_x f_k - \nabla_x f(x_k;t_{k-1}) }{h} = 2\nabla_x f_k - \nabla_x f(x_k;t_{k-1}).\]
Again, if $|\nabla_t f_k| - \zeta \| \nabla_x f_k \| \leq 0$ holds, $f(x_{k+1 \mid k};t_{k+1}) - f_k = O(h^2)$ is satisfied even when we use this finite difference approximation.

The details of the proposed algorithms are described in Algorithm~\ref{algo:foa_min_cp}.
Although in the above explanation, $-\nabla_x f_k$ and $-(\nabla_x f_k + h \nabla_{tx} f_k)$ were chosen for the moving directions in the first- and second-order approximation optimizations, respectively, we can use both of them in both optimizations. The prediction is not implemented when the norm of $g_k$ is sufficiently small.

FOA-Min updates the iterates in the direction $-g_k$ with an adaptive stepsize normalized by the norm of $g_k$.
This normalization technique is used by common optimizers in machine learning, such as RMSProp~\citep{hinton2012neural} and Adam~\citep{kingma2014adam}, and is even applicable to large-scale problems. CP has the potential to outperform FOA-Min by minimizing the second-order Taylor expansion more precisely when the problem size is so small that Hessian can be efficiently computed.

\subsection{Assumptions}
Here, we introduce a new assumption, under which $f(x_{k+1 \mid k};t_{k+1}) - f_k$ can be improved from $O(h)$ to $O(h^2)$.

\begin{assu}\label{assu:prediction-correction}
There exists a function $\mathscr{G}:\mathbb{R}\to\mathbb{R}$ such that $f(x;\cdot) - \mathscr{G}(\cdot)$ is twice differentiable for any $x\in\mathbb{R}^d$, and it holds that
    \begin{itemize} \setlength{\itemsep}{6pt}
        \item[$\mathrm{(i)}$]
        $\forall x\in\mathbb{R}^d, \forall t\geq 0,\ |\nabla_{t} [f(x;t) -  \mathscr{G}(t)]| \leq G_2$.
        \item[$\mathrm{(ii)}$]
        $\forall x\in\mathbb{R}^d, \forall t\geq 0,\ |\nabla_{tt} [f(x;t) - \mathscr{G}(t)]| \leq L_3$.
        \item[$\mathrm{(iii)}$]
         $\exists Z \geq 0,\ \forall x\in\mathbb{R}^d, \forall t\geq 0,\ |\nabla_t [f(x;t) - \mathscr{G}(t)]| - Z \| \nabla_x f(x;t) \| \leq 0.$
    \end{itemize}
\end{assu}

We use a function $\mathscr{G}$ as in Assumption~\ref{assu:lip_t} and assume that the redefined objective function 
\[\bar{f}(x;t) := f(x;t) - \mathscr{G}(t)\]
is twice differentiable, $G_2$-Lipschitz, and $L_3$-smooth in terms of $t$ in Assumptions~\ref{assu:prediction-correction}(i)-(ii).
Although Assumption~\ref{assu:prediction-correction}(iii) seems rather strong, we will see that a wide range of time-varying optimization problems satisfy it by reformulating the optimization problem.
First, we show that the SC optimization problem satisfies it. Here, the boundedness of third derivatives, which is a common assumption in the analyses of prediction-correction algorithms (see e.g. ~\citep{bastianello2020primal, lin2019simplified}), is not required.

\begin{prop} \label{prop:SC_satisfies_assump}
Let $f(x;t)$ be an objective function satisfying Assumptions~\ref{assu:prev_sc} and \ref{assu:prediction-correction}(i)-(ii). We also assume that its gradient $\nabla_x f(x;t)$ is continuously differentiable in terms of $x$ and $t$. 
Then, the function $\check{f}(x;t) := \bar{f}(x;t) - \int_0^t \nabla_t \bar{f}(y;\tau)|_{y = x^\ast(\tau)} \mathrm{d}\tau$ has continuously differentiable gradient in terms of $x$ and $t$, satisfies Assumption~\ref{assu:prev_sc}, and satisfies Assumption~\ref{assu:prediction-correction}(i)-(iii) as follows:
    \begin{align*}
        &\forall x\in\mathbb{R}^d, \forall t\geq 0,\ | \nabla_{t} \check{f}(x;t) | \leq 2G_2,\\
        &| \nabla_{tt} \check{f}(x;t) | \leq 2L_3 + \frac{L_2^2}{m},\ |\nabla_t \check{f}(x;t)| - \frac{L_2}{m} \| \nabla_x \check{f}(x;t) \| \leq 0.
    \end{align*}
\end{prop}

This proposition claims that we can redefine the objective function equivalently to satisfy Assumption~\ref{assu:prediction-correction}(iii) while preserving other assumptions.

\begin{proof}
    Since $\check{f}(x;t) - \bar{f}(x;t)$ does not depend on $x$, $\check{f}(x;t)$ satisfies Assumptions~\ref{assu:prev_sc}, and $\nabla_x \check{f}(x;t) = \nabla_x f(x;t)$ is continuously differentiable in terms of $x$ and $t$. The differentiability and Lipschitzness of $\check{f}(x;t)$ in terms of $t$  follows from
    \begin{align*}
        \nabla_{t} \check{f}(x;t) &= \nabla_{t} \bar{f}(x;t) - \nabla_{t} \bar{f}(y;t)|_{y = x^\ast(t)},\\
        | \nabla_{t} \check{f}(x;t) | &\leq | \nabla_{t} \bar{f}(x;t) | + | \nabla_{t} \bar{f}(y;t)|_{y = x^\ast(t)} | \leq 2G_2.
    \end{align*}
    Let us prove that the remaining Assumptions~\ref{assu:prediction-correction}(ii)-(iii) also hold in the following.

    \subsubsection*{(a) $| \nabla_{tt} \check{f}(x;t) | \leq 2L_3 + \frac{L_2^2}{m}$}$\ $

For any $t \geq 0$, the first-order optimality for $\bar{f}(\cdot;t)$ yields
\begin{align*}
    \nabla_x \bar{f}(x^\ast(t);t) = 0.
\end{align*}
Since the Hessian matrix $\nabla_{xx} \bar{f}(x;t)$ is regular due to Assumption~\ref{assu:prev_sc}(iii), and $\nabla_x \bar{f}(x;t)$ is continuously differentiable in terms of $x$ and $t$, we can apply the implicit function theorem:
\begin{align*}
    \frac{\mathrm{d}x^\ast(t)}{\mathrm{d}t} = - [\nabla_{xx} \bar{f}(x^\ast(t);t)]^{-1} \nabla_{tx} \bar{f}(y;t)|_{y = x^\ast(t)}.
\end{align*}

Next, we consider the following finite difference: 
\begin{align}
    &\frac{\nabla_t \bar{f}(y;t+\Delta t)|_{y = x^\ast(t + \Delta t)} - \nabla_t \bar{f}(y;t)|_{y = x^\ast(t)}}{\Delta t} \label{eq:limiting_value} \\
    &= \frac{\nabla_t \bar{f}(y;t+\Delta t)|_{y = x^\ast(t + \Delta t)} - \nabla_t \bar{f}(y;t)|_{y = x^\ast(t + \Delta t)}}{\Delta t} - \frac{\nabla_t \bar{f}(y;t)|_{y = x^\ast(t + \Delta t)} - \nabla_t \bar{f}(y;t)|_{y = x^\ast(t)}}{\Delta t}. \nonumber
\end{align}
   When $\Delta t$ goes to 0, the first term in the last line converges to $\nabla_{tt} \bar{f}(y;t)|_{y = x^\ast(t)}$, and the second term in the last line converges to
\begin{align*}
   \agbra*{ \nabla_{tx} \bar{f}(y;t)|_{y = x^\ast(t)}, \frac{\mathrm{d}}{\mathrm{d}t}x^\ast(t) } = - \nabla_{tx} \bar{f}(y;t)|_{y = x^\ast(t)}^\top  [\nabla_{xx} \bar{f}(x^\ast(t);t)]^{-1} \nabla_{tx} \bar{f}(y;t)|_{y = x^\ast(t)}.
\end{align*} This implies that \eqref{eq:limiting_value} also converges in the limit of $\Delta t \to 0$, and the limiting value is $\nabla_t(\nabla_t \bar{f}(y;t)|_{y = x^\ast(t)})$.
Hence, we can obtain
\begin{align*}
 & | \nabla_{tt} \check{f}(x;t) |\\
 & \leq | \nabla_{tt} \bar{f}(x;t)| + |\nabla_t (\nabla_t \bar{f}(y;t)|_{y = x^\ast(t)}) |\\
 & = \abs*{ \nabla_{tt} \bar{f}(x;t) } + \abs*{ \nabla_{tt} \bar{f}(y;t)|_{y = x^\ast(t)} + \nabla_{tx} \bar{f}(y;t)|_{y = x^\ast(t)}^\top  [\nabla_{xx} \bar{f}(x^\ast(t);t)]^{-1} \nabla_{tx} \bar{f}(y;t)|_{y = x^\ast(t)}} \\
 & \leq \abs*{ \nabla_{tt} \bar{f}(x;t) } + \abs*{ \nabla_{tt} \bar{f}(y;t)|_{y = x^\ast(t)} } + \abs*{ \nabla_{tx} \bar{f}(y;t)|_{y = x^\ast(t)}^\top  [\nabla_{xx} \bar{f}(x^\ast(t);t)]^{-1} \nabla_{tx} \bar{f}(y;t)|_{y = x^\ast(t)} } \\
 &\leq 2L_3 + \frac{L_2^2}{m},
\end{align*}
where the last inequality holds due to Assumptions~\ref{assu:prev_sc}(ii)-(iii) and \ref{assu:prediction-correction}(ii).

\subsubsection*{(b) $|\nabla_t \check{f}(x;t)| - \frac{L_2}{m} \| \nabla_x \check{f}(x;t) \| \leq 0$} $\ $

By the definition of $\check{f}(x;t)$ and $\nabla_x \bar{f}(x^\ast(t);t) = 0$, we have
\begin{align*}
    |\nabla_t \check{f}(x;t)| &= \abs*{ \nabla_t \bar{f}(x;t) - \nabla_t \bar{f}(y;t)|_{y = x^\ast(t)} } \\
    \| \nabla_x \check{f}(x;t) \| &= \| \nabla_x \bar{f}(x;t) - \nabla_x \bar{f}(x^\ast(t);t) \|.
\end{align*}
Since $\bar{f}(x;t)$ is SC in terms of $x$, a stationary point is always an optimal point.
Thus, when $\| \nabla_x \check{f}(x;t) \| = 0$ holds, we get $x = x^\ast(t)$, which yields $|\nabla_t \check{f}(x;t)| - \frac{L_2}{m} \| \nabla_x \check{f}(x;t) \| = \abs*{ \nabla_t \bar{f}(x;t) - \nabla_t \bar{f}(y;t)|_{y = x^\ast(t)} } = 0$.
When $\| \nabla_x \check{f}(x;t) \| = 0$ does not hold, we have
\begin{align*}
    \frac{|\nabla_t \check{f}(x;t)|}{ \| \nabla_x \check{f}(x;t) \|}
    &= \frac{\abs*{ \nabla_t \bar{f}(x;t) - \nabla_t \bar{f}(y;t)|_{y = x^\ast(t)} }}{\| \nabla_x \bar{f}(x;t) - \nabla_x \bar{f}(x^\ast(t);t) \|} \\
    &\leq \frac{\abs*{ \agbra*{ \nabla_{tx} \bar{f}(y;t)|_{y = y_1}, x - x^\ast(t) }}}{\norm*{\agbra*{ \nabla_{xx} \bar{f}(y_2;t), x - x^\ast(t)}}} \leq \frac{L_2}{m},
\end{align*}
where the first inequality follows from Taylor's theorem, and
$y_1, y_2\in[\min(x, x^\ast(t)), \max(x^\ast(t), x)]$ (the operators $\min(\cdot)$ and $\max(\cdot)$ are element-wise).
The last inequality holds due to Assumption~\ref{assu:prev_sc}(ii) and (iii) and Lemma~\ref{lemm:ratio_between_inner_products} in Appendix~\ref{chap:proofs_for_section4}.
\end{proof}

The next proposition implies that when a function satisfies Assumption~\ref{assu:prediction-correction}(iii), nonlinear transformations of it still satisfy this assumption.

\begin{prop} \label{prop:nonlinear_transform_satisfies_assump}
    Given a nonlinear transformation $f_1(y):\mathbb{R}^m \to \mathbb{R}$ and a time-varying function $f_2(x;t):\mathbb{R}^d \times \mathbb{R} \to \mathbb{R}^m$, define $f(x;t) := f_1(f_2(x;t))$.
    Suppose that $\nabla_x f_2(x;t)$ is row full rank for any $x \in \mathbb{R}^d$ and $t \geq 0$, and let $\sigma_{\min}(x;t)$ be its minimum singular value.
    Then, we have
    $ 
     \frac{| \nabla_t f(x;t) |}{ \| \nabla_x f(x;t) \|} \leq \frac{\|\nabla_t f_2(x;t)\|}{ \sigma_{\min}(x;t) }.
    $
    In particular, when $m=1$, we have
    $
     \frac{|\nabla_t f(x;t)|}{ \| \nabla_x f(x;t) \|} \leq \frac{|\nabla_t f_2(x;t)|}{ \| \nabla_x f_2(x;t) \|}.
    $
\end{prop}

\begin{proof}
The derivative and gradient of $f$ can be written as
    \begin{align*}
    \nabla_t f(x;t) &= \agbra*{ \nabla_y f_1(f_2(x;t)), \nabla_t f_2(x;t)}, \\
    \nabla_x f(x;t) &= \agbra*{ \nabla_y f_1(f_2(x;t)), \nabla_x f_2(x;t)}.
    \end{align*}
    Thus,  Lemma~\ref{lemm:ratio_between_inner_products} yields
    $$ 
     \frac{| \nabla_t f(x;t) |}{ \| \nabla_x f(x;t) \|} \leq \frac{\|\nabla_t f_2(x;t)\|}{ \sigma_{\min}(x;t) }.
    $$    
   When $m=1$ holds, we have $\sigma_{\min}(x;t) = \| \nabla_x f_2(x;t) \|$, which yields
    $$ 
     \frac{| \nabla_t f(x;t) |}{ \| \nabla_x f(x;t) \|} \leq \frac{|\nabla_t f_2(x;t)|}{ \| \nabla_x f_2(x;t) \| }.
    $$    
\end{proof}


By exploiting the above proposition, we can see that an important function class satisfies Assumption~\ref{assu:prediction-correction}(iii).

\begin{coro} \label{coro:linear_satisfies_assump}
Let $A(t) \in \mathbb{R}^{m \times d}$ be a time-varying row full rank matrix. Denote its minimum singular value by $\sigma_{\min}(t)$ and suppose that its derivative is bounded as $\|A'(t) \| \leq G_A$. Given the matrix $A(t)$, a time-varying vector $b(t)\in \mathbb{R}^{m}$ whose derivative is bounded as $\|b'(t) \| \leq G_B$, and a nonlinear transformation $f_1(y):\mathbb{R}^m \to \mathbb{R}$, define $f(x;t) = f_1(\agbra*{A(t), x } + b(t))$. Then, when $\| x \| \leq R$ holds, we have
$
     \frac{|\nabla_t f(x;t)|}{ \| \nabla_x f(x;t) \|} \leq \frac{ G_AR + G_B}{ \sigma_{\min}(t) }.
$
\end{coro}

This corollary implies that the cost function in the linear regression $\| \agbra*{A(t), x } - b(t) \|^2$ and functions with a parallel shift $g(x + h(t))$ also satisfy Assumption~\ref{assu:prediction-correction}(iii).

In Section~\ref{chap:analysis_of_gd}, the optimal value of the objective function was proved to be Lipschitz continuous under Assumption~\ref{assu:lip_t}. In the following, we prove a stronger result that the optimal value is constant under Assumption~\ref{assu:prediction-correction}. 
It plays a pivotal role to improve $f(x_{k+1 \mid k};t_{k+1}) - f_k$ from $O(h)$ to $O(h^2)$.

\begin{lemm} \label{lemm:optimal_value_const_assu4}
   When Assumption~\ref{assu:prediction-correction}(iii) holds, we have 
   $\forall t \geq 0,\ \bar{f}(x^\ast(t);t) = const.$
\end{lemm}

\begin{proof}
    Since we can get $\abs*{ \nabla_t \bar{f}(y;t)|_{y=x^\ast(t)} } \leq Z \| \nabla_x \bar{f}(x^\ast(t);t) \| = 0$ from Assumption~\ref{assu:prediction-correction}(iii), we have
    \[
        \frac{\mathrm d}{\mathrm d t}\bar f(x^\ast(t);t)
        =\agbra*{ \nabla_x\bar f(x^\ast(t);t), \frac{\mathrm{d}}{\mathrm{d}t}x^\ast(t) }+\nabla_t \bar{f}(y;t)|_{y=x^\ast(t)} =0,
    \]
    which implies $\bar{f}(x^\ast(t);t)$ is constant.
    \end{proof}

\subsection{Convergence Error Analysis}
As in Section~\ref{chap:analysis_of_gd}, without loss of generality, we can assume $\forall t\geq 0,\ \mathscr{G}(t)=0$  in the remainder of this subsection since Algoirthm~\ref{algo:foa_min_cp} does not exploit the time variation of the function value.

The following lemma shows that $f(x_{k+1\mid k};t_{k+1}) - f(x_k;t_k)$ becomes the same order as that of TVGD under the same settings and can be improved if Assumption~\ref{assu:prediction-correction} holds as well. 

\begin{lemm}\label{lemm:func_value_error_foa_min_cp}
    Consider the sequences $\{ x_k \}$ and $\{ x_{k+1\mid k} \}$ generated by
    Algorithm~\ref{algo:foa_min_cp}.
    \begin{itemize}
        \item[(a)]
    Suppose that Assumptions~\ref{assu:prev_sc}(i)-(ii) and \ref{assu:lip_t} hold.
    Then, $f(x_{k+1\mid k};t_{k+1}) - f(x_k;t_k) = O(h)$.
        \item[(b)]
    Suppose that Assumptions~\ref{assu:prev_sc}(i)-(ii) and \ref{assu:prediction-correction} hold.
    If the parameters $\zeta, \delta$ satisfy $\zeta \geq Z$ and $\delta \leq h$,
    the prediction yields better bounds: $f(x_{k+1\mid k};t_{k+1}) - f(x_k;t_k) = O(h^2)$.
    \end{itemize}
\end{lemm}

\begin{proof} $\ $
    \subsubsection*{(a) $\ $}
    We prove that Algorithm~\ref{algo:foa_min_cp} achieves $f(x_{k+1\mid k};t_{k+1}) - f(x_k;t_k) = O(h)$ regardless of choice of $g_k$.
    By the definition of the prediction in Lines \ref{line:prediction_start}-\ref{line:prediction_end} in Algoirthm~\ref{algo:foa_min_cp}, the prediction $x_{k+1 \mid k}$ can always be denoted as $x_{k+1 \mid k} = x_k - C_1\zeta h g_k / \norm*{g_k}$, where $C_1$ satisfies $0 \leq C_1 \leq 1$, thus, the moving distance $\| x_{k+1 \mid k} - x_k \|$ can be bounded as
    \begin{align}\label{eq:foa_min_cp_moving_distance_bound}
        \| x_{k+1\mid k} - x_k \| \leq \zeta h.
    \end{align}
    Hence, we can obtain 
    \begin{align*}
        f(x_{k+1 \mid k};t_{k+1}) - f_k &= [f(x_{k+1 \mid k};t_{k}) - f_k] + [f(x_{k+1 \mid k};t_{k+1}) - f(x_{k+1 \mid k};t_k)] \\
        &\leq \agbra*{\nabla_x f_k, x_{k+1 \mid k} - x_k} + \frac{L_1}{2}\norm*{x_{k+1 \mid k} - x_k}^2 + G_2h \\
        &= -\frac{C_1\zeta h}{\| g_k \|} \agbra*{\nabla_x f_k, g_k} + \frac{L_1}{2}\norm*{x_{k+1 \mid k} - x_k}^2 + G_2h \\
        &\leq C_1\zeta h \agbra*{g_k - \nabla_x f_k, \frac{g_k}{\| g_k \|}} -C_1\zeta h \| g_k \| + \frac{L_1 \zeta^2}{2}h^2 + G_2h \\
        &\leq C_1\zeta h\| g_k - \nabla_x f_k \| + \frac{L_1 \zeta^2}{2}h^2 + G_2h \\
        &\leq \paren*{C_1L_2\zeta + \frac{L_1 \zeta^2}{2}} h^2 + G_2h = O(h), 
    \end{align*}
    where the first inequality follows from $G_2$-Lipschitzness of $f(x;t)$ in terms of $t$ and $L_1$-smoothness in terms of $x$, and the final inequality holds due to the definition of $g_k$ and Assumption~\ref{assu:prev_sc}(ii).
    
    \textbf{(b)} $\ $
    We first prove that $f(x_{k+1 \mid k};t_{k+1}) - f_k = O(h^2)$ holds for FOA-Min.
    
    
    \subsubsection*{Case 1} ($g_k = \nabla_x f_k$)
    
    From Taylor's theorem, we have
    \begin{align}
        f(x;t_{k+1})
        &= f_k + \agbra*{\nabla_x f_k, x-x_k} + h\nabla_t f_k + \frac{1}{2}(x-x_k)^\top \nabla_{xx}f(y_k;s_k)(x - x_k) \nonumber \\
        &+ h\agbra*{\nabla_{tx}f(y_k;s_k), x - x_k} + \frac{h^2}{2}\nabla_{tt}f(y_k;s_k), \nonumber 
    \end{align}
    where $s_k\in[t_k, t_{k+1}],\ y_k\in[\min(x, x_k), \max(x, x_k)]$ (the operators $\min(\cdot)$ and $\max(\cdot)$ are element-wise). When the gradient norm is so small that the prediction is not implemented (see Line~\ref{line:if_gk_is_small} in Algorithm~\ref{algo:foa_min_cp}), we have
    \begin{align*}
        \agbra*{\nabla_x f_k, x_{k+1 \mid k}-x_k} + h\nabla_t f_k = h\nabla_t f_k \leq Zh \| \nabla_x f_k \| \leq \zeta h\delta.
    \end{align*}
    The following inequality holds when the prediction is implemented:
    \begin{align*}
    \agbra*{\nabla_x f_k, x_{k+1 \mid k}-x_k} + h\nabla_t f_k
    &= h(\nabla_t f_k - \zeta\| \nabla_x f_k \|) \\
    &\leq h(\nabla_t f_k - Z\| \nabla_x f_k \|) \leq 0.
    \end{align*}
    Therefore, together with boundedness of second derivatives and \eqref{eq:foa_min_cp_moving_distance_bound}, we can obtain
    \begin{align}
        &f(x_{k+1\mid k};t_{k+1}) - f_k \nonumber \\
        &\leq \zeta h\delta + \frac{1}{2}(x_{k+1\mid k}-x_k)^\top \nabla_{xx}f(y_k;s_k)(x_{k+1\mid k} - x_k)+ h \agbra*{\nabla_{tx}f(y_k;s_k), x_{k+1\mid k} - x_k} + \frac{h^2}{2}\nabla_{tt}f(y_k;s_k) \nonumber \\
        &\leq \zeta h\delta + \frac{L_1}{2}\norm*{x_{k+1\mid k}-x_k}^2 + hL_2 \| x_{k+1\mid k} - x_k \| + \frac{h^2L_3}{2} \nonumber \\
        &\leq h^2 \paren*{ \frac{\zeta^2 L_1}{2} + \zeta L_2 + \frac{L_3}{2} + \frac{\zeta\delta}{h} } \label{eq:pred_corr_error} \\
        &= O\paren*{h^2}, \nonumber
    \end{align} where the last equality follows from $\delta \leq h$. 
    
\subsubsection*{Case 2} ($g_k = 2\nabla_x f_k - \nabla_x f(x_k;t_{k-1})$)

    Define $\tilde{\nabla}_{tx} f_k = \frac{\nabla_x f_k - \nabla_x f(x_k;t_{k-1})}{h}$, then $g_k$ can be denoted by $g_k = \nabla_x f_k + h\tilde{\nabla}_{tx} f_k$.
    We can see that Assumption~\ref{assu:prev_sc}(ii) yields
    \begin{align*}
        \| \tilde{\nabla}_{tx} f_k \| &\leq \frac{\| \nabla_x f_k - \nabla_x f(x_k;t_{k-1}) \|}{h} \leq L_2, \\
        \|\tilde{\nabla}_{tx} f_k - \nabla_{tx} f_k \|
        &\leq \| \tilde{\nabla}_{tx} f_k \| + \| \nabla_{tx} f_k \| \leq 2L_2.
    \end{align*}
    Furthermore, if Assumption~\ref{assu:prediction-correction}(iii) also holds and the parameter $\zeta$ satisfies $\zeta \geq Z$, we have 
    \begin{align*}
    &\agbra*{\nabla_x f_k, x_{k+1 \mid k}-x_k} + h\nabla_t f_k\\
    &=\agbra*{\nabla_x f_k, -\zeta h\frac{\nabla_x f_k + h \tilde{\nabla}_{tx} f_k}{\| \nabla_x f_k + h \tilde{\nabla}_{tx} f_k \|}} + h\nabla_t f_k \\
    &=h\nabla_t f_k + \agbra*{\nabla_x f_k + h \tilde{\nabla}_{tx} f_k, -\zeta h\frac{\nabla_x f_k + h \tilde{\nabla}_{tx} f_k}{\| \nabla_x f_k + h \tilde{\nabla}_{tx} f_k \|}} - \agbra*{h \tilde{\nabla}_{tx} f_k, -\zeta h\frac{\nabla_x f_k + h \tilde{\nabla}_{tx} f_k}{\| \nabla_x f_k + h \tilde{\nabla}_{tx} f_k \|}} \\
    &= h(\nabla_t f_k - \zeta\| \nabla_x f_k + h\tilde{\nabla}_{tx} f_k \|)
    + h^2\zeta \agbra*{\tilde{\nabla}_{tx} f_k, \frac{\nabla_x f_k + h \tilde{\nabla}_{tx} f_k}{\| \nabla_x f_k + h \tilde{\nabla}_{tx} f_k \|}} \\
    &\leq h(\nabla_t f_k - \zeta \| \nabla_x f_k \| + \zeta h \ \|\tilde{\nabla}_{tx} f_k \|)
    + h^2\zeta \| \tilde{\nabla}_{tx} f_k \| \\
    &\leq h(\nabla_t f_k - Z\| \nabla_x f_k \|) + 2h^2\zeta L_2\\
    &\leq 2h^2\zeta L_2.
    \end{align*}
    Therefore, together with the boundedness of second derivatives and \eqref{eq:foa_min_cp_moving_distance_bound}, we can obtain
    \begin{align*}
        &f(x_{k+1\mid k};t_{k+1}) - f_k \\
        &= 2h^2 \zeta L_2 \\
        &+ \frac{1}{2}(x_{k+1\mid k}-x_k)^\top \nabla_{xx}f(y_k;s_k)(x_{k+1\mid k} - x_k)+ h \agbra*{\nabla_{tx}f(y_k;s_k), x_{k+1\mid k} - x_k} + \frac{h^2}{2}\nabla_{tt}f(y_k;s_k) \\
        &= O\paren*{h^2}.
    \end{align*}
    Next, we prove that $\{ x_k \}$ and $\{ x_{k+1 \mid k} \}$ generated by CP also satisfy the statement regardless of the choice of $g_k$. Let $\hat{x}_{k+1 \mid k} = x_k - \zeta h \frac{g_k}{\| g_k \|}$ be a point generated by following the update rule of FOA-Min.
    Since $x_{k+1 \mid k}$ is the Cauchy point of $\hat{f}(x) := f_k + \agbra*{g_k, x-x_k}
    + h\nabla_t f_k
    + \frac{1}{2}(x-x_k)^\top \nabla_{xx}f_k(x - x_k)
    $, we obtain $\hat{f}(x_{k+1 \mid k}) \leq \hat{f}(\hat{x}_{k+1 \mid k})$ \citep{nocedal2006numerical}. Together with the boundedness of $\nabla_{xx} f(x;t)$ and $\nabla_{tx} f(x;t)$, \eqref{eq:foa_min_cp_moving_distance_bound} and $\| \hat{x}_{k+1 \mid k} - x_k \| \leq \zeta h$, we can obtain
    \begin{align*}
        &f(x_{k+1 \mid k};t_{k+1}) - f_k \\
        &=[f(x_{k+1 \mid k};t_{k+1}) - \hat{f}(x_{k+1 \mid k})] + [\hat{f}(x_{k+1 \mid k}) - \hat{f}(\hat{x}_{k+1 \mid k})] \\
        &- [f(\hat{x}_{k+1 \mid k};t_{k+1}) - \hat{f}(\hat{x}_{k+1 \mid k})] + [f(\hat{x}_{k+1 \mid k};t_{k+1}) - f_k ]\\
        &= [\hat{f}(x_{k+1 \mid k}) - \hat{f}(\hat{x}_{k+1 \mid k})] + [f(\hat{x}_{k+1 \mid k};t_{k+1}) - f_k ]\\
        &+ \left[\langle \nabla_x f_k + h \nabla_{tx} f(y_k;s_k) - g_k, x_{k + 1 \mid k} - x_k \rangle\right.\\
        &+ \left. \frac{1}{2} (x_{k+1 \mid k} - x_k)^\top \paren*{\nabla_{xx} f(y_k;s_k) - \nabla_{xx} f_k }(x_{k+1 \mid k} - x_k) + \frac{h^2}{2} \nabla_{tt} f(y_k;s_k) \right]\\
        &- \left[\langle \nabla_x f_k + h \nabla_{tx} f(\hat{y}_k;\hat{s}_k) - g_k, \hat{x}_{k + 1 \mid k} - x_k \rangle \right. \\
        &+ \left. \frac{1}{2} (\hat{x}_{k + 1 \mid k} - x_k)^\top \paren*{\nabla_{xx} f(\hat{y}_k;\hat{s}_k) - \nabla_{xx} f_k }(\hat{x}_{k + 1 \mid k} - x_k) + \frac{h^2}{2} \nabla_{tt} f(\hat{y}_k;\hat{s}_k) \right] \\
        &\leq f(\hat{x}_{k+1 \mid k};t_{k+1}) - f_k + O(h^2) = O(h^2),
    \end{align*}
    where the second equality holds due to Taylor's theorem, and  $s_k\in[t_k, t_{k+1}],\ y_k\in[\min(x, x_k), \max(x, x_k)],\ \hat{s}_k\in[t_k, t_{k+1}],\ \hat{y}_k\in[\min(x, \hat{x}_k), \max(x, \hat{x}_k)]$ (the operators $\min(\cdot)$ and $\max(\cdot)$ are element-wise).
\end{proof}

We can analyze the convergence results of the proposed algorithms for non-convex and PL functions by using this lemma and Lemma~\ref{lemm:optimal_value_const_assu4}.
We can not only provide an analysis of the average convergence error but also guarantee the quality of every solution after convergence, as in Section~\ref{chap:analysis_of_gd}. When Assumption~\ref{assu:prediction-correction} holds, the convergence error is squared in comparison to TVGD because of the improvement in the value of $f(x_{k+1 \mid k};t_{k+1}) - f_k$, and it matches that of TVGD otherwise. In the following analyses, we will only present the former results. We also set the number of correction steps $C$ to 1 for simplicity.


 For non-convex objective functions, we will provide convergence analyses of only FOA-Min with $g_k = \nabla_x f_k$, although similar results can be obtained for different choices of $g_k$ and algorithms. Here, let us define $\bar{G}_2:= \frac{\zeta^2 L_1}{2} + \zeta L_2 + \frac{L_3}{2} + \frac{\zeta\delta}{h}$.

\begin{theo}\label{theo:smooth_pred_corr_ANE}
    Consider the sequence $\brc{x_{k \mid k-1}}$ generated by FOA-Min with $g_k=\nabla_x f_k$, and
    suppose that Assumptions~\ref{assu:prev_sc}(i)-(ii) and \ref{assu:prediction-correction} hold. Set the stepsize as $\beta = 1/L_1$ and the parameters $\zeta, \delta$ so that $\zeta \geq Z$ and $\delta \leq h$ are satisfied. Then, for all $k_0\in\mathbb{Z}_{\geq 0}$,
    the average of the gradient norm for $T_{k_0}:= \frac{f_{k_0} - f_{k_0}^\ast}{h^2}$ iterations satisfies
$
\frac{1}{T_{k_0}}\sum_{k=k_0}^{k_0 + T_{k_0}-1} \norm*{ \nabla_x f(x_{k \mid k-1};t_k)}
\leq
\sqrt{2L_1(1+\bar{G}_2)}h
$.
\end{theo}

\begin{proof}
For simplicity, we prove the statement when $k_0=0$.
From the update rule of $\{\hat{x}_k^c\}_{c}$, we have
\begin{align}
    \paren*{\beta - \frac{L_1\beta^2}{2}}\norm*{ \nabla_x f(x_{k+1\mid k};t_{k+1})}^2
    &\leq 
    f(x_{k+1\mid k};t_{k+1}) - f(x_{k+1};t_{k+1}) \label{eq:pred_corr_Inequality}
    \\
    &\leq 
    f(x_{k};t_{k}) - f(x_{k+1};t_{k+1}) + \bar{G}_2h^2, \label{eq:pred_corr_Inequality_G2h}
\end{align}
where the last inequality holds due to \eqref{eq:pred_corr_error} in Lemma~\ref{lemm:func_value_error_foa_min_cp}.
Now, set the stepsize as $\beta=1/L_1$, and sum up the above inequality for all iterations $0\leq k \leq T-1$. Then, we can obtain
\begin{align*}
    \frac{1}{2L_1} \sum_{k=1}^{T} \norm*{ \nabla_x f(x_{k \mid k - 1};t_k)}^2
    &\leq
    \paren*{f(x_0;t_0) - f(x_{T};t_{T})} + T\bar{G}_2h^2 \\
    &\leq
    f(x_0;t_0) - f^\ast_T + T\bar{G}_2h^2 \\
    &=
    f(x_0;t_0) - f^\ast_0 + T\bar{G}_2h^2,
\end{align*}
where the last equality holds due to Lemma~\ref{lemm:optimal_value_const_assu4}.
Since $f(x_{0 \mid -1};t_0) - f(x_{0};t_{0}) \geq \frac{1}{2L_1} \| \nabla_x f(x_{0 \mid -1};t_0) \|^2$ follows from the descent lemma, we have
\begin{align*}
    \frac{1}{2L_1} \sum_{k=0}^{T-1} \norm*{ \nabla_x f(x_{k \mid k - 1};t_k)}^2
    &\leq
    \paren*{f(x_{0 \mid -1};t_0) - f(x_{0};t_{0})} + \paren*{f(x_0;t_0) - f_{0}^\ast} + T\bar{G}_2h^2 \\
    &=
    f(x_{0 \mid -1};t_0) - f^\ast_0 + T\bar{G}_2h^2.
\end{align*}
From
the Cauchy-Schwarz inequality, when the number of iterations is $T = \frac{f(x_0;t_0) - f^\ast_0}{h^2}$, we can obtain
\begin{align*}
\frac{1}{T}\sum_{k=0}^{T-1} \norm*{ \nabla_x f(x_k;t_k)} &\leq \sqrt{ \frac{1}{T} \sum_{k=0}^{T-1} \norm*{ \nabla_x f(x_k;t_k)}^2 }\\
&\leq \sqrt{\frac{2L_1(f(x_0;t_0) - f_0^\ast)}{T} + 2L_1\bar{G}_2h^2} \\
&= \sqrt{2L_1(1+\bar{G}_2)}h.
\end{align*}
\end{proof}

\begin{theo}\label{theo:smooth_pred_corr_last}
    Suppose the same settings as in Theorem~\ref{theo:smooth_pred_corr_ANE}.
     Then, once $\sqrt{2L_1(1+\bar{G}_2)}h$-stationary is found by FOA-Min with $g_k = \nabla_x f_k$ at iteration $\bar{T}$,
     every subsequent iterate $x_{k \mid k-1}\ ( k\geq \bar{T})$
     satisfies at least one of the following two conditions:
     \begin{itemize}
         \item[(a)] The iterate $x_{k \mid k-1}$ is a $\sqrt{2L_1(1+\bar{G}_2)}h$-stationary point of $f(x;t_k)$.
         \item[(b)] There exists an integer $l<k$ such that $x_{l \mid l-1}$ is a $\sqrt{2L_1(1+\bar{G}_2)}h$-stationary point of $f(x;t_l)$, and 
         $
             (f(x_{k \mid k-1};t_k) - f_k^\ast) < (f(x_{l \mid l-1};t_{l}) - f_l^\ast) +\bar{G}_2h^2 - \frac{1}{2L_1} \| \nabla_x f(x_{l \mid l-1};t_l) \|^2
         $ holds.
     \end{itemize}
 \end{theo}

 \begin{proof}
   For some $k \geq \bar{T}$, suppose that (a) does not hold, that is, $\| \nabla_x f(x_{k \mid k-1};t_k) \| > \sqrt{2L_1(1+\bar{G}_2)}h$ holds. Let $l<k$ be an integer satisfying
    \begin{align*}
        \| \nabla_x f(x_{l \mid l-1};t_l) \| &\leq \sqrt{2L_1(1+\bar{G}_2)}h, \\ 
        l < \forall j < k,\ \| \nabla_x f(x_{j \mid j-1};t_j) \| &> \sqrt{2L_1(1+\bar{G}_2)}h.
    \end{align*}
Inequalities \eqref{eq:pred_corr_Inequality} and \eqref{eq:pred_corr_Inequality_G2h} with $\beta = 1/L_1$ yield that for any $k \geq 0$, we have
\begin{align*}
   \frac{1}{2L_1}\norm*{ \nabla_x f(x_{k+1 \mid k};t_{k+1})}^2 &\leq f(x_{k+1 \mid k};t_{k+1}) - f(x_{k+1};t_{k+1}), \\
   \frac{1}{2L_1}\norm*{ \nabla_x f(x_{k+1 \mid k};t_{k+1})}^2 &\leq (f(x_{k};t_k) - f(x_{k+1};t_{k+1})) + \bar{G}_2h^2.
\end{align*}
    The above inequalities imply 
    \begin{align*}
        f(x_{l};t_{l}) - f(x_{l \mid l-1};t_{l})
        &\leq - \frac{1}{2L_1} \norm*{ \nabla_x f(x_{l \mid l-1};t_{l})}^2,\\
        l < \forall j < k,\ f(x_{j};t_{j}) - f(x_{j-1};t_{j-1})
        &\leq \frac{1}{2L_1} \paren*{2L_1\bar{G}_2h^2 - \norm*{ \nabla_x f(x_{j \mid j-1};t_{j})}^2} \\
        &< \frac{1}{2L_1} \paren*{2L_1\bar{G}_2h^2 - 2L_1(1+\bar{G}_2)h^2} = -h^2 < 0.
    \end{align*}
By summing up the above inequalities
we can obtain
\begin{align*}
    f(x_{k-1};t_{k-1}) - f(x_{l \mid l-1};t_l) < - \frac{1}{2L_1} \norm*{ \nabla_x f(x_{l \mid l-1};t_{l})}.
\end{align*}
Therefore, (b) follows by using Inequality \eqref{eq:pred_corr_error} and Lemma~\ref{lemm:optimal_value_const_assu4}:
\begin{align*}
    (f(x_{k \mid k-1};t_{k}) - f_{k}^\ast) - (f(x_{l \mid l-1};t_l) - f_{l}^\ast) < \bar{G}_2h^2 - \frac{1}{2L_1} \norm*{ \nabla_x f(x_{l \mid l-1};t_{l})}.
\end{align*}
\end{proof}

\begin{theo} \label{theo:PL_pred_corr_AFE}
Consider the sequence $\brc{x_{k \mid k-1}}$ generated by Algorithm~\ref{algo:foa_min_cp}.
Suppose that the objective function $f$ is $\mu$-PL function in terms of $x$ and that Assumptions~\ref{assu:prev_sc}(i)-(ii) and \ref{assu:prediction-correction} hold.
Set the stepsize as $\beta = 1/L_1$, and let $\rho:=1-\frac{\mu}{L_1}\in[0, 1)$. If the parameters $\zeta, \delta$ satisfy $\zeta \geq Z$ and $\delta \leq h$, we have 
$\forall k \in \mathbb{N},\ f(x_{k \mid k-1};t_{k}) - f^\ast_{k} \leq \rho^{k} (f(x_{0 \mid -1};t_0) - f^\ast_0) + O(h^2),\ \lim_{k\to\infty} (f(x_{k \mid k-1};t_{k}) - f^\ast_{k}) = O(h^2).$
\end{theo}

\begin{proof}
    We can prove similarly to Theorem~\ref{theo:GD_PL_linear_convergence} by using Lemmas~\ref{lemm:optimal_value_const_assu4} and \ref{lemm:func_value_error_foa_min_cp}:
    \begin{align*}
    f(x_{k+1 \mid k};t_{k+1}) - f_{k+1}^\ast
    &\leq 
    f(x_{k};t_{k}) - f_{k}^\ast + O(h^2)\\
    &\leq 
    \rho (f(x_{k\mid k-1};t_{k}) - f_{k}^\ast) + O(h^2) \\
    &\leq 
    \rho^{k+1} (f(x_{0 \mid -1};t_{0}) - f_0^\ast) + O(h^2),\\
    \end{align*}
    which yields $\lim_{k\to\infty} (f(x_{k \mid k-1};t_{k}) - f^\ast_{k}) = O(h^2)$.
\end{proof}


\section{Numerical Experiments}

We conducted three experiments, i.e., optimization of a non-convex toy function, linear regression using synthetic datasets, and matrix factorization using real datasets.
All the experiments are implemented in Python 3.9.7 on a MacBook Pro whose chip is M1 Pro and memory is 16GB.

\subsection{Non-convex Toy Problem} \label{sec:nonconvex_toy}
First, we will show that a prediction method designed for SC functions may suffer from instability for the non-convex objective functions. Subsequently, we will demonstrate it by using a toy problem.

\subsubsection{Instability of Prediction Method Designed for SC Functions} \label{subsec:instability}

A Taylor expansion-based prediction \eqref{eq:existing_inexact_prediction} for SC functions is  
\begin{align*} 
   \hat{x}_k^{p+1} = \hat{x}_k^{p} -\alpha \paren{\nabla_{xx} f_k(\hat{x}_k^{p}-x_k) + h\nabla_{tx} f_k + \gamma\nabla_x f_k} \tag{\ref{eq:existing_inexact_prediction}},
\end{align*}
where $\gamma \in [0, 1]$. It aims to optimize the following quadratic function:
\begin{align} \label{eq:opt_problem_pred_for_sc}
    \min_{x \in \mathbb{R}^d}\ \hat{f}^{2, \gamma}(x;t_{k+1}) := f_k &+ \gamma \agbra*{\nabla_x f_k, x-x_k} + h\nabla_t f_k \nonumber \\
    &+ \frac{1}{2}(x-x_k)^\top \nabla_{xx} f_k(x - x_k) + h\agbra*{\nabla_{tx}f_k, x - x_k} + \frac{h^2}{2}\nabla_{tt}f_k.
\end{align}
An optimal solution to this problem will yield an accurate prediction for the SC objective function. However, if the objective function is non-convex, $\nabla_{xx} f_k$ may be negative definite, which implies the function value $\hat{f}^{2, \gamma}(\cdot;t_{k+1})$ decreases unboundedly when iterates move in the direction of a negative eigenvector. 
Therefore, the prediction method may attempt to move in the direction endlessly and the solution may be unstable or unbounded.
We will be likely to encounter such a problem
especially when $\gamma$ is non-zero since the update length $\hat{x}_k^{p+1} - \hat{x}_k^{p}$ may amount to $\Omega(1)$.

\subsubsection{Problem Settings and Results}

Let us consider the following non-convex objective function with a parallel shift:
\begin{align*}
   f(x;t) := \frac{(x-10t)^2}{20} + \sin(x - 10t).
\end{align*}
We plot the graph of the objective function when $t=0$ in Figure~\ref{fig:poly_sin_graph}. We compared the performances of four algorithms: TVGD (Algorithm~\ref{algo:GD}), U-FOPC 
, FOA-Min (with $g_k = \nabla_x f_k$), and CP (with $g_k = 2\nabla_x f_k - \nabla_x f(x_k;t_{k-1})$). The initial point was chosen as $x_{0 \mid -1} = 8$, where the second derivative in terms of $x$ is negative, and the stepsizes were set to $\alpha = \beta = 1.0 \simeq 1/1.1 = 1/L_1$. We set the parameter $\gamma$ of U-FOPC to $0$ or $1$ and selected $\zeta = 10$ for FOA-Min and CP since $|\nabla_t f(x;t)| = 10\| \nabla_x f(x;t) \|$ holds. The complete parameter settings are described in Table~\ref{table:poly_sin_param_setting}.

\begin{figure}[h] 
  \includegraphics[bb=0.000000 0.000000 455.040910 329.040658,width=0.5\linewidth]{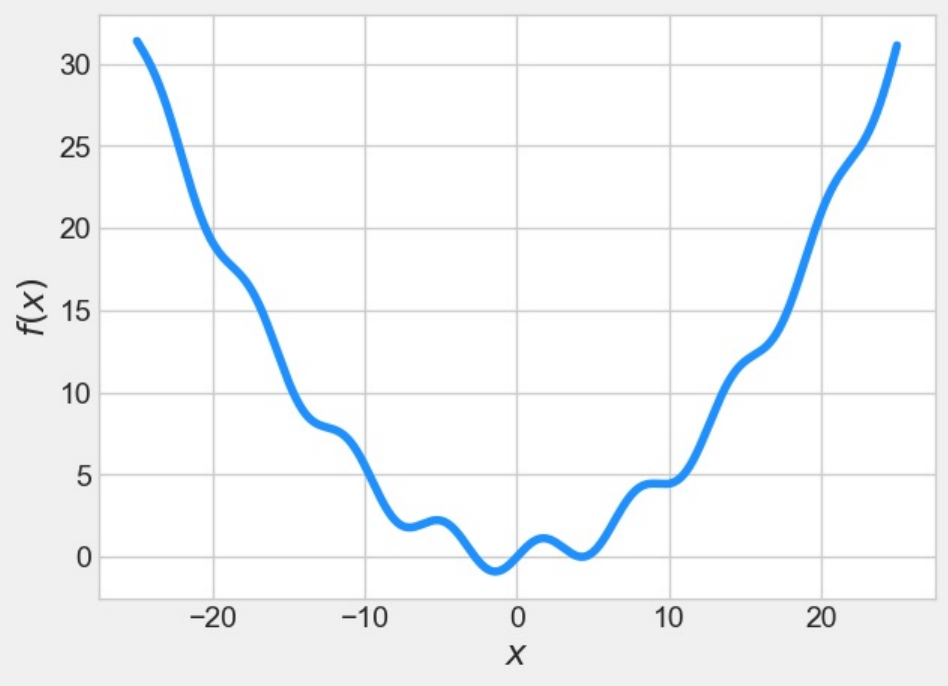}
  \centering
  \caption{Graph of $f(x) = \frac{x^2}{20} + \sin(x)$.}
  \label{fig:poly_sin_graph}
\end{figure}

\begin{table}[h]
\centering
  \caption{Parameter settings in the non-convex toy problem.}
  \begin{tabular}{cccc|ccc|cc}
    \toprule
    \multicolumn{4}{c}{Common} & \multicolumn{3}{|c|}{U-FOPC} &  \multicolumn{2}{c}{Proposed}
    \\ \hline
    $(T, h)$ & $x_{0 \mid -1}$ & $C$ & $\beta$ & $P$  & $ \alpha$ & $\gamma$ & $\zeta$ & $\delta$
    \\ \hline
    $(100, 0.1)$ & $8.0$ & $1$ & $1.0$ & $10$ & $1.0$ & $0$ or $1.0$ & $10$ & $1e^{-10}$ \\ 
    \bottomrule
    \end{tabular}%
\label{table:poly_sin_param_setting}%
\end{table}%

Figure~
\ref{fig:poly_sin_all} shows plots of the iterates $x_{k \mid k-1} - 10 t_k$ and gradient norm $\| \nabla_x f(x_{k \mid k-1};t_{k}) \|$ generated by the algorithms. We subtracted $10 t_k$ from $x_{k \mid k-1}$ to cancel the effect of a parallel shift. U-FOPC with $\gamma = 1$ diverges within a few iterations due to an unbounded prediction, which was suggested in the above observation. U-FOPC with $\gamma=0$ also goes to a point far away from the initial point, leading to a worse objective value than other algorithms. FOA-Min and CP can successfully track the stationary point nearest the initial point in the descent direction, and in particular, CP achieves the lowest gradient norm. This improvement implies that accurate optimization of the Taylor series approximation can yield better tracking accuracy. TVGD (labeled as GD in the figure, and thereafter) accidentally arrives at a point near an optimal solution, and the gradient norm is stable around 1.

%

\begin{figure}[h]
    \centering

\begin{tabular}{cc}
    \begin{minipage}[t]{0.45\linewidth}
    \centering
        \includegraphics[bb=0.000000 0.000000 426.240852 296.640593,width=2.5in]{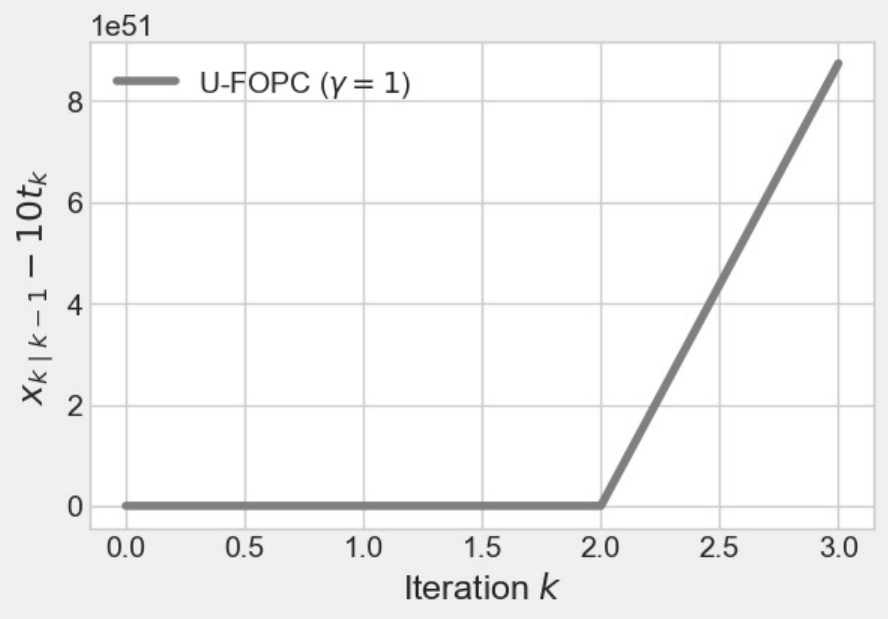}
    \end{minipage}

    \begin{minipage}[t]{0.45\linewidth}
    \centering
        \includegraphics[bb=0.000000 0.000000 426.960854 296.640593,width=2.5in]{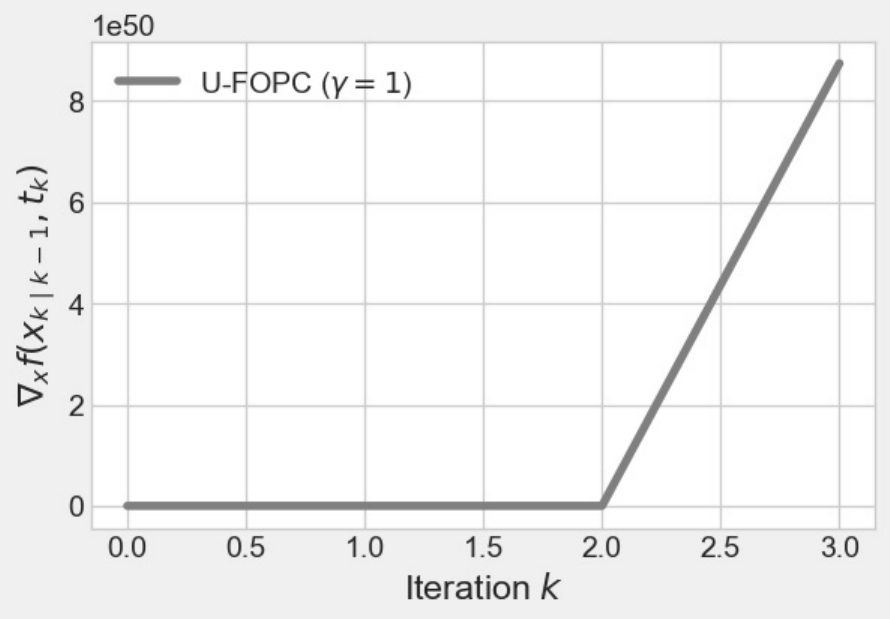}
    \end{minipage}%
\end{tabular}
\centering

\begin{tabular}{cc}
    \begin{minipage}[t]{0.45\linewidth}
    \centering
        \includegraphics[bb=0.000000 0.000000 433.440867 283.680567,width=2.5in]{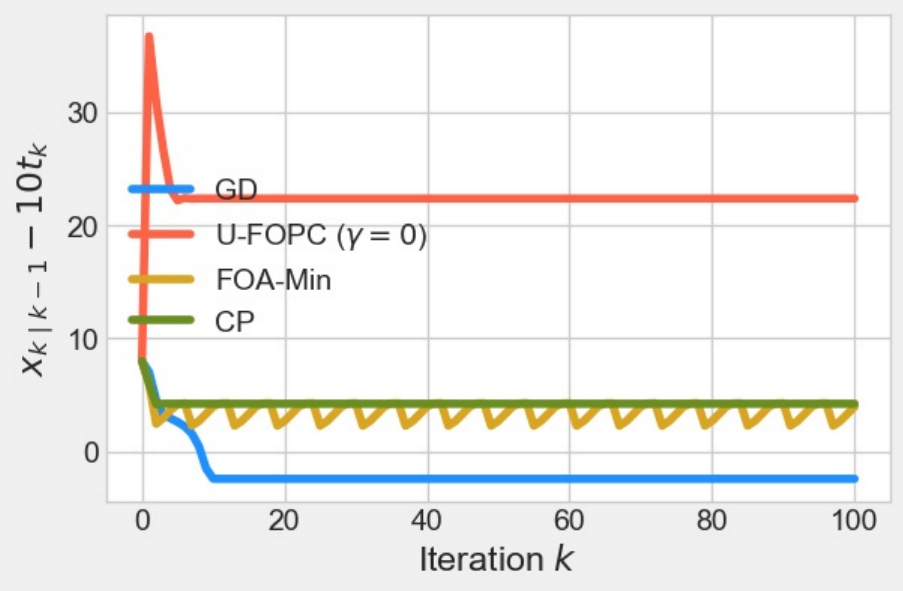}
     \end{minipage}

    \begin{minipage}[t]{0.45\linewidth}
    \centering
        \includegraphics[bb=0.000000 0.000000 426.960854 283.680567,width=2.5in]{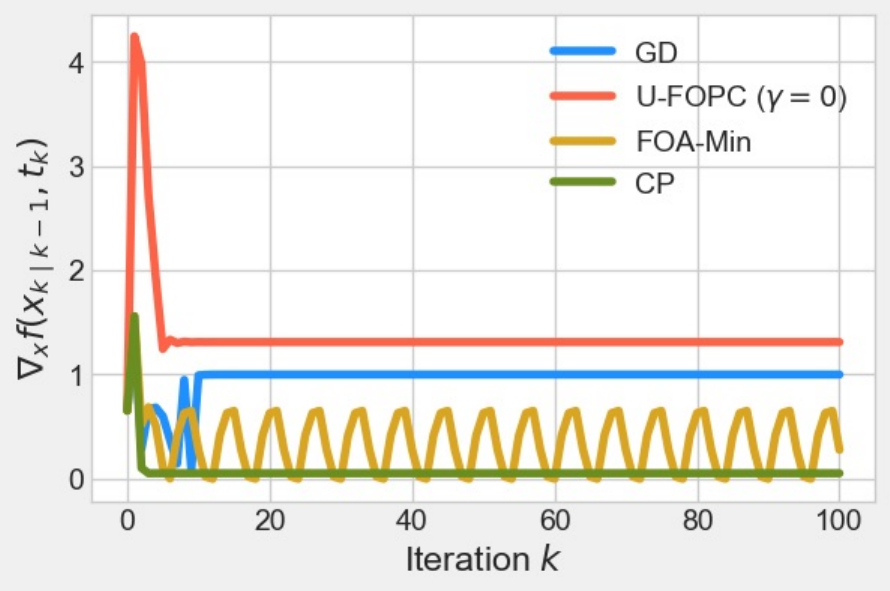}
    \end{minipage}%
\end{tabular}
\centering
\caption[Caption for LOF]{
    Plots of the iterate $x_{k \mid k-1} - 10 t_k$ \footnotemark and gradient norm $\| \nabla_x f(x_{k \mid k-1};t_{k}) \|$ generated by the algorithms.
    }
\label{fig:poly_sin_all}
\end{figure}

\footnotetext{We define $x_{k \mid k-1} := x_k$ for TVGD.}

\subsection{Linear Regression} \label{sec:linear_regression_main}
The objective function is defined as
\begin{align*}
  f(x;t) = \frac{1}{2} \norm*{ \langle A, x \rangle - b(t)}^2,
\end{align*} where $A\in \mathbb{R}^{10 \times 10}$ is a diagonal matrix whose $(i, i)$ entry of $A$ is 0.1 when $i \leq 5$ and is 10 otherwise, and $b(t) \in \mathbb{R}^{10}$ is a time-varying vector whose $i$-th entry is set to
$
  b_i(t) = 10 \sin\paren*{\frac{t}{100} + \frac{2\pi i}{10}}.
$
We also implemented a linear regression with a time-varying matrix $A(t)$; this experiment and its results are described in Appendix~\ref{sec:linear_regression_appendix}.

%

We compared the performances of TVGD, U-FOPC,
FOA-Min (with $g_k = \nabla_x f_k$), and CP (with $g_k = 2\nabla_x f_k - \nabla_x f(x_k;t_{k-1}))$.
Table~\ref{table:lr_invariant_param_setting} summarizes the parameter settings.
The parameter $\gamma$ of U-FOPC was fixed to 0 based on the observations in Section~\ref{sec:nonconvex_toy}. We set the stepsizes to $(\alpha = ) \beta = 1/L_1 = 1/10^2 = 0.01$ for all the algorithms.
The parameter $\zeta$ for FOA-Min and CP was determined by following
\begin{align*}
  \frac{| \nabla_t f(x;t) |}{\| \nabla_x f(x;t) \|} \leq (\| A'(t) \| \| x \| + \|b'(t)\|)\| A^{-1}(t) \| \leq (10\sqrt{5} / 100) \times (1 / 0.1) \leq 2.5 =: \zeta.
\end{align*} 
The number of correction steps $C$ for each algorithm was determined so that the computational time per iteration was the same for all the algorithms. The computational time for the prediction and correction steps of each algorithm is summarized in Table~\ref{table:lr_comp_time}.

\begin{table}[H]
\centering
  \caption{Parameter settings in linear regression with invariant curvature.}
  \begin{tabular}{ccc}
    \toprule
    \multicolumn{2}{c}{Common} 
    \\ \hline
    $(T, h)$ & $x_{0 \mid -1}$
    \\ \hline
    $(2e^{3}, 0.1), (2e^{4}, 0.01), (2e^{5}, 1e^{-3})$ & $x_{0 \mid -1}\sim\mathcal{N}(0, \mathrm{I})$ \\ 
    \bottomrule
    \end{tabular}%

\vspace{2.5mm}

\begin{tabular}{cc|ccccc|cccc|cccc}
    \toprule
    \multicolumn{2}{c}{GD} & \multicolumn{5}{|c|}{U-FOPC} &  \multicolumn{4}{c|}{FOA-Min} &  \multicolumn{4}{c}{CP}
    \\ \hline
    $C$ & $\beta$ & $P$ & $C$ & $\alpha$ & $ \beta$ & $\gamma$ & $C$ & $\beta$ & $\zeta$ & $\delta$ & $C$ & $\beta$ & $\zeta$ & $\delta$ 
    \\ \hline
    $4$ & $0.01$ & $10$ & $1$ & $0.01$ & $0.01$ & $0$ & $3$ & $0.01$ & $2.5$ & $1e^{-10}$ & $1$ & $0.01$ & $2.5$ & $1e^{-10}$\\ 
    \bottomrule
    \end{tabular}%
\label{table:lr_invariant_param_setting}%
\end{table}%

\begin{table}[H]
\centering
  \caption{Computational time required to correct and predict.}
  \begin{tabular}{ccccc}
    \toprule
    & Corr. & Pred. (U-FOPC) & Pred. (FOA-Min) & Pred. (CP) \\ \hline
    Time [s] & $7.3e^{-5}$ & $2.1e^{-4}$ & $6.7e^{-5}$ & $1.7e^{-4}$ \\ 
    \bottomrule
    \end{tabular}%
  \label{table:lr_comp_time}
\end{table}

Figure~\ref{fig:lr_invariant_all} illustrates the optimization results when the sampling period is $h=1e^{-3}$. We can see that FOA-Min achieves a significant improvement in the function value and gradient norm in comparison with the existing methods and that CP also outperforms the existing methods in both evaluation metrics. The accuracy of U-FOPC is worse than TVGD;
this would be because the improvement in the solution by the prediction was smaller than that by the correction within the same time.

\begin{figure}[h]
    \centering
\begin{tabular}{cc}
    \begin{minipage}[t]{0.45\linewidth}
    \centering
        \includegraphics[bb=0.000000 0.000000 452.880906 290.160580,width=2.5in]{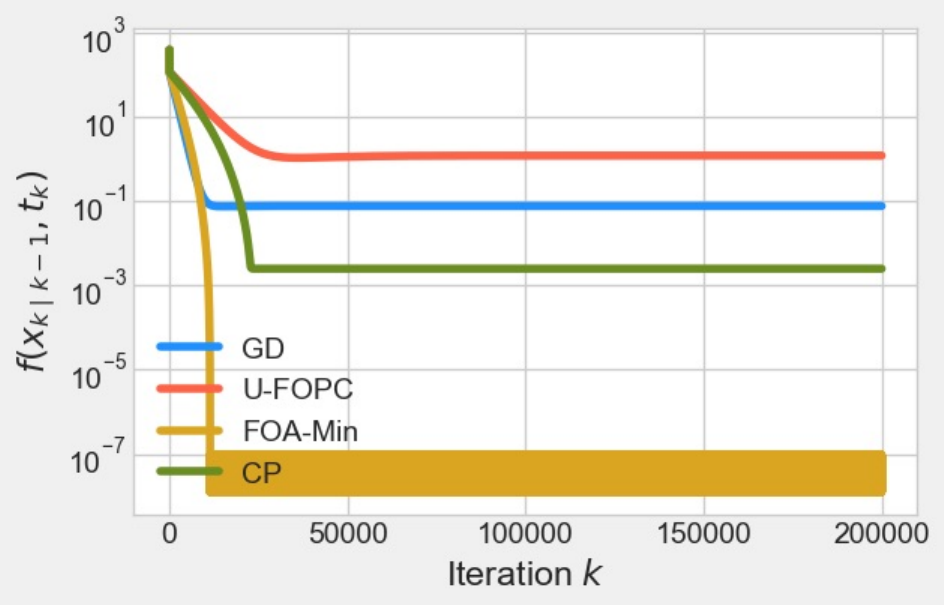}
    \end{minipage}

    \begin{minipage}[t]{0.45\linewidth}
    \centering
        \includegraphics[bb=0.000000 0.000000 452.880906 283.680567,width=2.5in]{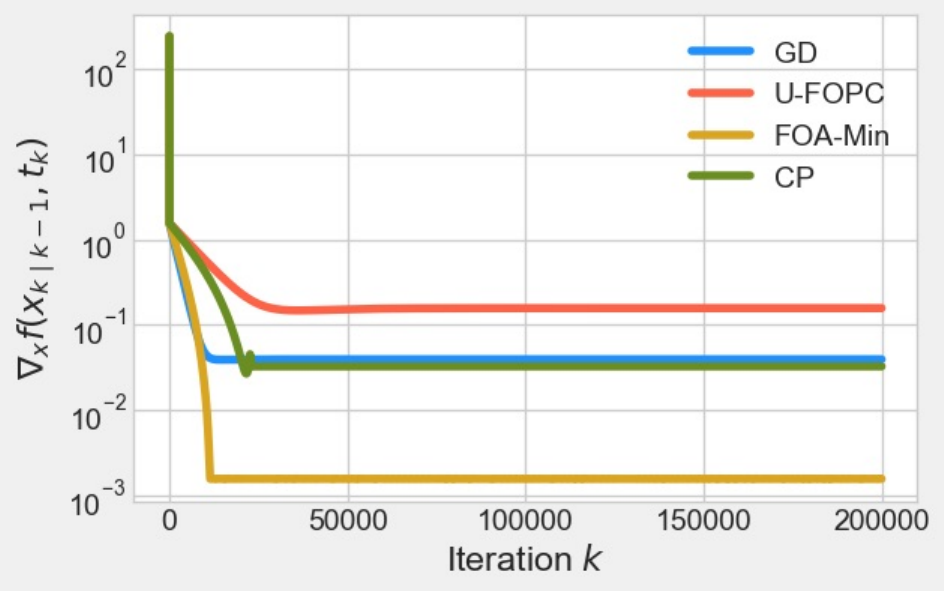}
    \end{minipage}%
\end{tabular}
\centering
    \caption{Log plots of function value and gradient norm when $h = 1e^{-3}$.}
    \label{fig:lr_invariant_all}
\end{figure}

We also ran experiments with different sampling periods to check the change in the accuracy of the solutions.
The top two plots of Figure~\ref{fig:lr_invariant_different_h} indicate that FOA-Min and CP achieve an $O(h)$ gradient norm, while the sampling period dependencies of the existing methods seem to be worse than $O(h)$.
The sampling period dependencies of the optimality gap of the proposed algorithms are also better than those of the existing algorithms (see the bottom plots of Figure~\ref{fig:lr_invariant_different_h}).

\begin{figure}[h]
    \centering

\begin{tabular}{cc}
    \begin{minipage}[t]{0.45\linewidth}
    \centering
        \includegraphics[bb=0.000000 0.000000 439.920880 295.200590,width=2.5in]{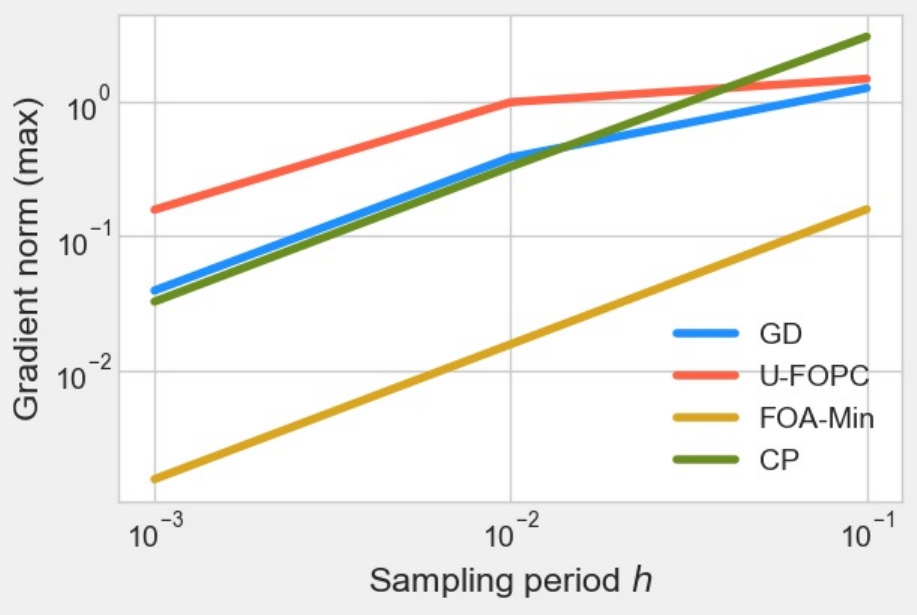}
    \end{minipage}

    \begin{minipage}[t]{0.45\linewidth}
    \centering
        \includegraphics[bb=0.000000 0.000000 439.920880 295.200590,width=2.5in]{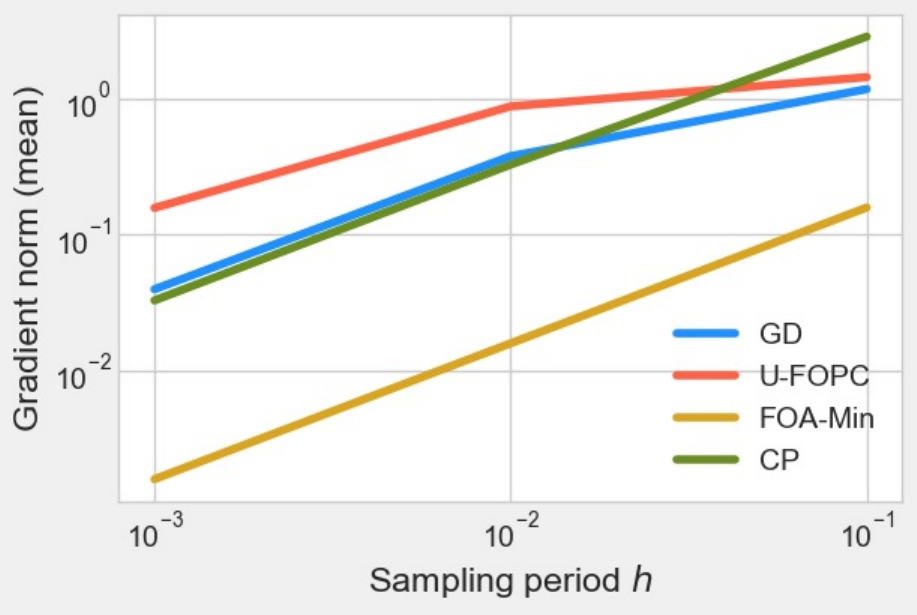}
    \end{minipage}%
\end{tabular}

\begin{tabular}{cc}
    \begin{minipage}[t]{0.45\linewidth}
    \centering
        \includegraphics[bb=0.000000 0.000000 439.920880 295.200590,width=2.5in]{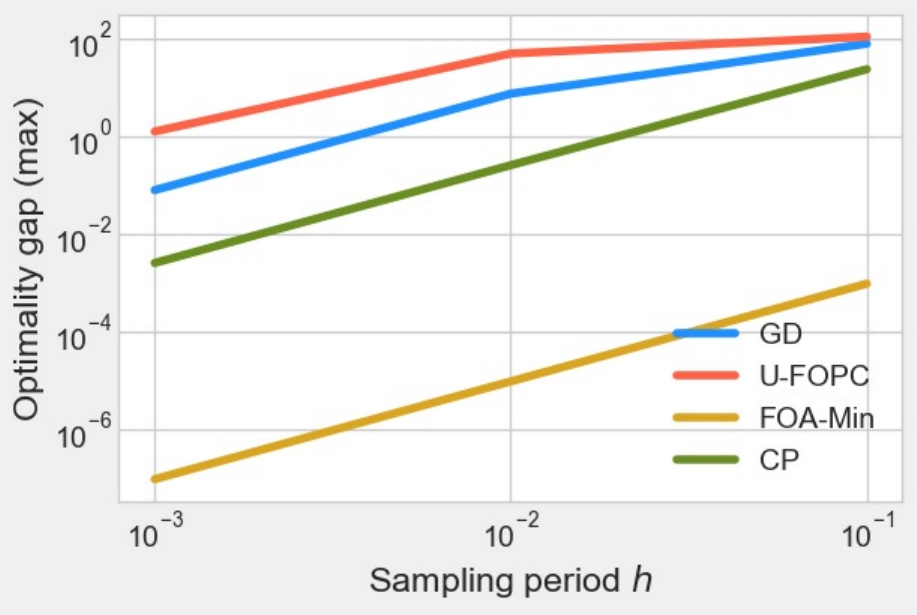}
    \end{minipage}

    \begin{minipage}[t]{0.45\linewidth}
    \centering
        \includegraphics[bb=0.000000 0.000000 439.920880 295.200590,width=2.5in]{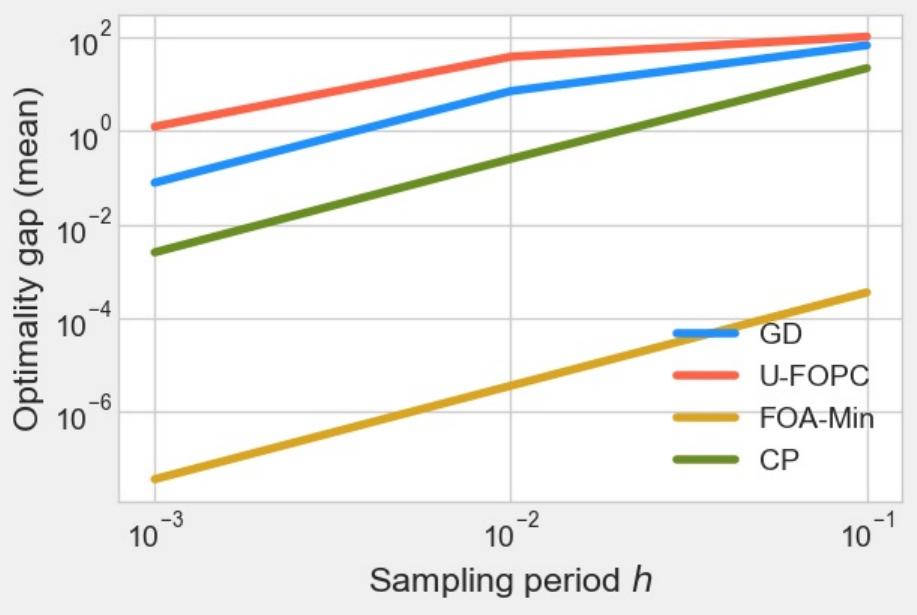}
    \end{minipage}%
\end{tabular}

\centering
    \caption{(Top) Log-log plots of maximum and mean of the gradient norm $ \| \nabla_x f(x_{k \mid k-1};t_k) \|$ versus sampling period. (Bottom) Log-log plots of maximum and mean of the optimality gap $f(x_{k \mid k-1};t_k) - f_k^\ast$ versus sampling period.
    Maximum and mean are computed based on the results of the last half of the iterations.}
    \label{fig:lr_invariant_different_h}
\end{figure}

\subsection{Non-convex Robust Regression} \label{sec:linear_regression}
The objective function is defined as
\begin{align*}
  f(x;t) := \sum_{i=1}^n \ell( (\langle A(t), x \rangle - b(t))_i ).
\end{align*}
Here, $A(t)$ and $b(t)$ represent a time-varying matrix and vector, respectively, and $(\cdot)_i$ represents $i$-th entry.
For the loss function $\ell$, we used the following two robust loss functions:
\begin{align*}
    \ell_1(y) &:= \frac{2y^2}{y^2 + 4}, \\
    \ell_2(y) &:= 1 - \exp{\paren*{-\frac{y^2}{2}}}.
\end{align*}
The functions $\ell_1$ and $\ell_2$ are referred to as the Geman-McClure loss function~\citep{geman1985bayesian} and Welsch loss function~\citep{dennis1978techniques}, respectively. They are non-convex as shown in Figure~\ref{fig:nonconvex_robust}, which implies that the objective function $f(\cdot;t)$ is also non-convex.

\begin{figure}[h]
    \centering
    \begin{tabular}{cc}
        \begin{minipage}[t]{0.45\linewidth}
        \centering
        \includegraphics[bb=0.000000 0.000000 447.840896 166.320333,width=2.5in]{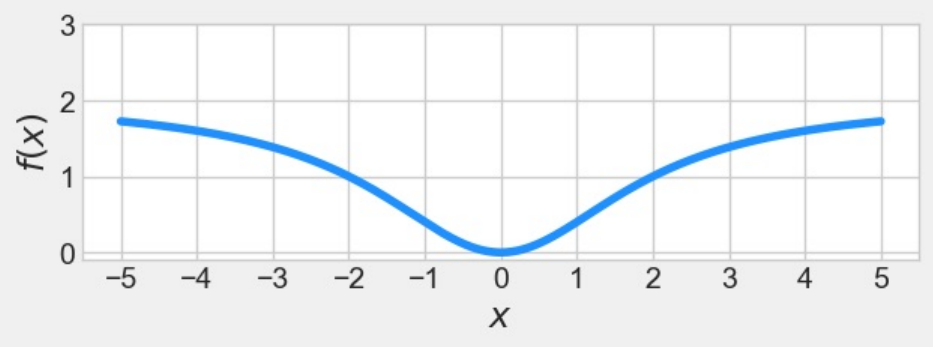}
        \end{minipage}
    
        \begin{minipage}[t]{0.45\linewidth}
        \centering
        \includegraphics[bb=0.000000 0.000000 447.840896 164.880330,width=2.5in]{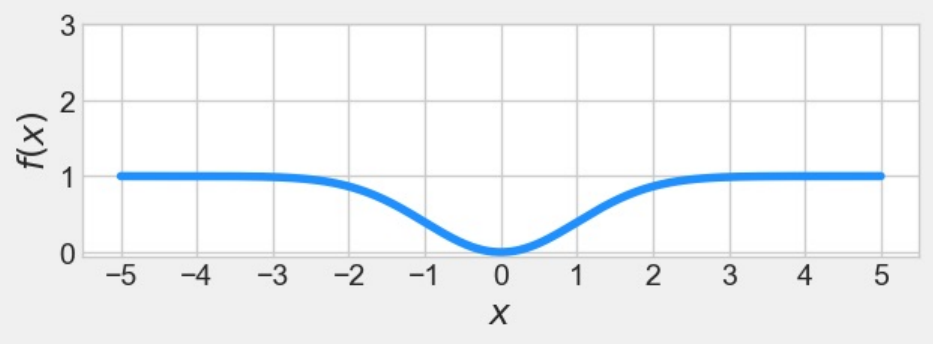}
        \end{minipage}%
    \end{tabular}
    \centering
    \caption{(Left) Graph of Geman-McClure loss function. (Right) Graph of Welsch loss function.}
    \label{fig:nonconvex_robust}
\end{figure}

The $(i,j)$ entry of $A(t)\in \mathbb{R}^{10 \times 10}$ and $i$-th entry of $b(t)\in \mathbb{R}^{10}$ were set to
\begin{align*} 
  A_{ij}(t) &=
  \left\{
  \begin{array}{cc}
   1 + 0.05 \cos\paren*{\frac{t}{200} + \frac{2\pi i}{10}},  & i = j \leq 5\\
   10 (1 + 0.05 \cos\paren*{\frac{t}{200} + \frac{2\pi i}{10}}), & i = j > 5\\
   0, & i \neq j
  \end{array}
  \right., \\
  b_i(t) &= 50 \sin\paren*{\frac{t}{100} + \frac{2\pi i}{10}},\ 1 \leq i \leq 10.
\end{align*}

The parameter settings for the two loss functions are the same, and they are summarized in Table~\ref{table:nr_param_setting}.
The parameter $\gamma$ of U-FOPC was fixed to 0 based on the observations in Section~\ref{sec:nonconvex_toy}. We set the stepsizes to $(\alpha = ) \beta = 0.01 \simeq 1/L_1 = 1/(10.5)^2$ for all the algorithms.
The parameter $\zeta$ for FOA-Min and CP was determined by assuming $\| x \| \leq 100$ and following
\begin{align*}
  \frac{| \nabla_t f(x;t) |}{\| \nabla_x f(x;t) \|} \leq (\| A'(t) \| \| x \| + \|b'(t)\|)\| A^{-1}(t) \| \leq \frac{(0.5 / 200) \times 100  + 50\sqrt{5} / 100}{0.95} \leq 1.5 =: \zeta.
\end{align*}
The number of correction steps $C$ for each algorithm was determined so that the computational time per iteration was the same for all the algorithms. The computational time for the prediction and correction steps of each algorithm is summarized in Table~\ref{table:nr_comp_time}.

\begin{table}[h]
\centering
  \caption{Parameter settings in non-convex regression.}
  \begin{tabular}{ccc}
    \toprule
    \multicolumn{2}{c}{Common} 
    \\ \hline
    $(T, h)$ & $x_{0 \mid -1}$
    \\ \hline
    $(2e^{4}, 0.05), (1e^{5}, 0.01), (1e^{6}, 1e^{-3})$ & $x_{0 \mid -1}\sim\mathcal{N}(0, \mathrm{I})$ \\ 
    \bottomrule
    \end{tabular}%

\vspace{2.5mm}

\begin{tabular}{cc|ccccc|cccc|cccc}
    \toprule
    \multicolumn{2}{c}{TVGD} & \multicolumn{5}{|c|}{U-FOPC} &  \multicolumn{4}{c|}{FOA-Min} &  \multicolumn{4}{c}{CP}
    \\ \hline
    $C$ & $\beta$ & $P$ & $C$ & $\alpha$ & $ \beta$ & $\gamma$ & $C$ & $\beta$ & $\zeta$ & $\delta$ & $C$ & $\beta$ & $\zeta$ & $\delta$ 
    \\ \hline
    $4$ & $0.01$ & $10$ & $1$ & $0.01$ & $0.01$ & $0$ & $3$ & $0.01$ & $1.5$ & $1e^{-10}$ & $1$ & $0.01$ & $2.5$ & $1e^{-10}$\\ 
    \bottomrule
    \end{tabular}%
\label{table:nr_param_setting}%
\end{table}%

\begin{table}[h]
\centering
  \caption{Computational time required to correct and predict.}
  \begin{tabular}{ccccc}
    \toprule
    & Corr. & Pred. (U-FOPC) & Pred. (FOA-Min) & Pred. (CP) \\ \hline
    Time [s] & $6.8e^{-5}$ & $2.7e^{-4}$ & $7.3e^{-5}$ & $2.5e^{-4}$ \\ 
    \bottomrule
    \end{tabular}%
  \label{table:nr_comp_time}
\end{table}

Figure~\ref{fig:rational_1e-3} and Figure~\ref{fig:exp_1e-3} show the optimization results for two losses when $h=1e^{-3}$.
FOA-Min outperforms the existing algorithms regarding the function value and the gradient norm. 
The performance of CP is worse than that of FOA-Min and about the same as that of TVGD, 
because CP takes much time to predict and the number of corrections per iteration is a quarter of that of TVGD.

We also ran experiments with different sampling periods to check the change in the accuracy of the solutions.
The results are shown in Figure~\ref{fig:rational_different_h} and Figure~\ref{fig:exp_different_h}.
The figures indicate that two proposed methods, including CP, achieve $O(h)$ optimality gap and gradient norm in terms of both maximum and mean.


\begin{figure}[h]
    \centering
\begin{tabular}{cc}
    \begin{minipage}[t]{0.45\linewidth}
    \centering
        \includegraphics[bb=0.000000 0.000000 447.120894 283.680567,width=2.5in]{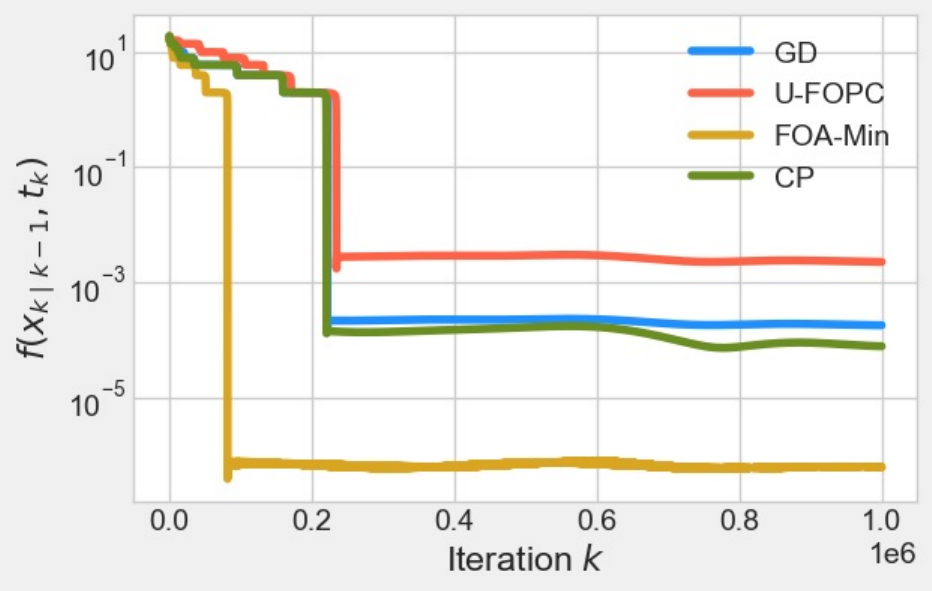}
    \end{minipage}

    \begin{minipage}[t]{0.45\linewidth}
    \centering
        \includegraphics[bb=0.000000 0.000000 447.120894 283.680567,width=2.5in]{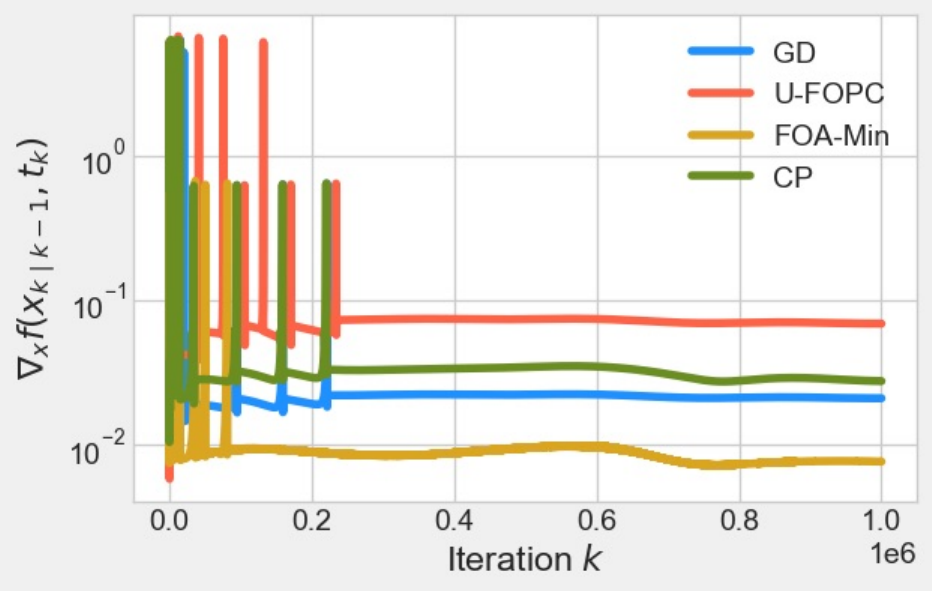}
    \end{minipage}%
\end{tabular}
\centering
    \caption{Log plots of function value and gradient norm for Geman-McClure loss function with the sampling period of $h = 1e^{-3}$.}
    \label{fig:rational_1e-3}
\end{figure}

\begin{figure}[h]
    \centering
\begin{tabular}{cc}
    \begin{minipage}[t]{0.45\linewidth}
    \centering
        \includegraphics[bb=0.000000 0.000000 447.120894 283.680567,width=2.5in]{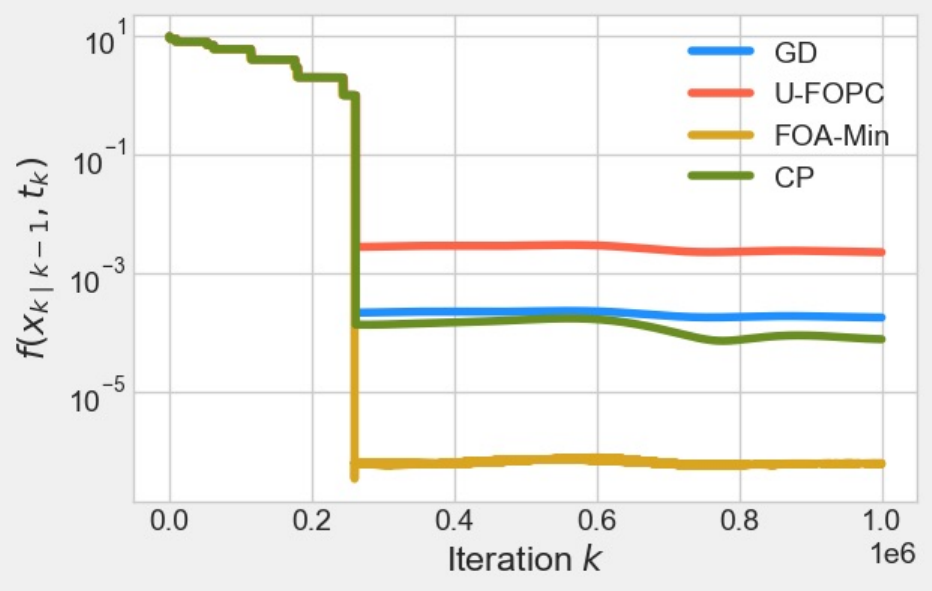}
    \end{minipage}

    \begin{minipage}[t]{0.45\linewidth}
    \centering
        \includegraphics[bb=0.000000 0.000000 452.160904 283.680567,width=2.5in]{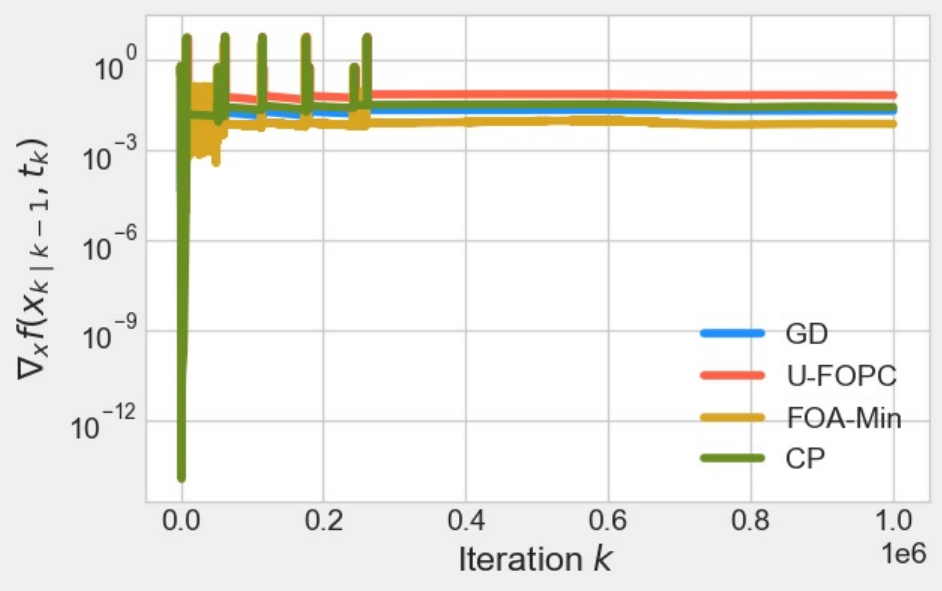}
    \end{minipage}%
\end{tabular}
\centering
    \caption{Log plots of function value and gradient norm for Welsch loss function with the sampling period of $h = 1e^{-3}$.}
    \label{fig:exp_1e-3}
\end{figure}

\begin{figure}[h]
    \centering
\begin{tabular}{cc}
    \begin{minipage}[t]{0.45\linewidth}
    \centering
        \includegraphics[bb=0.000000 0.000000 439.920880 295.200590,width=2.5in]{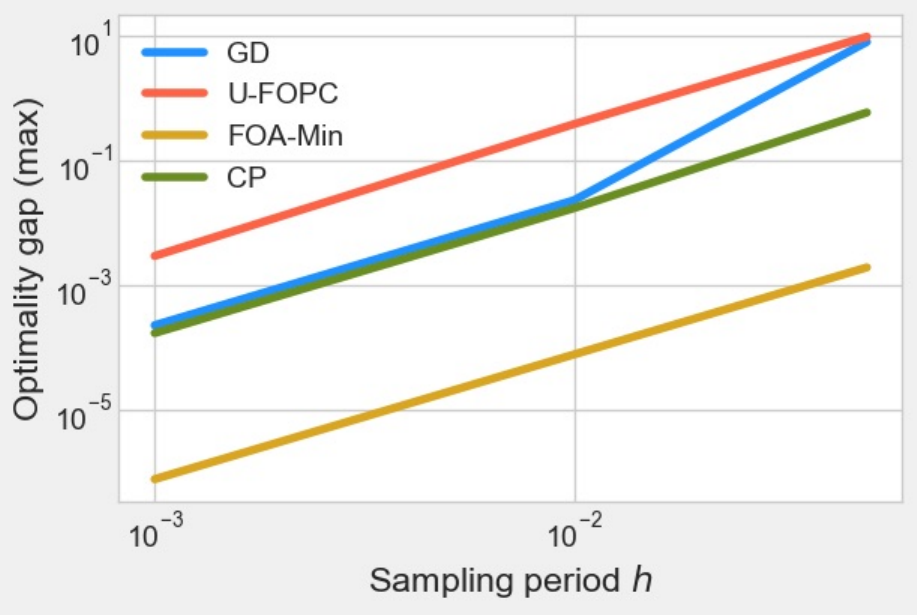}
    \end{minipage}

    \begin{minipage}[t]{0.45\linewidth}
    \centering
        \includegraphics[bb=0.000000 0.000000 439.920880 295.200590,width=2.5in]{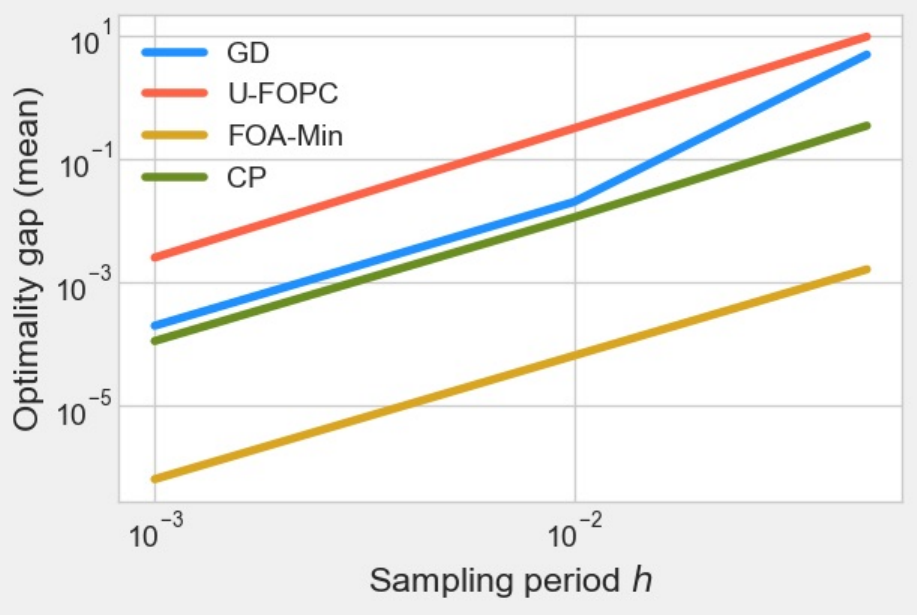}
    \end{minipage}%
\end{tabular}

\begin{tabular}{cc}
    \begin{minipage}[t]{0.45\linewidth}
    \centering
        \includegraphics[bb=0.000000 0.000000 439.920880 295.200590,width=2.5in]{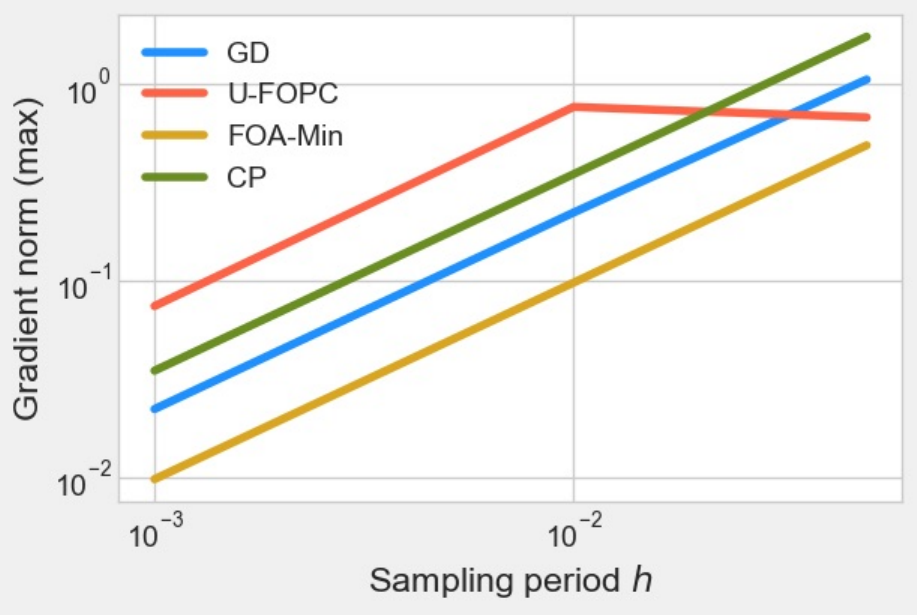}
    \end{minipage}

    \begin{minipage}[t]{0.45\linewidth}
    \centering
        \includegraphics[bb=0.000000 0.000000 439.920880 295.200590,width=2.5in]{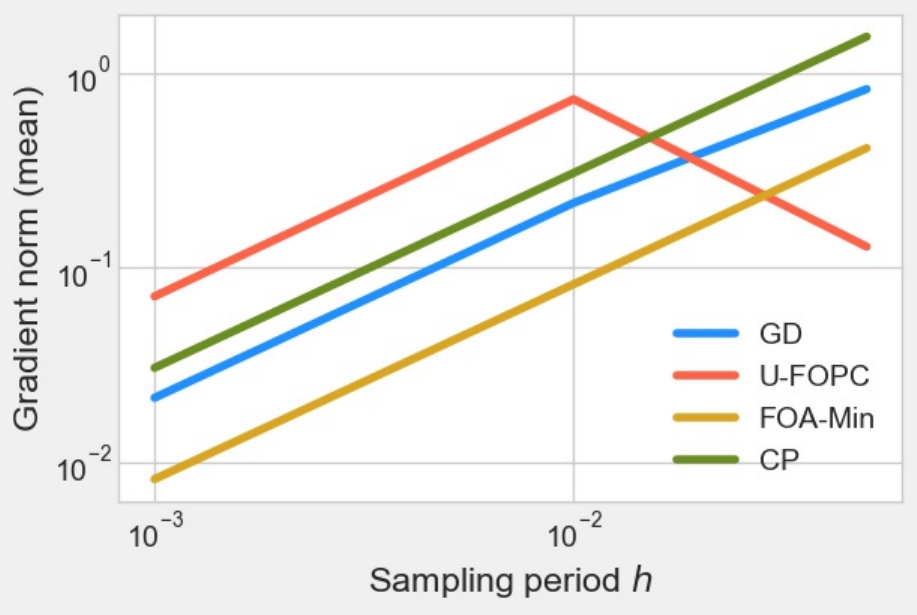}
    \end{minipage}%
\end{tabular}

\centering
    \caption{For Geman-McClure loss function, (Top) Log-log plots of maximum and mean of the gradient norm $ \| \nabla_x f(x_{k \mid k-1};t_k) \|$ versus sampling period. (Bottom) Log-log plots of maximum and mean of the optimality gap $f(x_{k \mid k-1};t_k) - f_k^\ast$ versus sampling period.
    Maximum and mean are computed based on the results of the last half of the iterations.}
    \label{fig:rational_different_h}
\end{figure}

\begin{figure}[h]
    \centering
\begin{tabular}{cc}
    \begin{minipage}[t]{0.45\linewidth}
    \centering
        \includegraphics[bb=0.000000 0.000000 439.920880 303.120606,width=2.5in]{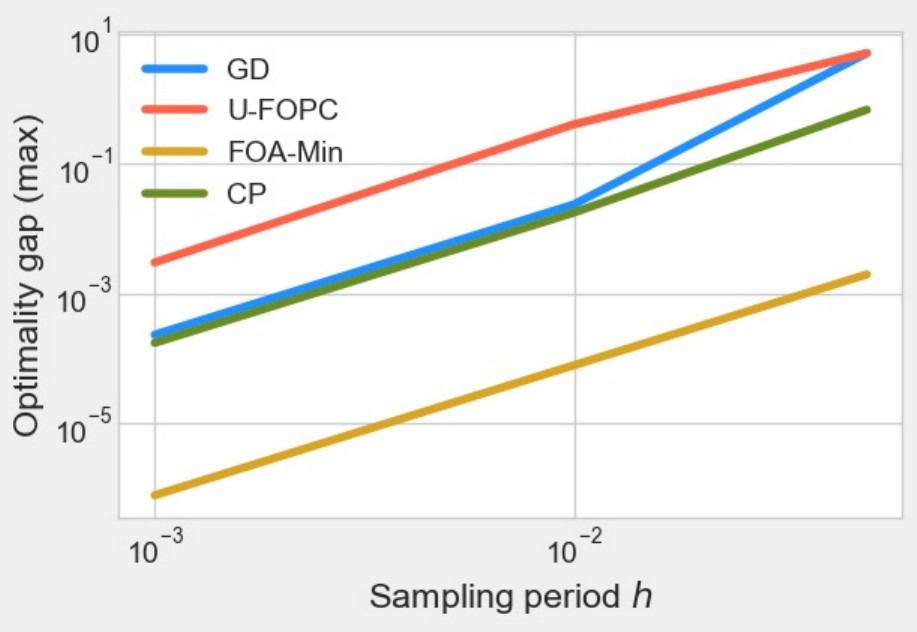}
    \end{minipage}

    \begin{minipage}[t]{0.45\linewidth}
    \centering
        \includegraphics[bb=0.000000 0.000000 439.920880 303.120606,width=2.5in]{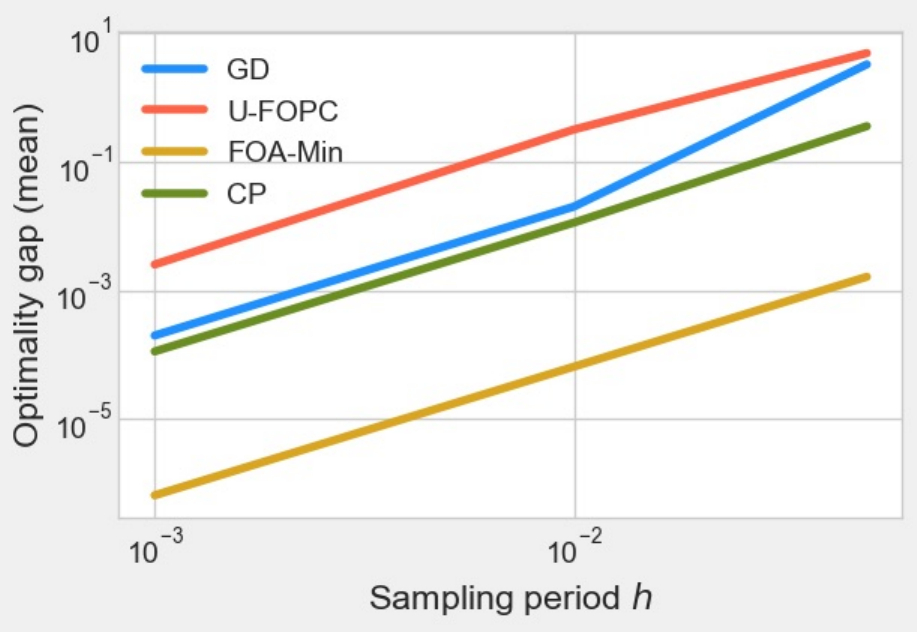}
    \end{minipage}%
\end{tabular}

\begin{tabular}{cc}
    \begin{minipage}[t]{0.45\linewidth}
    \centering
        \includegraphics[bb=0.000000 0.000000 439.920880 295.200590,width=2.5in]{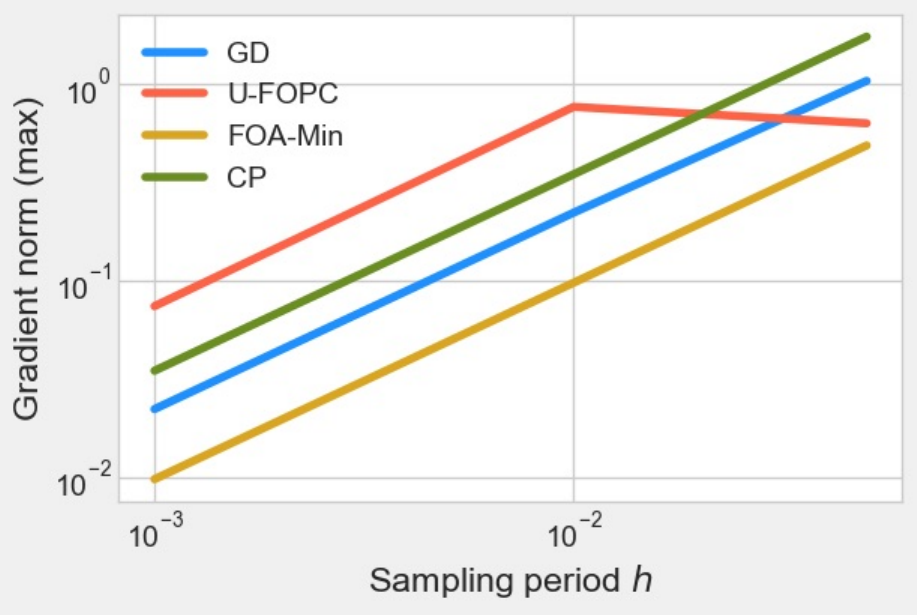}
    \end{minipage}

    \begin{minipage}[t]{0.45\linewidth}
    \centering
        \includegraphics[bb=0.000000 0.000000 439.920880 295.200590,width=2.5in]{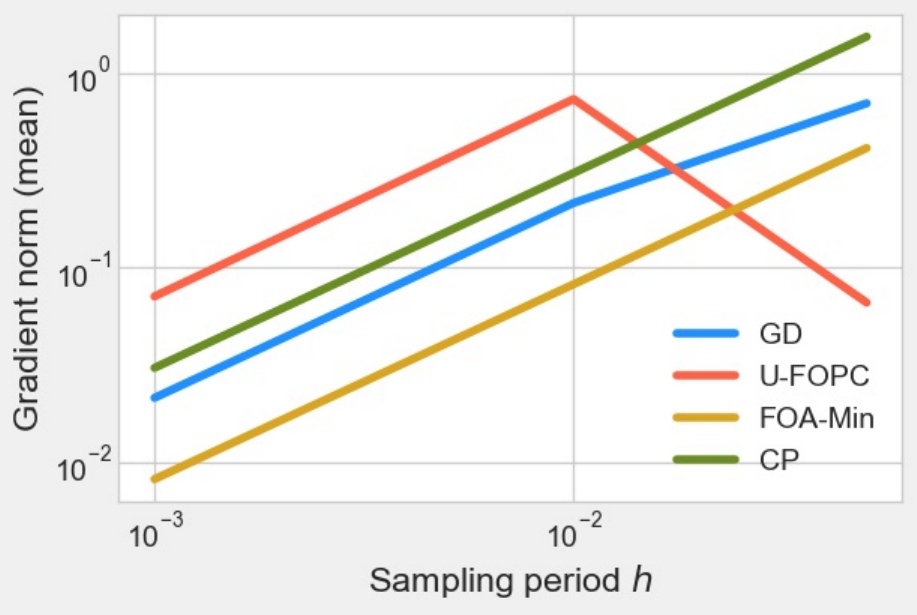}
    \end{minipage}%
\end{tabular}

\centering
    \caption{For Welsch loss function, (Top) Log-log plots of maximum and mean of the gradient norm $ \| \nabla_x f(x_{k \mid k-1};t_k) \|$ versus sampling period. (Bottom) Log-log plots of maximum and mean of the optimality gap $f(x_{k \mid k-1};t_k) - f_k^\ast$ versus sampling period.
    Maximum and mean are computed based on the results of the last half of the iterations.}
    \label{fig:exp_different_h}
\end{figure}

\subsection{Matrix Factorization} \label{sec:mf_main}
Matrix factorization is a non-convex large-scale problem; here, we consider the following time-varying objective function:
\begin{align*}
\min _{
        P, Q 
    }
    \frac{1}{|K(t)|}
     \sum_{(u, i) \in K(t)}
     & \left( R_{u i}-P_u^T Q_i \right)^2 +\lambda \left(\left\|P_u\right\|^2+\left\|Q_i\right\|^2\right).
\end{align*}  
where $P \in \mathbb{R}^{F \times U}$ and $Q \in \mathbb{R}^{F \times I}$ are the matrices whose $u$-th column $P_u$ and $i$-th column $Q_i$ represent properties of the user $u$ and item $i$, respectively.
Let $R\in \mathbb{R}^{U \times I}$ be a matrix whose $(u, i)$ entry $R_{ui}$ denotes the user $u$'s rating of the item $i$, and $\lambda$ is a regularization parameter. The time-varying set $K(t)$ consists of index pairs $(u, i)$ for which ratings $R_{ui}$ are known, and its size $|K(t)|$ is increased by $N$ at each iteration as new $N$ rating data are revealed.


In the numerical experiment, the Netflix Prize Dataset~\citep{bennett2007netflix} was used.
We deleted users and items whose number of ratings was smaller than thresholds from part of the dataset so that the total number of ratings was $443371$.
After arranging these rating data in chronological order, we split them into two equal parts: Datasets 1 and 2. Dataset 1 included ratings from 1999/12 to 2004/7, and its size of $R$ was $811 \times 711$. Dataset 2 included ratings from 2004/7 to 2005/12, and its size of $R$ was $1094 \times 774$. We conducted experiments on each dataset. When using Dataset 1, we revealed new $N=5, 10, 15$ rating data per time step and ran the algorithms until all the data were revealed.
When using Dataset 2, we set the number of ratings revealed per time step to three times the value used in Dataset 1,
since the latter period had three times the number of ratings per unit of time. The initial set $K(0)$ consisted of the first 100000 ratings in chronological order.
U-FOPC and CP could not be implemented since the problem was so large in scale that it took a significant computational time to calculate the Hessian for each iteration; thus, we only compared TVGD
and FOA-Min.
We chose two types of initial value: points whose gradient norms were $0.1$ or $1e^{-4}$. These points were obtained by solving the invariant matrix factorization problem with $K(t)=K(0)$ by applying GD. The parameters $\beta$ and $\zeta$ were both set to $10$, which was the best value among $\{1, 2, 5, 10, 100\}$. 
The number of correction steps $C$ was determined from the computational time of the correction and prediction shown in Table~\ref{table:netflix_comp_time}. 
See Table~\ref{table:netflix_param_setting} for the settings of other parameters.

\begin{table}[H]
\centering
  \caption{Computational time required to correct and predict.}
  \begin{tabular}{ccc}
    \toprule
    & Corr. & Pred. (FOA-Min) \\ \hline
    Time [s] & $7.3e^{-5}$ & $6.7e^{-5}$ \\ 
    \bottomrule
    \end{tabular}%
  \label{table:netflix_comp_time}
\end{table}

\begin{table}[H]
  \centering
    \caption{Parameter settings in matrix factorization.}
    \begin{tabular}{ccccc|c|ccc}
      \toprule
      \multicolumn{5}{c}{Common} & \multicolumn{1}{|c|}{TVGD} & \multicolumn{3}{c}{FOA-Min}
      \\ \hline
      $F$ & $\lambda$ & $h$ & $x_{0 \mid -1}$ & $\beta$ & $C$ & $C$ & $\zeta$ & $\delta$
      \\ \hline
      $20$ & $0.01$ & $0.01$ & $\| \nabla_x f(x_{0 \mid -1};t_0) \| = 0.1\ \mathrm{or}\ 1e^{-4}$ & $10$ & $2$ & $1$ & $10$ & $1e^{-10}$ \\ 
      \bottomrule
      \end{tabular}%
  \label{table:netflix_param_setting}%
\end{table}%

Figure~\ref{fig:netflix_head_10} shows the optimization results when we used Dataset 1, and the number of data revealed per iteration was $N=10$.
Here, FOA-Min achieves higher accuracy than that of TVGD in terms of both the function value and gradient norm.
We can see that regardless of the gradient norm at the initial point, TVGD and FOA-Min converge to points with similar gradient norms. The maximum gradient norms over the last 2000 iterations were more than $3.5e^{-3}$ for TVGD and less than $1.7e^{-3}$ for FOA-Min.
Similar results were obtained when we changed the dataset and the number of data revealed per iteration (see Appendix~\ref{sec:mf_appendix}).
These experimental results imply that the proposed prediction method, which uses a stepsize normalized by the gradient norm, would be effective for tracking stationary points of the time-varying matrix factorization problems. 

\begin{figure}[h]
    \centering
\begin{tabular}{cc}
    \subfigure[Function value (initial gradient norm $=0.1$)]{
    \begin{minipage}[t]{0.45\linewidth}
    \centering
        \includegraphics[bb=0.000000 0.000000 467.280935 283.680567,width=2.5in]{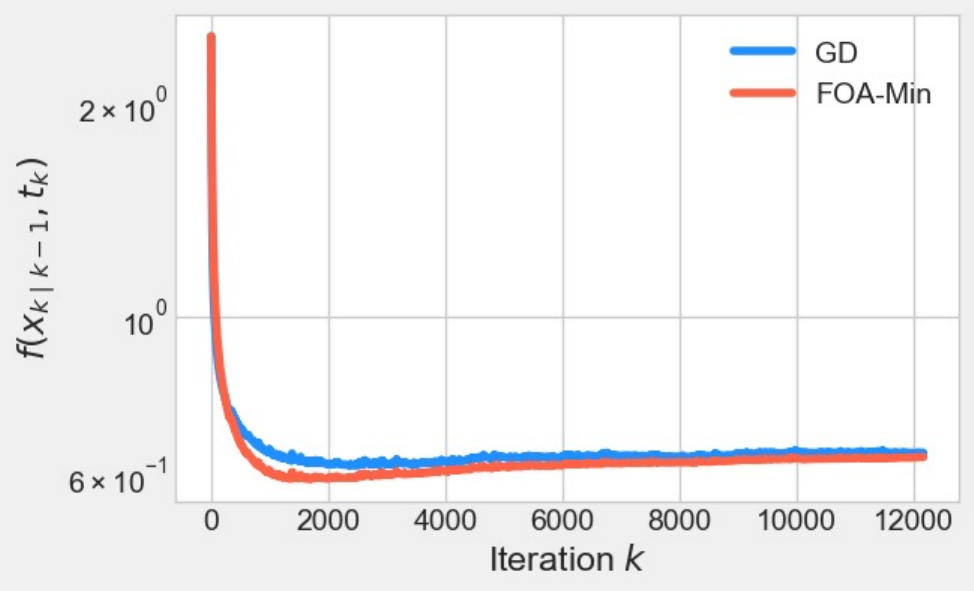}
    \end{minipage}
    }

    \subfigure[Gradient norm (initial gradient norm $=0.1$)]{
    \begin{minipage}[t]{0.45\linewidth}
    \centering
        \includegraphics[bb=0.000000 0.000000 447.120894 283.680567,width=2.5in]{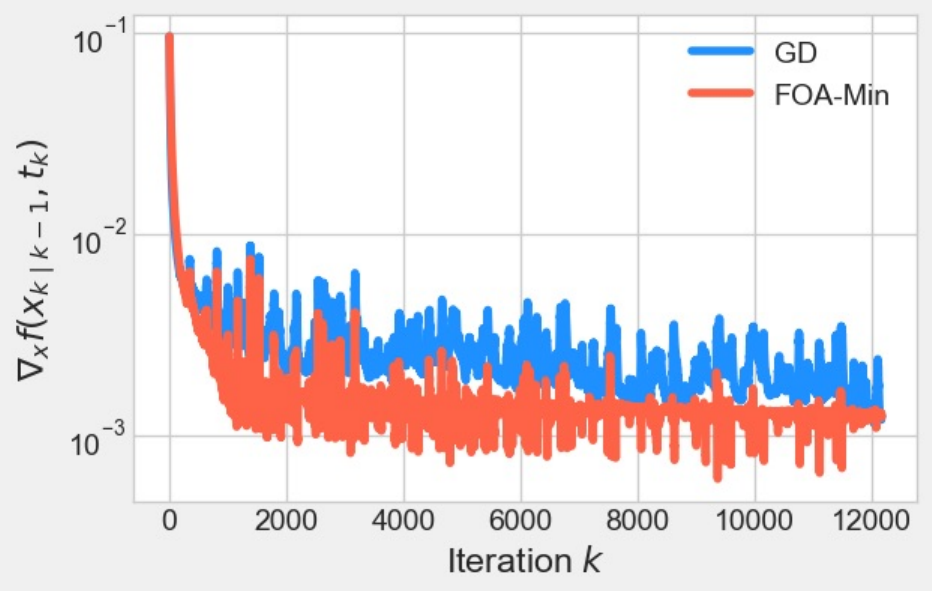}
    \end{minipage}%
    }
\end{tabular}

\begin{tabular}{cc}
    \subfigure[Function value (initial gradient norm $= 1e^{-4}$)]{
    \begin{minipage}[t]{0.45\linewidth}
    \centering
        \includegraphics[bb=0.000000 0.000000 446.400893 283.680567,width=2.5in]{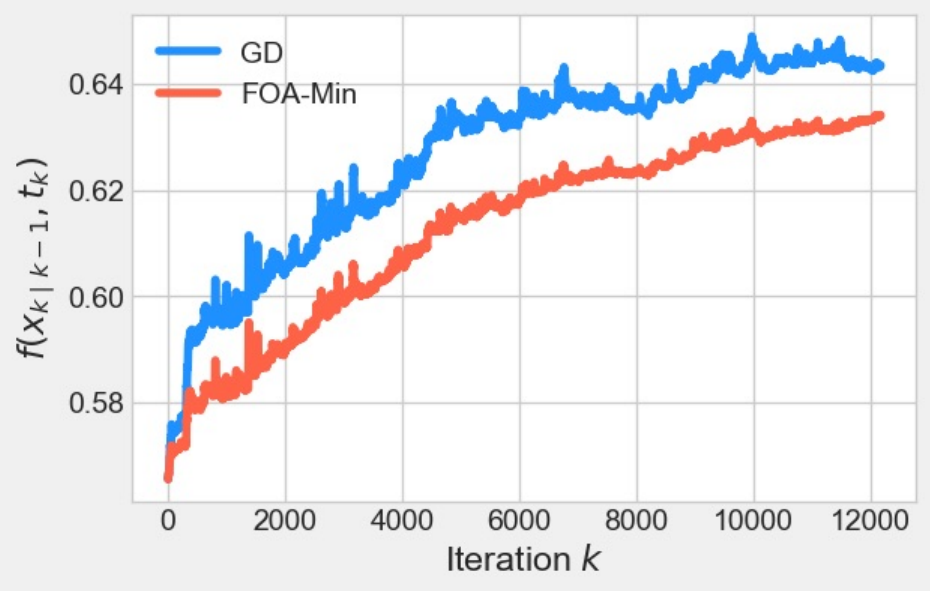}
    \end{minipage}
    }

    \subfigure[Gradient norm (initial gradient norm $=1e^{-4}$)]{
    \begin{minipage}[t]{0.45\linewidth}
    \centering
        \includegraphics[bb=0.000000 0.000000 447.120894 290.160580,width=2.5in]{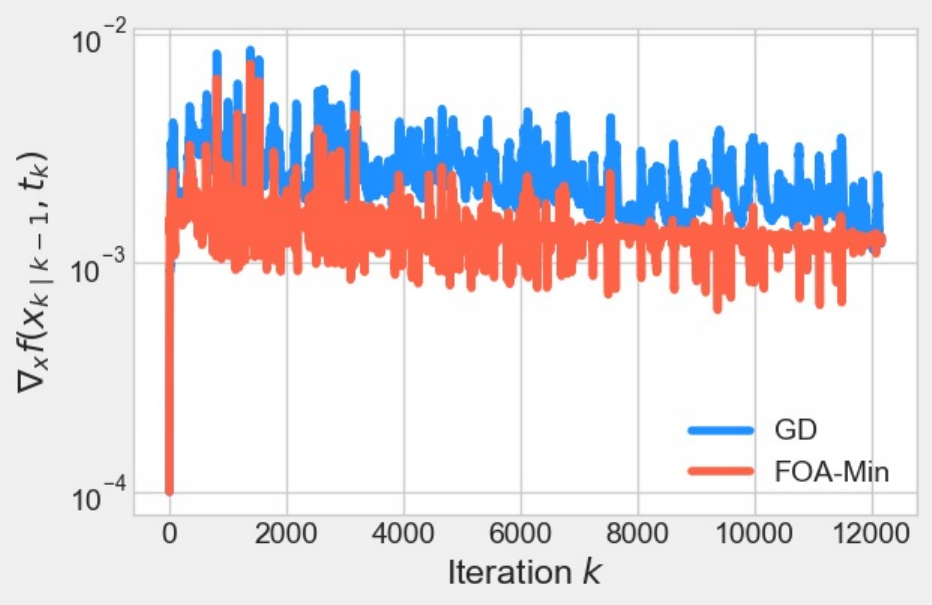}
    \end{minipage}%
    }
\end{tabular}

\centering
    \caption{Plots of function value and gradient norm for Dataset 1, and number of data revealed per iteration $N$ set to 10.}
    \label{fig:netflix_head_10}
\end{figure}

\section{Summary \& Future Directions}
We proposed a new prediction-correction algorithm for time-varying smooth non-convex optimization that is more accurate than the existing algorithms and that is applicable to large-scale problems.
We also performed convergence analyses on the TVGD and proposed algorithms for non-SC objective functions.
Numerical experiments demonstrated that the proposed methods are able to track an $O(h)$-stationary point and outperform existing methods on problems with synthetic and real datasets.

The development of prediction-correction algorithms for non-SC and large-scale optimization problems is important since its range of applications is not restricted to time-varying optimization problems (see Appendix~\ref{chap:applications} for more details).
Further study will be needed to extend the problem settings to constrained optimization, stochastic optimization, and optimization with fewer assumptions on the smoothness and derivatives.  

\appendix

\section{Applications} \label{chap:applications}

The range of applications of prediction-correction algorithms is not restricted to time-varying optimization problems. Here, we introduce promising application fields with related work on methods using the prediction-correction scheme.

Throughout this paper, we considered the following optimization problem
\begin{align*}
    \min_{x \in \mathbb{R}^d} f(x;t_k),\ k\in\mathbb{Z}_{\geq 0}. \tag{\ref{eq:time_varying_problem}}
\end{align*}
If we treat $t_k$ as a general parameter, not necessarily a time parameter, then, 
Problem~\eqref{eq:time_varying_problem} (or its continuous version) can be regarded as a general parametric optimization problem. 
(Continuous version of) Problem~\eqref{eq:time_varying_problem} can be regarded as a special case of a parametric optimization problem 
\begin{align*}
    \min_{x\in\mathbb{R}^d} f(x;\lambda)
\end{align*} where $\lambda\in[\lambda_0, \lambda_T] \subseteq \mathbb{R}$ denotes a parameter.
This problem has been solved in various fields including parametric programming, continuation methods, and interior point methods in constrained optimization.
A method to track the optimal solution path $x^\ast(t)$
using the prediction-correction scheme similar to those provided in this paper is referred to as a {\it predictor-corrector method}. It has been widely studied in these fields from the past to the present~\citep{dontchev2013euler, allgower1993continuation, potra2000interior}.

In parametric programming in machine learning, various predictor-corrector algorithms have been proposed and applied to estimate the optimal solution path of the nonlinear regularized optimization~\citep{park2007l1, si2022path, wang2007kernel, krishnamurthy2011pathwise}.
\citet{park2007l1} proposed Newton's method-based prediction for the $\ell_1$ regularized problem with adaptive stepsize for the regularization parameter $\lambda$. \citet{si2022path} considered an algorithm using Newton's method in both the prediction and correction for a possibly nonconvex optimization problem including the $\ell_p$ regularized optimization. However, these algorithms cannot apply to large-scale or general non-convex optimization problems due to using the Hessian inverse.
Prediction-correction algorithms applicable to such problems will improve the tracking accuracy of the solution path or reduce the total computational time when the accuracy to be obtained is fixed.

Moreover, in recent work, \citet{hazan2016graduated, iwakiri2022single} apply continuation or smoothing methods to non-convex optimization problems to obtain an optimal or a good solution regardless of the choice of the initial point. It might be interesting to combine prediction-correction algorithms with them to solve non-convex optimization problems efficiently.

\section{Theorems for PL Functions with Lipschitzness of Gradient in Terms of $t$} \label{chap:PL_lip_grad_t}

\begin{theo} \label{theo:PL_lip_t_grad}
Consider the sequence $\left\{x_k\right\}$ generated by Algorithm~\ref{algo:GD}. Suppose that the objective function $f$ is $\mu$-PL function in terms of $x$, and Assumptions~\ref{assu:prev_sc}(i)-(ii) hold. Set the stepsize as $\beta = 1/L_1$.
Then, when $2\mu > L_1$ holds, Algorithm~\ref{algo:GD} can find an $\frac{1 + \sqrt{2(L_1/\mu - 1)}}{2 - L_1/\mu}L_2h$-stationary point in
$
\frac{2(2\mu - L_1)(f(x_0;t_0) - f_0^\ast)}{L_2^2 h^2}
$
iterations.
\end{theo}

\begin{proof}
    Let us define
\begin{align*}
    r := \frac{1 + \sqrt{2(L_1/\mu - 1)}}{2 - L_1/\mu}L_2h > 0.
\end{align*}
Now, we assume that $\| \nabla_x f(x_{k+1};t_{k+1}) \| > r $ holds. Then, we can see that $\| \nabla_x f(x_{k+1};t_{k+1}) \| > L_2h$ holds and
the optimality gap can be bounded as 
\begin{align}
    &\paren*{1-\frac{\mu}{L_1}}[(f(x_k;t_k) - f_{k}^\ast) - (f(x_{k+1};t_{k+1}) - f_{k+1}^\ast)] \nonumber\\
    & \geq (f(x_{k+1};t_k) - f_k^\ast) - \paren*{1-\frac{\mu}{L_1}}(f(x_{k+1};t_{k+1}) - f_{k+1}^\ast) \nonumber\\
    &\geq \frac{1}{2L_1} \| \nabla_x f(x_{k+1};t_{k}) \|^2 - \frac{L_1 - \mu}{2L_1\mu} \| \nabla_x f(x_{k+1};t_{k+1}) \|^2 \nonumber\\
    &\geq \frac{1}{2L_1} (\| \nabla_x f(x_{k+1};t_{k+1}) \| - L_2h) ^2 - \frac{L_1 - \mu}{2L_1\mu} \| \nabla_x f(x_{k+1};t_{k+1}) \|^2  \nonumber\\
    &= \frac{2 - L_1/\mu}{2L_1}\left[ \paren*{\norm{\nabla_x f(x_{k+1};t_{k+1})} - \frac{L_2h}{2 - L_1/\mu}}^2 - \paren*{\frac{L_2h\sqrt{L_1/\mu - 1}}{2 - L_1/\mu}}^2 \right] \nonumber\\
    &> \frac{2 - L_1/\mu}{2L_1}\left[ 2\paren*{\frac{L_2h\sqrt{L_1/\mu - 1}}{2 - L_1/\mu}}^2 - \paren*{\frac{L_2h\sqrt{L_1/\mu - 1}}{2 - L_1/\mu}}^2 \right] \nonumber\\
    &= \frac{(L_1 / \mu - 1) L_2^2}{2L_1(2 - L_1 / \mu)}h^2,\nonumber
\end{align}
where the second inequality holds due to Assumption~\ref{assu:prev_sc}(i) and the definition of a PL function, the third inequality holds due to Assumption~\ref{assu:prev_sc}(ii) and $\| \nabla_x f(x_{k+1};t_{k+1}) \| > L_2h$, and the last inequality holds from $\norm{\nabla_x f(x_{k+1};t_{k+1})} > r > \frac{L_2h}{2 - L_1 \ \mu}$. By dividing both sides by $1 - \frac{\mu}{L_1}$, we can get
\begin{align} \label{eq:GD_PL_Inequality}
    (f(x_k;t_k) - f_{k}^\ast) - (f(x_{k+1};t_{k+1}) - f_{k+1}^\ast) > \frac{L_2^2}{2(2\mu - L_1)}h^2 > 0.
\end{align} 
This inequality implies that the optimality gap $\{f(x_{k};t_{k}) - f_k^\ast\}$ decreases $\frac{L_2^2}{2(2\mu - L_1)}h^2$ if $x_{k+1}$ is not an $r$-stationary point.
Therefore, the number of iterations to find an $r$-stationary point is at most
\begin{align*}
    T = \frac{2(2\mu-L_1)}{L_2^2}\frac{f(x_0;t_0) - f_0^\ast}{h^2}.
\end{align*}
\end{proof}

As in the non-convex case, we can also guarantee that every iterate satisfies a desirable property after finding an $\frac{1 + \sqrt{2(L_1/\mu - 1)}}{2 - L_1/\mu}L_2h$-stationary point.

\begin{theo} \label{theo:PL_without_prediction_last}
Consider the same setting as Theorem~\ref{theo:PL_lip_t_grad}, and define $r:=\frac{1 + \sqrt{2(L_1/\mu - 1)}}{2 - L_1/\mu}L_2h$.
Then, once an $r$-stationary point is reached at iteration $\bar{T}$, every subsequent iterate 
$x_k\ (k \geq \bar{T})$
satisfies at least one of the following two conditions:
    \begin{itemize}
        \item[(a)] The iterate $x_k$ is an $r$-stationary point of $f(x;t_k)$:
        \item[(b)] There exists an integer $l<k$ such that $x_l$ is an $r$-stationary point of $f(x;t_l)$, and 
        \begin{align*}
            f(x;t_k) - f_k^\ast < f(x_l;t_{l}) - f_l^\ast
        \end{align*} holds.
    \end{itemize}
\end{theo}

\begin{proof}
For some $k \geq \bar{T}$, suppose that (a) does not hold, that is, $\| \nabla_x f(x_k;t_k) \| > r$ holds. Let $l<k$ be an integer satisfying
    \begin{align}
        \| \nabla_x f(x_l;t_l) \| &\leq r, \nonumber \\ 
        l \leq \forall j < k,\ \| \nabla_x f(x_{j+1};t_{j+1}) \| &> r. \label{eq:non_rh_stationary}
    \end{align}
    Inequality~\eqref{eq:GD_PL_Inequality}, which can be used due to \eqref{eq:non_rh_stationary}, yields
    \begin{align*}
        l \leq \forall j < k,\ &(f(x_j;t_j) - f_{j}^\ast) - (f(x_{j+1};t_{j+1}) - f_{j+1}^\ast) > 0.
    \end{align*}
By summing up the above inequalities, we can obtain
\begin{align*}
     (f(x_l;t_l) - f_{l}^\ast) - (f(x_{k};t_{k}) - f_{k}^\ast) > 0,
\end{align*}
which implies that (b) holds.
\end{proof}

\section{Technical Lemma} \label{chap:proofs_for_section4}

\begin{lemm} \label{lemm:ratio_between_inner_products}
     Define $a,x \in \mathbb{R}^n$ and $A \in \mathbb{R}^{m \times n}$. When $A$ is full row rank, and its singular values $\sigma_1 \leq \ldots \leq \sigma_m$ are positive, we have
     \begin{align*}
          \frac{|\langle a, x \rangle|}{\| \langle A, x \rangle \|} \leq \frac{\| a \|}{\sigma_1}.
     \end{align*}
\end{lemm}

\begin{proof}
Without loss of generality, we define the singular value decomposition of $A$ as
\begin{align*}
     A = U
     \left[
     \begin{array}{cc}
          \begin{matrix}
          \sigma_1 & & \\
          & \ddots & \\
          & & \sigma_m
          \end{matrix} &
          \mbox{\Large 0}
     \end{array} 
     \right]
     V,
\end{align*}
where $U \in \mathbb{R}^{m \times m}$ and $V \in \mathbb{R}^{n \times n}$ are orthogonal matrices. We also define another matrix $A'$ as
\begin{align*}
     A' = V^{-1}
     \left[
     \begin{array}{c}
          \begin{matrix}
          \sigma_1^{-1} & & \\
          & \ddots & \\
          & & \sigma_m^{-1}
          \end{matrix}
          \\
          \mbox{\Large 0}
     \end{array} 
     \right]
     U^{-1}.
\end{align*}
Then, we have
\begin{align*}
     AA' = U
     \left[
     \begin{array}{cc}
          \begin{matrix}
          \sigma_1 & & \\
          & \ddots & \\
          & & \sigma_m
          \end{matrix} &
          \mbox{\Large 0}
     \end{array} 
     \right]
     \left[
     \begin{array}{c}
          \begin{matrix}
          \sigma_1^{-1} & & \\
          & \ddots & \\
          & & \sigma_m^{-1}
          \end{matrix}
          \\
          \mbox{\Large 0}
     \end{array} 
     \right]
     U^{-1} = \mathrm{I}_m,
\end{align*}
where $\mathrm{I}_m \in \mathbb{R}^{m \times m}$ is the identity matrix. Therefore, we can obtain
\begin{align*}
     \frac{|\langle a, x \rangle|}{\| \langle A, x \rangle \|} &= \frac{| a^\top A'^\top A^\top x |}{\| \langle A, x \rangle \|} \\
     &\leq \frac{\| \agbra*{A'^\top, a}\| \| \langle A, x \rangle \|}{\| \langle A, x \rangle \|} \\
     &\leq \| A' \| \| a \| = \frac{\| a \|}{\sigma_1}.
\end{align*}
\end{proof}

%
%
%
%
%
%

\section{Numerical Experiments} \label{chap:numerical_experiments}

\subsection{Linear Regression with Time-Varying Curvature} \label{sec:linear_regression_appendix}
%

We changed the curvature of the objective function with time. The $(i,j)$ entry of the matrix $A$ was set to
\begin{align*} 
  A_{ij}(t) =
  \left\{
  \begin{array}{cc}
   0.1 (1 + 0.05 \cos\paren*{\frac{t}{200} + \frac{2\pi i}{10}}),  & i = j \leq 5\\
   10 (1 + 0.05 \cos\paren*{\frac{t}{200} + \frac{2\pi i}{10}}), & i = j > 5\\
   0, & i \neq j
  \end{array}
  \right..
\end{align*}

The time-varying vector $b(t)$ was set as in Subsection~\ref{sec:linear_regression_main}, that is, 
\begin{align*}
  b_i(t) = 10 \sin\paren*{\frac{t}{100} + \frac{2\pi i}{10}},\ 1 \leq i \leq 10.
\end{align*}
The stepsizes were set to $(\alpha =) \beta = 0.01 \simeq 1/L_1 = 1/(10.5)^2$ for algorithms. 
The parameter $\zeta$ for FOA-Min and CP was determined by assuming $\| x \| \leq 10$ and following
\begin{align*}
  \frac{| \nabla_t f(x;t) |}{\| \nabla_x f(x;t) \|} \leq (\| A'(t) \| \| x \| + \|b'(t)\|)\| A^{-1}(t) \| \leq \frac{(0.5 / 200) \times 10  + 10\sqrt{5} / 100}{0.095} \leq 3.0 =: \zeta.
\end{align*}
Other parameter settings were the same as those described in Table~\ref{table:lr_invariant_param_setting}.

Figure~\ref{fig:lr_all} shows the optimization results when $h=1e^{-3}$. FOA-Min and CP are able to decrease the function value more than the existing algorithms as well as when the curvature is invariant. On the other hand, the proposed algorithms do not outperform the existing algorithms in terms of the gradient norm. However, we can still make sure that the proposed algorithm achieves the gradient error of $O(h)$ as Figure~\ref{fig:lr_different_h} shows.

\begin{figure}[H]
    \centering
\begin{tabular}{cc}
    \begin{minipage}[t]{0.45\linewidth}
    \centering
        \includegraphics[width=2.2in]{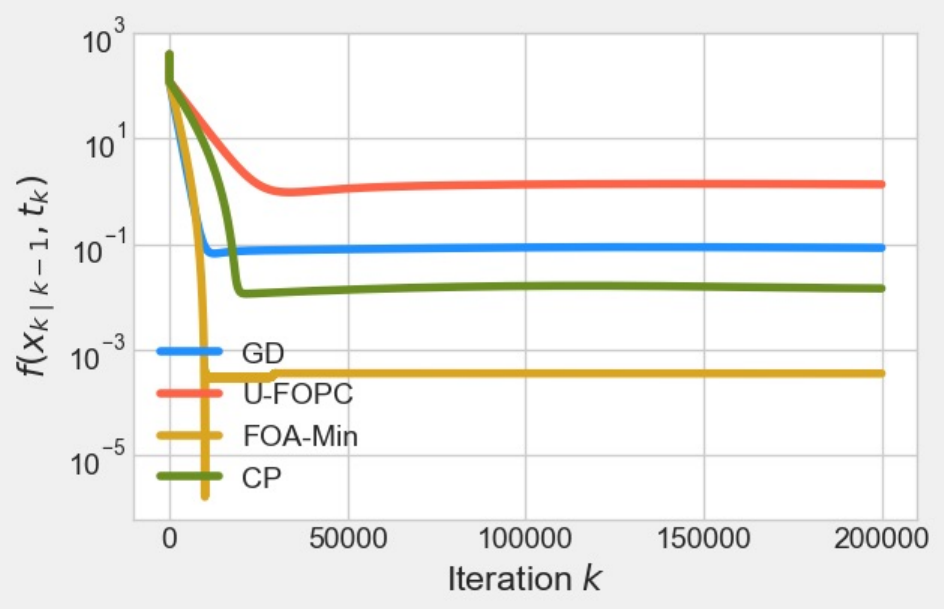}
    \end{minipage}

    \begin{minipage}[t]{0.45\linewidth}
    \centering
        \includegraphics[width=2.2in]{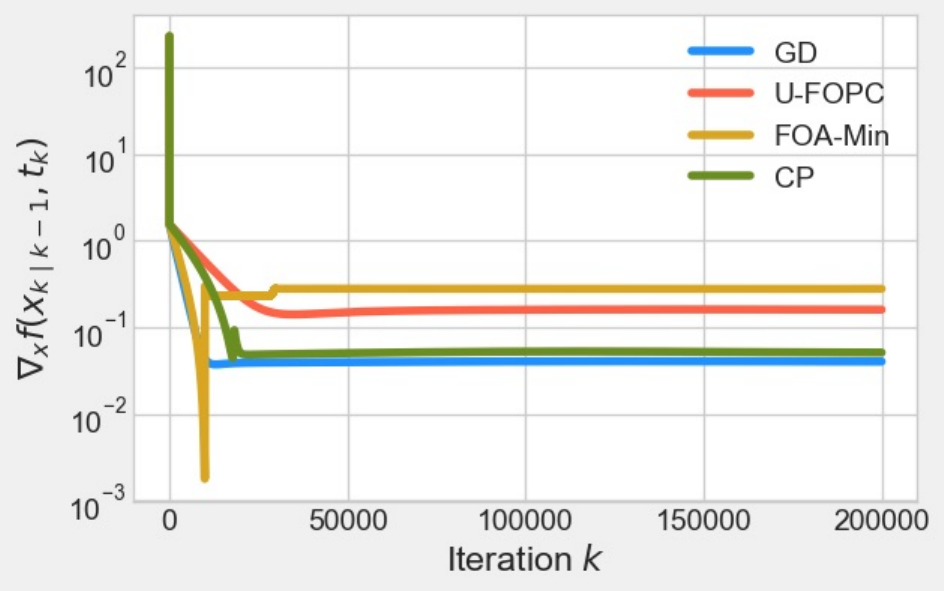}
    \end{minipage}%
\end{tabular}
\centering
    \caption{Log plots of function value and gradient norm when $h = 1e^{-3}$.}
    \label{fig:lr_all}
\end{figure}

\begin{figure}[H]
    \centering
\begin{tabular}{cc}
    \begin{minipage}[t]{0.45\linewidth}
    \centering
        \includegraphics[width=2.2in]{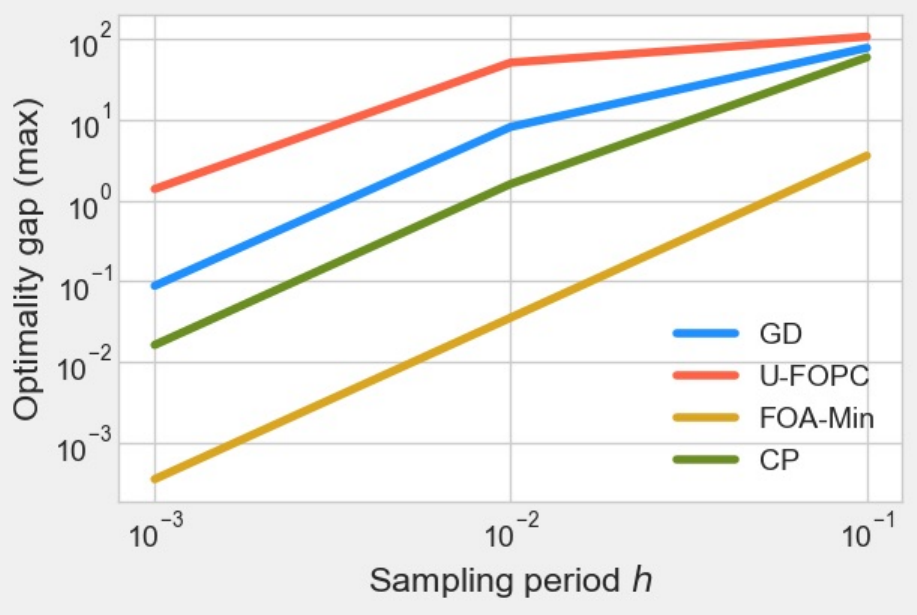}
    \end{minipage}

    \begin{minipage}[t]{0.45\linewidth}
    \centering
        \includegraphics[width=2.2in]{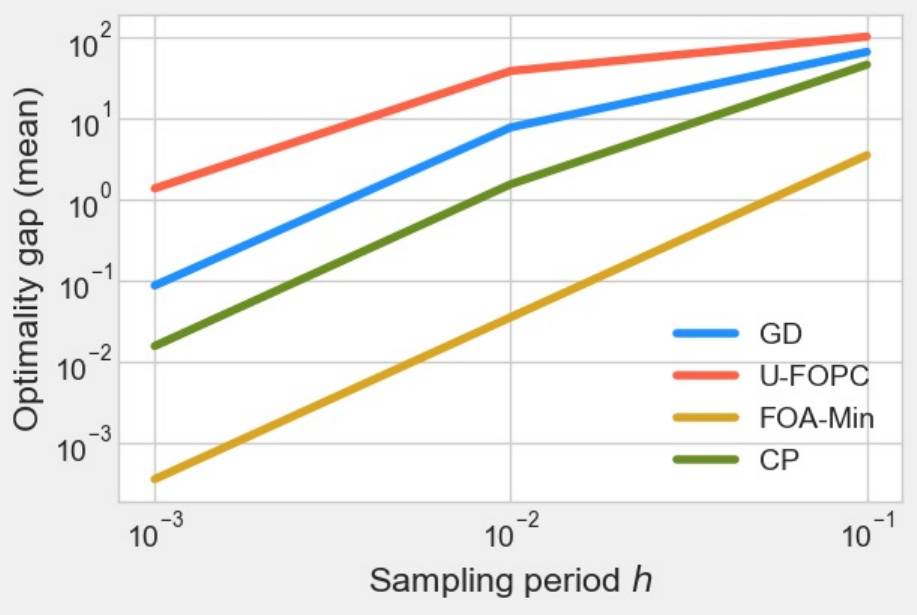}
    \end{minipage}%
\end{tabular}

\begin{tabular}{cc}
    \begin{minipage}[t]{0.45\linewidth}
    \centering
        \includegraphics[width=2.2in]{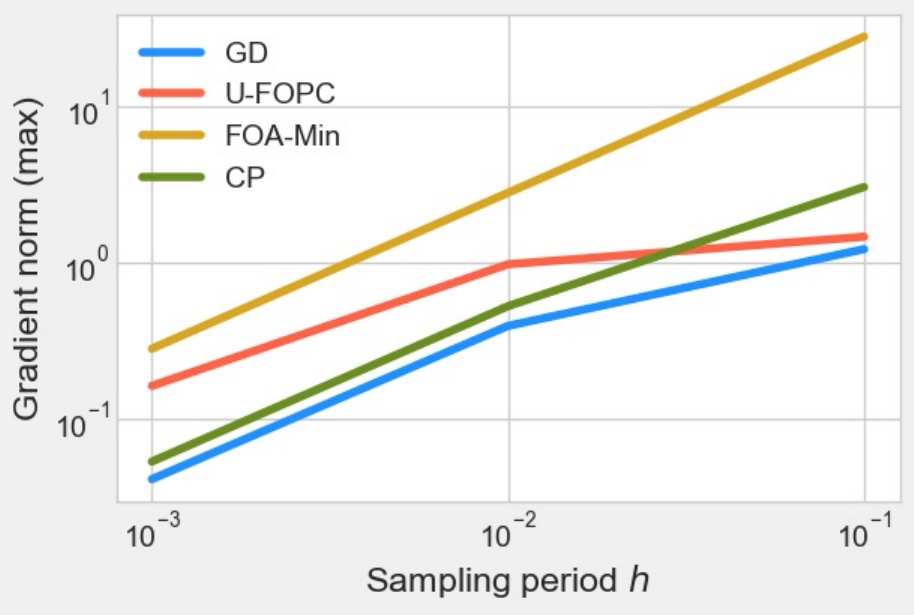}
    \end{minipage}

    \begin{minipage}[t]{0.45\linewidth}
    \centering
        \includegraphics[width=2.2in]{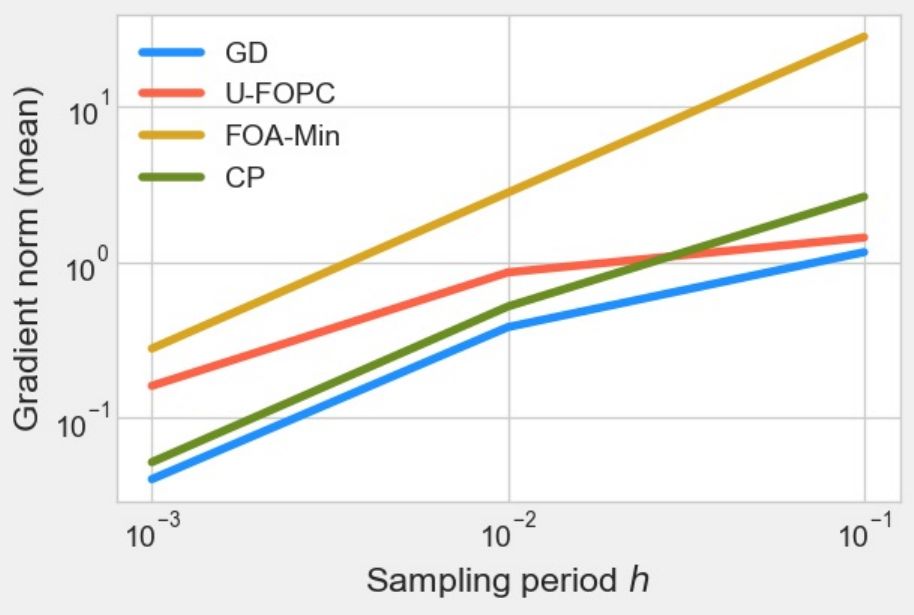}
    \end{minipage}%
\end{tabular}

\centering
    \caption{(Top) Log-log plots of maximum and mean of the gradient norm $ \| \nabla_x f(x_{k \mid k-1};t_k) \|$ versus sampling period. (Bottom) Log-log plots of maximum and mean of the optimality gap $f(x_{k \mid k-1};t_k) - f_k^\ast$ versus sampling period.
    Maximum and mean are computed based on the results of the last half of the iterations.}
    \label{fig:lr_different_h}
\end{figure}

\subsection{Additional Results for Matrix Factorization} \label{sec:mf_appendix}

Here, we present additional experimental results for the matrix factorization problem. We can see that the observation described in Section~\ref{sec:mf_main} is robust against the change of the dataset and the number of data revealed per iteration.

\paragraph{Plots of Optimization Results (Dataset 1)}
$\ $

\begin{figure}[H]
    \centering
\begin{tabular}{cc}
    \subfigure[Function value, $\norm*{ \nabla_x f(x_{0 \mid -1};t_0)} = 0.1$]{
    \begin{minipage}[t]{0.45\linewidth}
    \centering
        \includegraphics[width=2.2in]{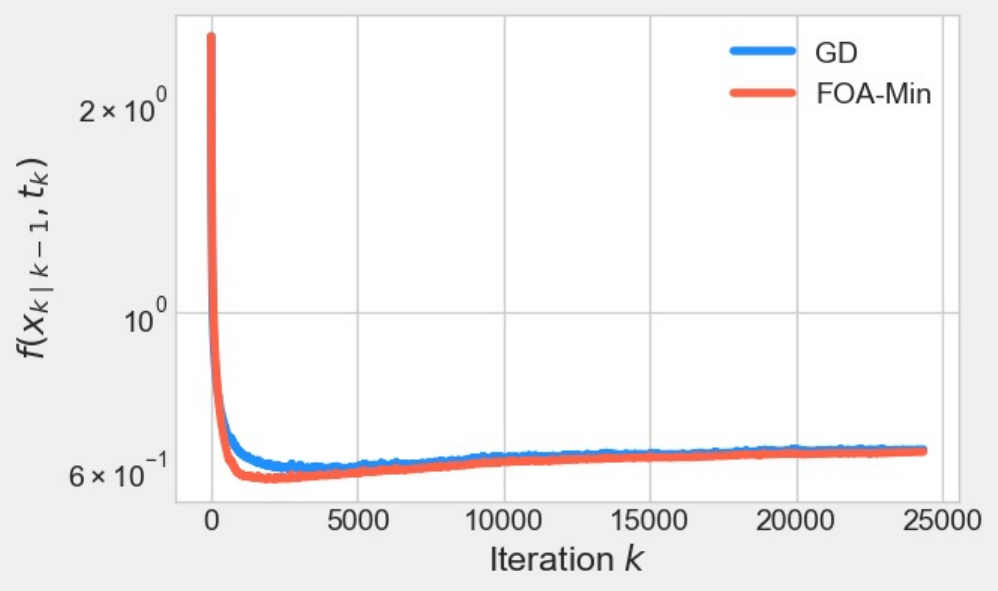}
    \end{minipage}
    }

    \subfigure[Gradient norm, $\norm*{ \nabla_x f(x_{0 \mid -1};t_0)} = 0.1$]{
    \begin{minipage}[t]{0.45\linewidth}
    \centering
        \includegraphics[width=2.2in]{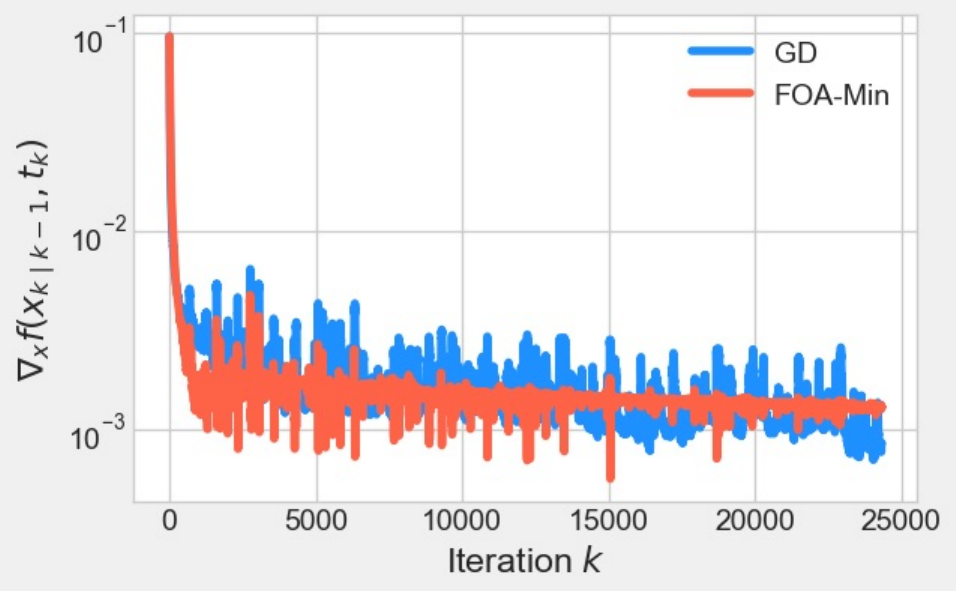}
    \end{minipage}%
    }
\end{tabular}

\begin{tabular}{cc}
    \subfigure[Function value, $\norm*{ \nabla_x f(x_{0 \mid -1};t_0)} = 1e^{-4}$]{
    \begin{minipage}[t]{0.45\linewidth}
    \centering
        \includegraphics[width=2.2in]{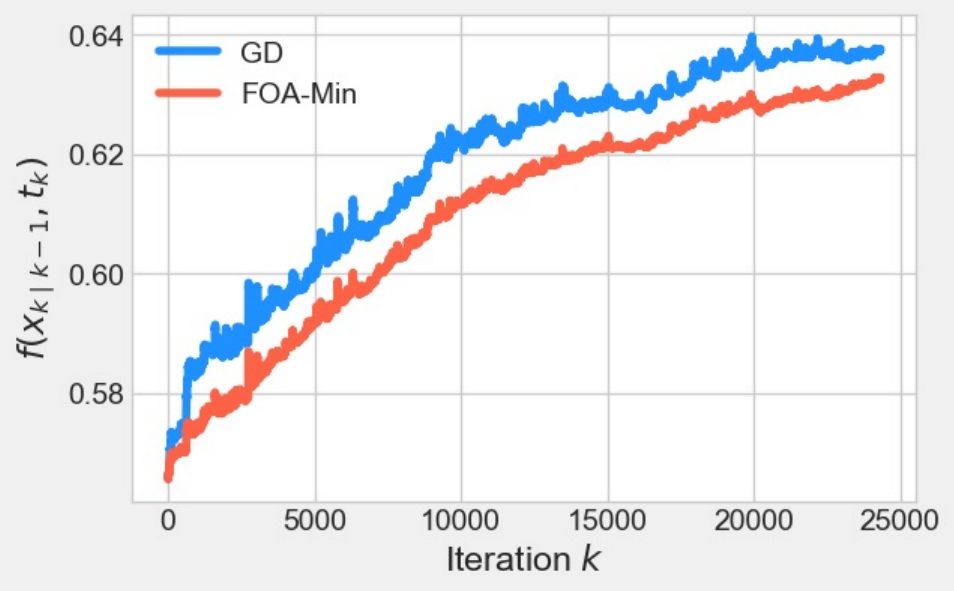}
    \end{minipage}
    }

    \subfigure[Gradient norm, $\norm*{ \nabla_x f(x_{0 \mid -1};t_0)} = 1e^{-4}$]{
    \begin{minipage}[t]{0.45\linewidth}
    \centering
        \includegraphics[width=2.2in]{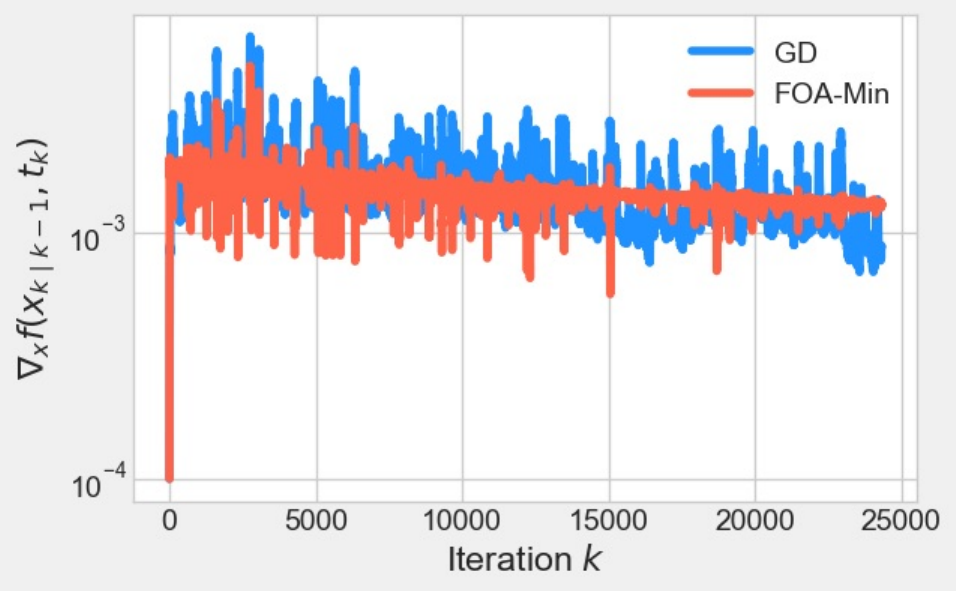}
    \end{minipage}%
    }
\end{tabular}

\centering
    \caption{Plots of function value and gradient norm for Dataset 1, and number of data revealed per iteration $N$ set to 5.}
    \label{fig:netflix_head_5}
\end{figure}

\begin{figure}[H]
    \centering
\begin{tabular}{cc}
    \subfigure[Function value, $\norm*{ \nabla_x f(x_{0 \mid -1};t_0)} = 0.1$]{
    \begin{minipage}[t]{0.45\linewidth}
    \centering
        \includegraphics[width=2.2in]{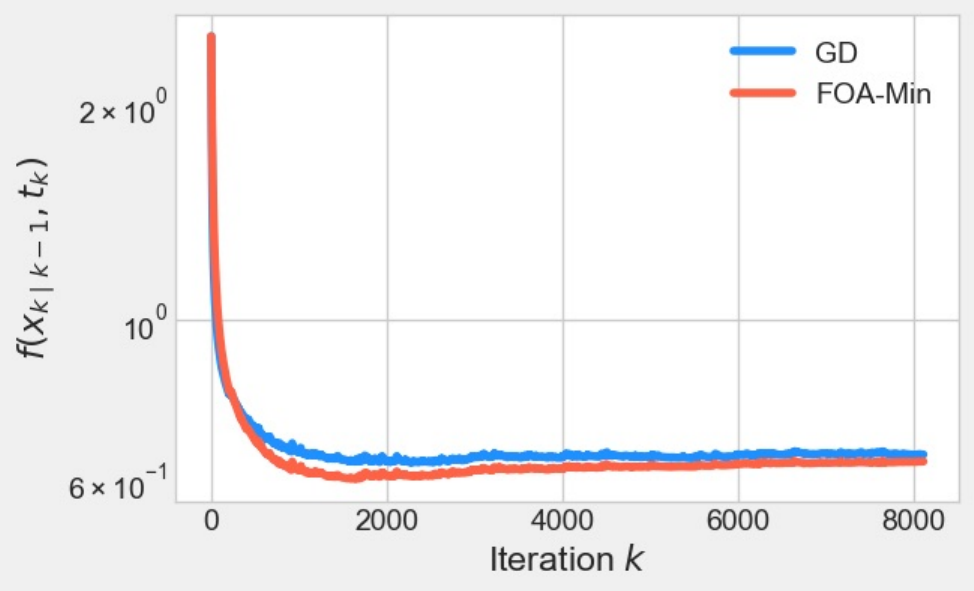}
    \end{minipage}
    }

    \subfigure[Gradient norm, $\norm*{ \nabla_x f(x_{0 \mid -1};t_0)} = 0.1$]{
    \begin{minipage}[t]{0.45\linewidth}
    \centering
        \includegraphics[width=2.2in]{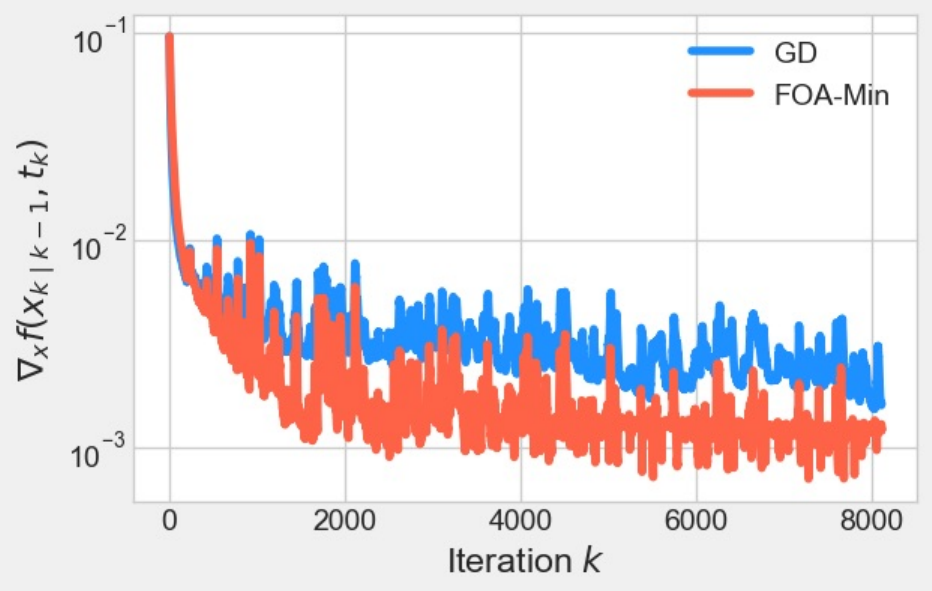}
    \end{minipage}%
    }
\end{tabular}

\begin{tabular}{cc}
    \subfigure[Function value, $\norm*{ \nabla_x f(x_{0 \mid -1};t_0)} = 1e^{-4}$]{
    \begin{minipage}[t]{0.45\linewidth}
    \centering
        \includegraphics[width=2.2in]{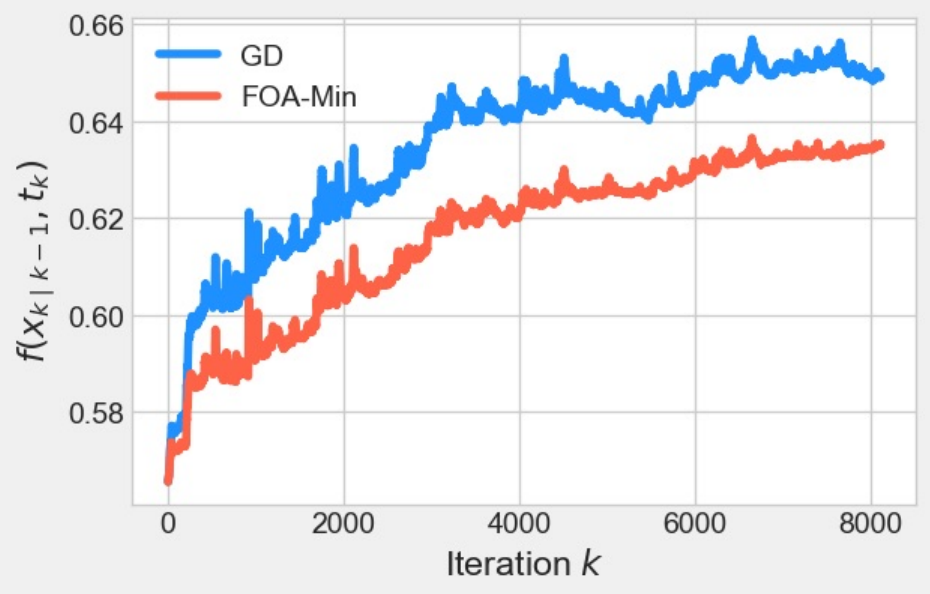}
    \end{minipage}
    }

    \subfigure[Gradient norm, $\norm*{ \nabla_x f(x_{0 \mid -1};t_0)} = 1e^{-4}$]{
    \begin{minipage}[t]{0.45\linewidth}
    \centering
        \includegraphics[width=2.2in]{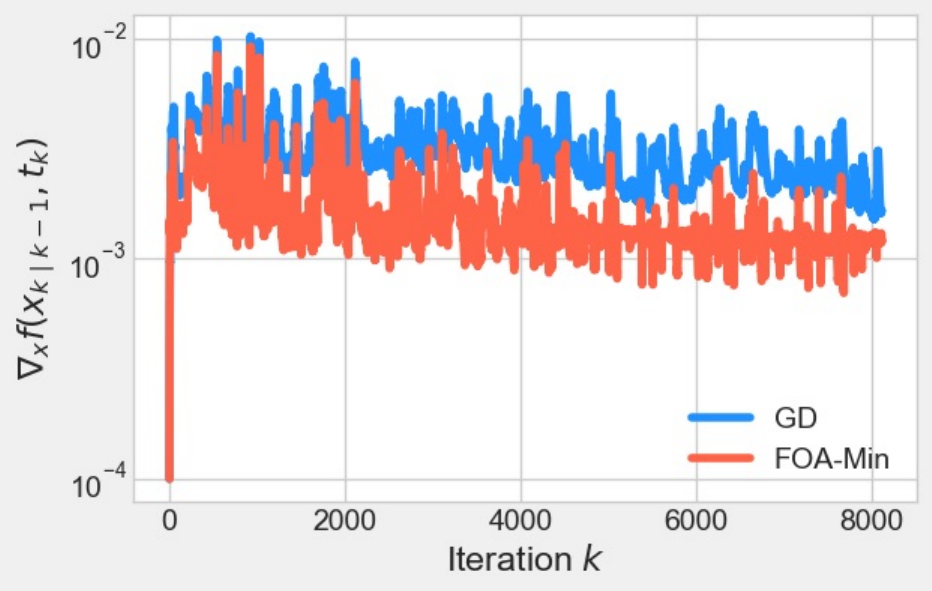}
    \end{minipage}%
    }
\end{tabular}

\centering
    \caption{Plots of function value and gradient norm for Dataset 1, and number of data revealed per iteration $N$ set to 15.}
    \label{fig:netflix_head_15}
\end{figure}

\paragraph{Plots of Optimization Results (Dataset 2)}
$\ $

\begin{figure}[H]
    \centering
\begin{tabular}{cc}
    \subfigure[Function value, $\norm*{ \nabla_x f(x_{0 \mid -1};t_0)} = 0.1$]{
    \begin{minipage}[t]{0.45\linewidth}
    \centering
        \includegraphics[width=2.2in]{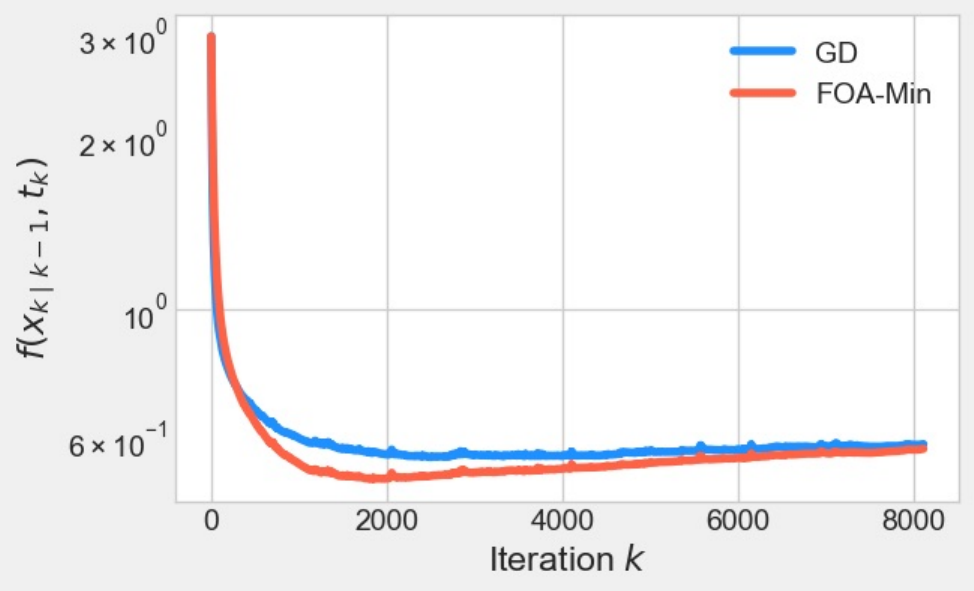}
    \end{minipage}
    }

    \subfigure[Gradient norm, $\norm*{ \nabla_x f(x_{0 \mid -1};t_0)} = 0.1$]{
    \begin{minipage}[t]{0.45\linewidth}
    \centering
        \includegraphics[width=2.2in]{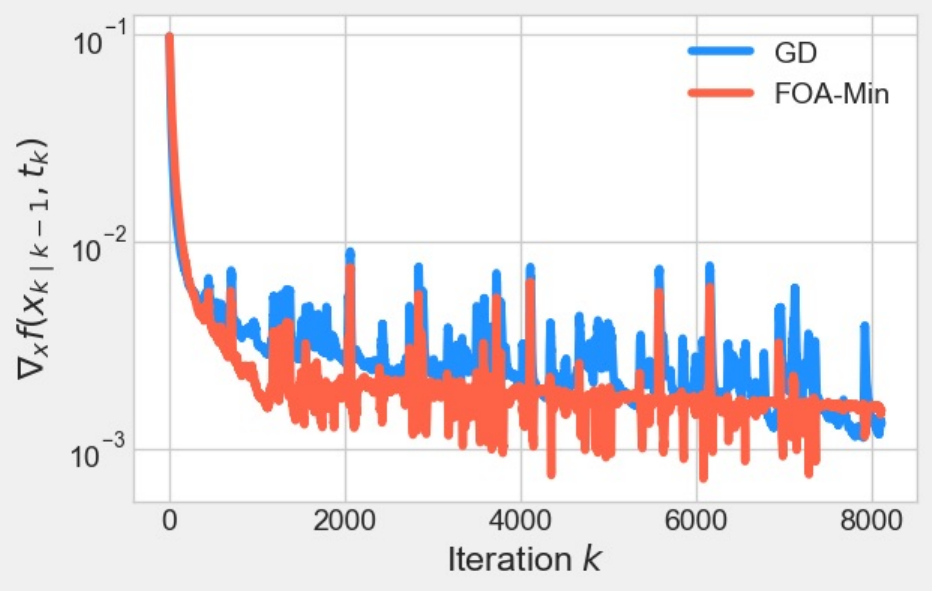}
    \end{minipage}%
    }
\end{tabular}

\begin{tabular}{cc}
    \subfigure[Function value, $\norm*{ \nabla_x f(x_{0 \mid -1};t_0)} = 1e^{-4}$]{
    \begin{minipage}[t]{0.45\linewidth}
    \centering
        \includegraphics[width=2.2in]{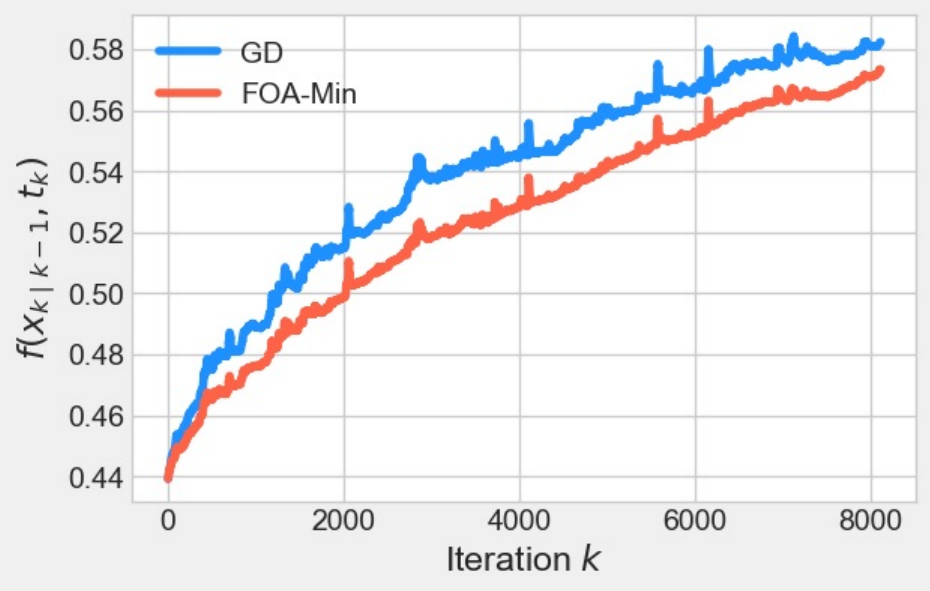}
    \end{minipage}
    }

    \subfigure[Gradient norm, $\norm*{ \nabla_x f(x_{0 \mid -1};t_0)} = 1e^{-4}$]{
    \begin{minipage}[t]{0.45\linewidth}
    \centering
        \includegraphics[width=2.2in]{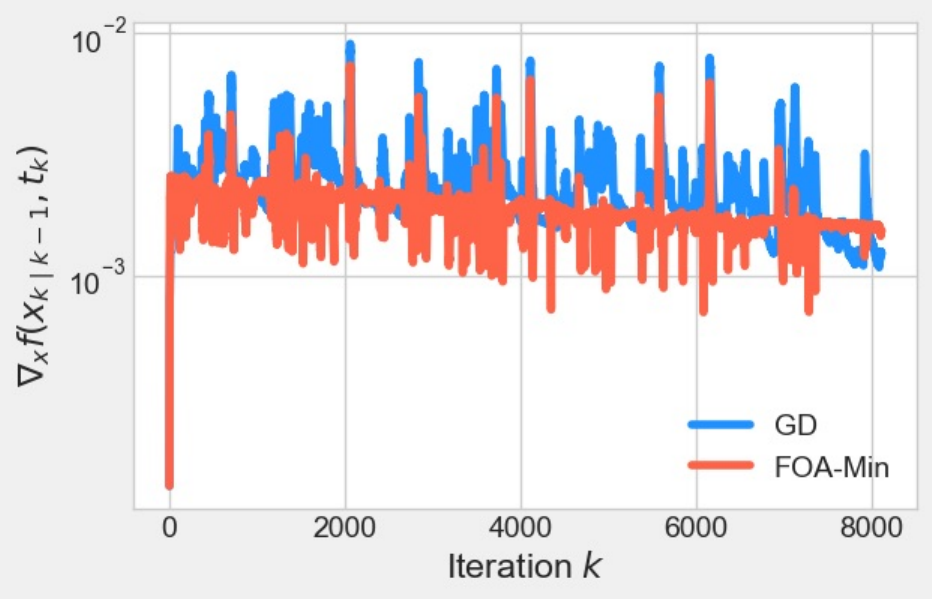}
    \end{minipage}%
    }
\end{tabular}

\centering
    \caption{Plots of function value and gradient norm for Dataset 2, and number of data revealed per iteration $N$ set to 15.}
    \label{fig:netflix_tail_15}
\end{figure}

\begin{figure}[H]
    \centering
\begin{tabular}{cc}
    \subfigure[Function value, $\norm*{ \nabla_x f(x_{0 \mid -1};t_0)} = 0.1$]{
    \begin{minipage}[t]{0.45\linewidth}
    \centering
        \includegraphics[width=2.2in]{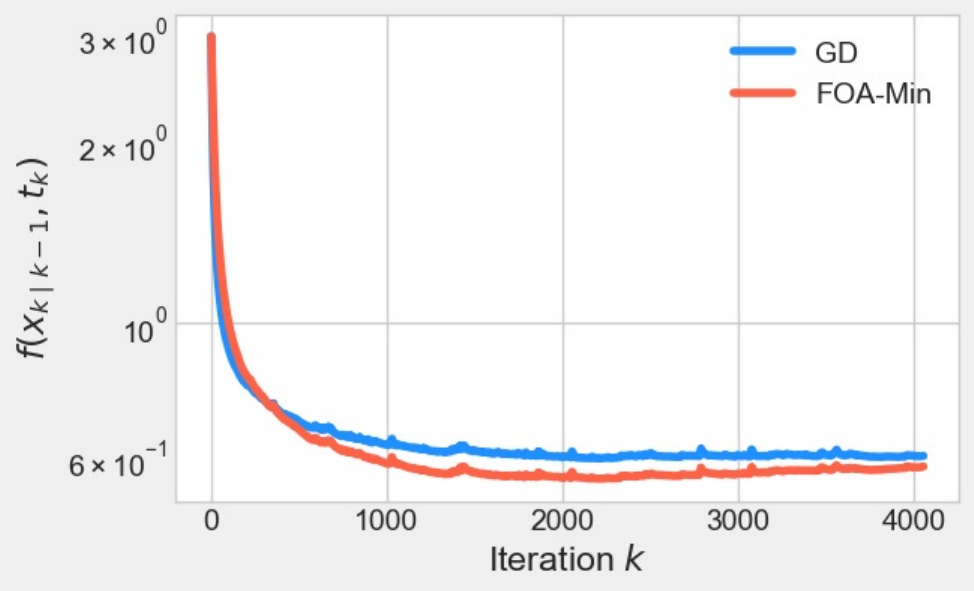}
    \end{minipage}
    }
    \subfigure[Gradient norm, $\norm*{ \nabla_x f(x_{0 \mid -1};t_0)} = 0.1$]{
    \begin{minipage}[t]{0.45\linewidth}
    \centering
        \includegraphics[width=2.2in]{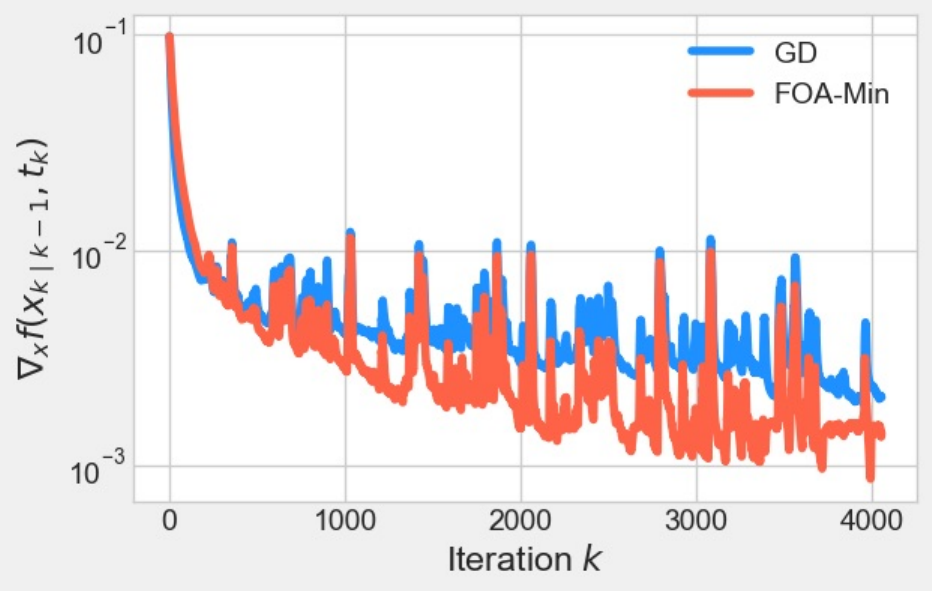}
    \end{minipage}%
    }
\end{tabular}

\begin{tabular}{cc}
    \subfigure[Function value, $\norm*{ \nabla_x f(x_{0 \mid -1};t_0)} = 1e^{-4}$]{
    \begin{minipage}[t]{0.45\linewidth}
    \centering
        \includegraphics[width=2.2in]{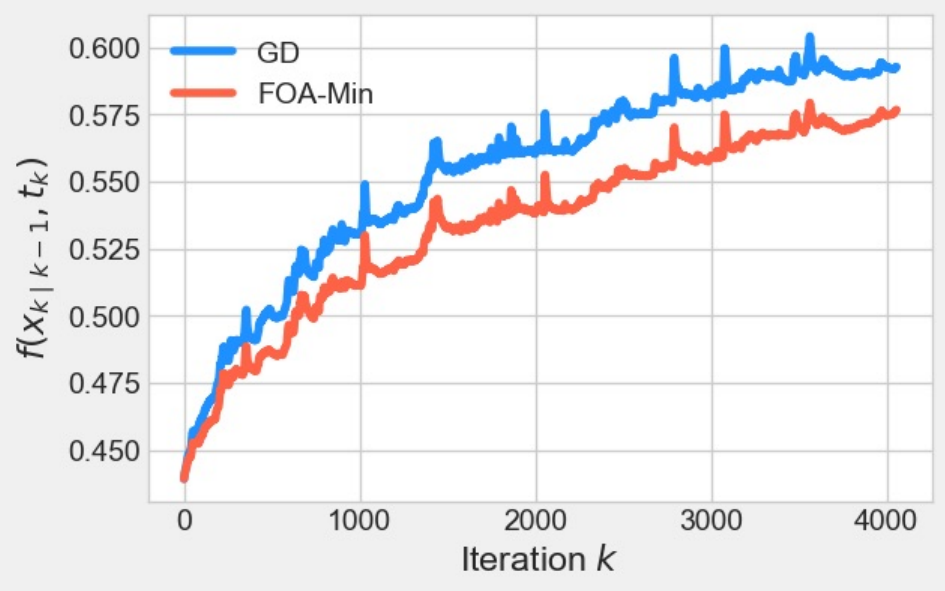}
    \end{minipage}
    }

    \subfigure[Gradient norm, $\norm*{ \nabla_x f(x_{0 \mid -1};t_0)} = 1e^{-4}$]{
    \begin{minipage}[t]{0.45\linewidth}
    \centering
        \includegraphics[width=2.2in]{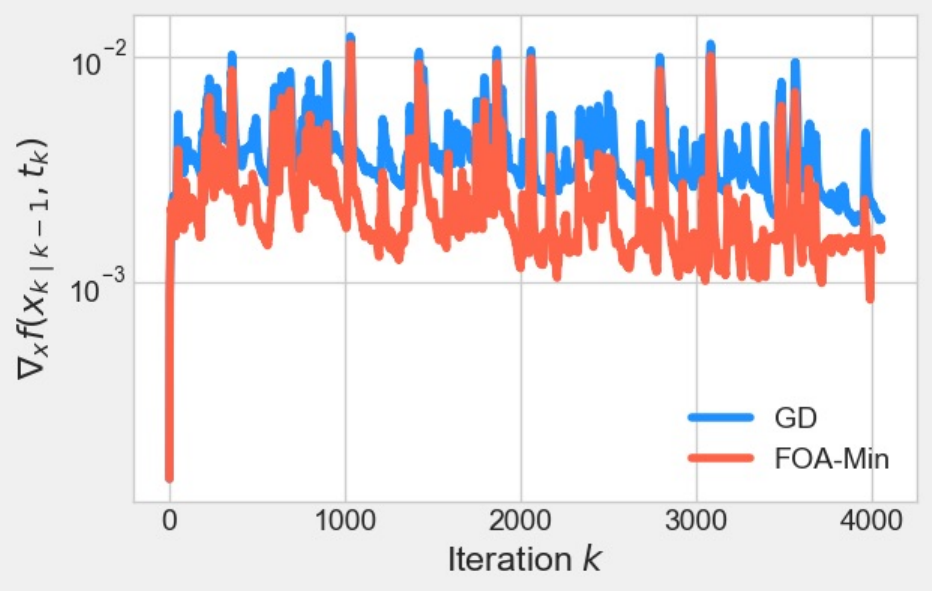}
    \end{minipage}%
    }
\end{tabular}

\centering
    \caption{Plots of function value and gradient norm for Dataset 2, and number of data revealed per iteration $N$ set to 30.}
    \label{fig:netflix_tail_30}
\end{figure}

\begin{figure}[H]
    \centering
\begin{tabular}{cc}
    \subfigure[Function value, $\norm*{ \nabla_x f(x_{0 \mid -1};t_0)} = 0.1$]{
    \begin{minipage}[t]{0.45\linewidth}
    \centering
        \includegraphics[width=2.2in]{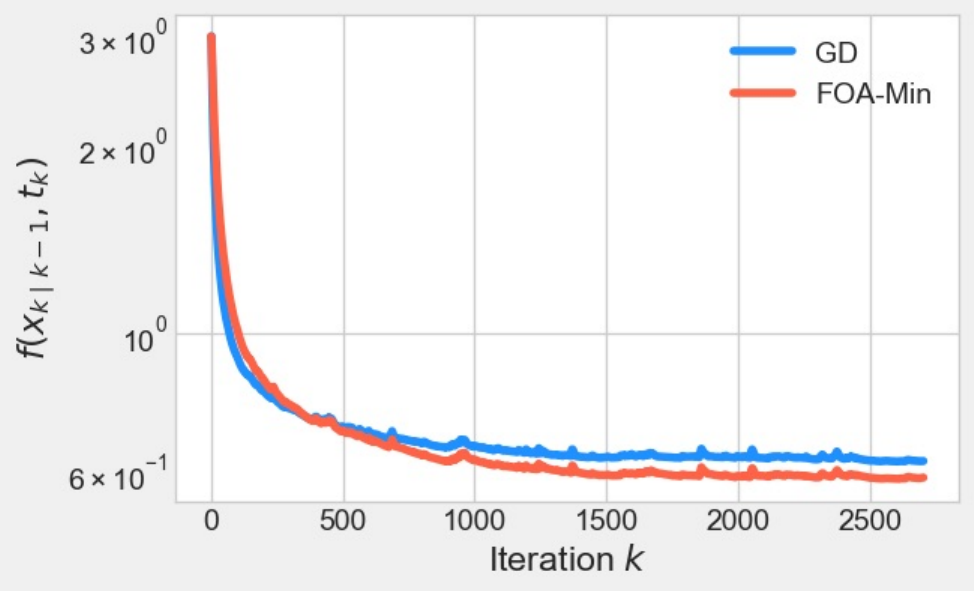}
    \end{minipage}
    }
    \subfigure[Gradient norm, $\norm*{ \nabla_x f(x_{0 \mid -1};t_0)} = 0.1$]{
    \begin{minipage}[t]{0.45\linewidth}
    \centering
        \includegraphics[width=2.2in]{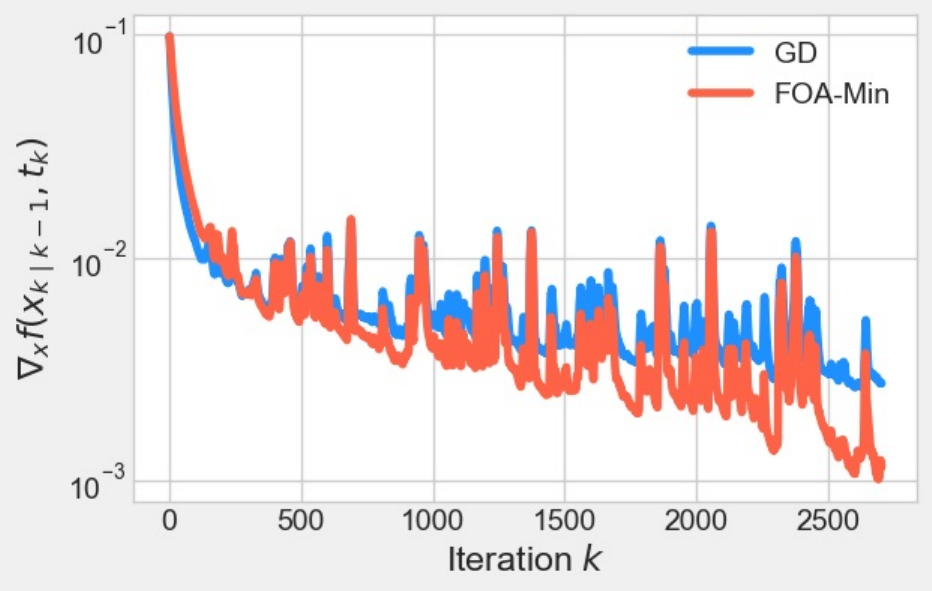}
    \end{minipage}%
    }
\end{tabular}

\begin{tabular}{cc}
    \subfigure[Function value, $\norm*{ \nabla_x f(x_{0 \mid -1};t_0)} = 1e^{-4}$]{
    \begin{minipage}[t]{0.45\linewidth}
    \centering
        \includegraphics[width=2.2in]{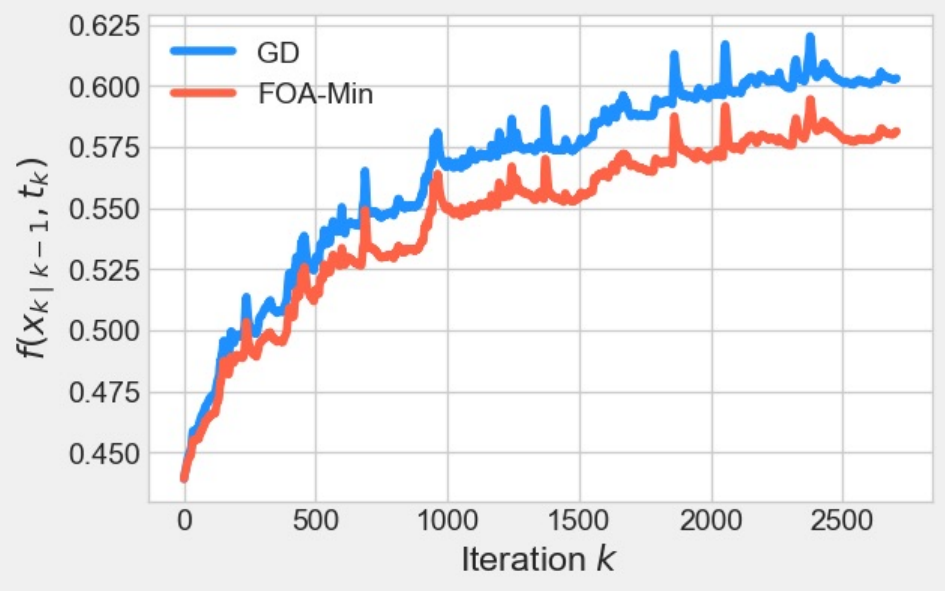}
    \end{minipage}
    }

    \subfigure[Gradient norm, $\norm*{ \nabla_x f(x_{0 \mid -1};t_0)} = 1e^{-4}$]{
    \begin{minipage}[t]{0.45\linewidth}
    \centering
        \includegraphics[width=2.2in]{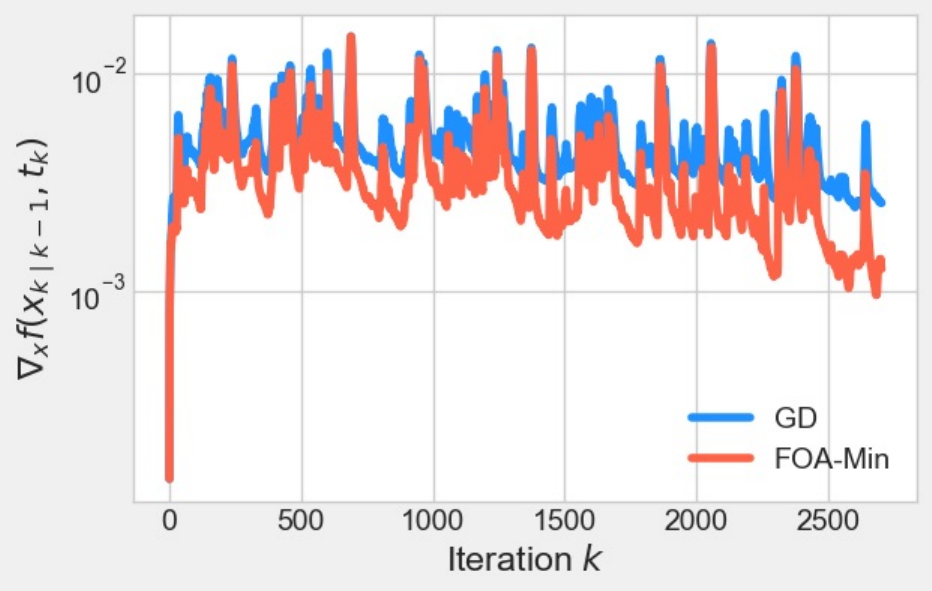}
    \end{minipage}%
    }
\end{tabular}

\centering
    \caption{Plots of function value and gradient norm for Dataset 2, and number of data revealed per iteration $N$ set to 45.}
    \label{fig:netflix_tail_45}
\end{figure}



\bibliographystyle{abbrvnat}
\bibliography{main} 


\end{document}